\documentclass{amsart}

\usepackage{amsmath}
\usepackage{amssymb}
\usepackage{algpseudocode}
\usepackage[hmargin=30mm,vmargin=30mm]{geometry}
\usepackage{graphicx}
\usepackage{hyperref}
\usepackage{xcolor}
\usepackage[foot]{amsaddr}

\newtheorem{theorem}{Theorem}[section]
\newtheorem{lemma}[theorem]{Lemma}
\newtheorem{prop}[theorem]{Proposition}
\newtheorem{corollary}[theorem]{Corollary}
\newtheorem{property}[theorem]{Property}

\theoremstyle{definition}
\newtheorem{definition}[theorem]{Definition}
\newtheorem{example}[theorem]{Example}

\newtheorem{remark}[theorem]{Remark}

\newcommand{\R}{\mathbb{R}}
\newcommand{\C}{\mathbb{C}}
\newcommand{\Sp}{\mathbb{S}}

\title{On length measures of planar closed curves and the comparison of convex shapes}

\author{Nicolas Charon$^1$}
\address{$^1$ Department of Applied Mathematics and Statistics, Johns Hopkins University.}
\email{charon@cis.jhu.edu}

\author{Thomas Pierron$^2$}
\address{$^2$ Department of Mathematics, ENS Paris-Saclay.} 
\email{thomas.pierron@ens-paris-saclay.fr}

\begin{document}

\begin{abstract}
In this paper, we revisit the notion of length measures associated to planar closed curves. These are a special case of area measures of hypersurfaces which were introduced early on in the field of convex geometry. The length measure of a curve is a measure on the circle $\Sp^1$ that intuitively represents the length of the portion of curve which tangent vector points in a certain direction. While a planar closed curve is not characterized by its length measure, the fundamental Minkowski-Fenchel-Jessen theorem states that length measures fully characterize convex curves modulo translations, making it a particularly useful tool in the study of geometric properties of convex objects. The present work, that was initially motivated by problems in shape analysis, introduces length measures for the general class of Lipschitz immersed and oriented planar closed curves, and derives some of the basic properties of the length measure map on this class of curves. We then focus specifically on the case of convex shapes and present several new results. First, we prove an isoperimetric characterization of the unique convex curve associated to some length measure given by the Minkowski-Fenchel-Jessen theorem, namely that it maximizes the signed area among all the curves sharing the same length measure. Second, we address the problem of constructing a distance with associated geodesic paths between convex planar curves. For that purpose, we introduce and study a new distance on the space of length measures that corresponds to a constrained variant of the Wasserstein metric of optimal transport, from which we can induce a distance between convex curves. We also propose a primal-dual algorithm to numerically compute those distances and geodesics, and show a few simple simulations to illustrate the approach. 
\end{abstract}

\maketitle

\section{Introduction}
There is a long history of interactions between geometric analysis and measure theory that goes back to the early 20th century alongside the development of convex geometry \cite{alexandroff1937,fenchel1938} and later on, in the 1960s, with the emergence of a whole new field known as geometric measure theory \cite{Federer} from the works of Federer, Fleming and their students. Since then, ideas from geometric measure theory have found their way into applied mathematics most notably in areas such as computational geometry \cite{Cohen-Steiner2003,abdallah2014reconstruction,buet2018discretization,Hu2019} and shape analysis \cite{hebert1995spherical,Glaunes2004,charon2013varifold,Roussillon2016} with applications in physics, image analysis and reconstruction or computer vision among many others.

A key element explaining the success of measure representations in geometric analysis and processing is their ability to encompass geometric structures of various regularities being either smooth immersed or embedded manifolds or discrete objects such as polytopes, polyhedra or simplicial complexes. They usually provide a comprehensive and efficient framework to capture the most essential geometric features in those objects. One of the earliest and most illuminating example is the use of the concept of \textit{area measures} \cite{alexandroff1937} in convex geometry which has proved instrumental to the theory of mixed volumes and the derivation of a multitude of geometric inequalities, in particular of isoperimetric or Brunn-Minkowski types \cite{gardner2002brunn,schneider_2013}. Area measures are typically defined for convex domains in $\R^n$ through the classical notion of support function of a convex set and can then be interpreted as a measure on the sphere $\Sp^{n-1}$ that represents the distribution of area of the domain's boundary along its different normal directions. A fundamental property of the area measure, that resulted from a series of works by Minkowski, Alexandrov, Fenchel and Jessen, is that there is a one-to-one correspondence between convex sets (modulo translations) and area measures: in particular the area measure characterizes a convex set up to translation and therefore fully encodes its geometry. 

In this paper, we are interested in area measures in the special and simplest case of planar curves. These are more commonly referred to as \textit{length measures} and are finite measures on the unit circle $\Sp^1$ that represents the distribution of length of the curve along its different tangent directions in the plane. Yet, unlike most previous works such as \cite{Letac1983} which are mainly focused on convex objects, one of our purpose is to first introduce and investigate length measures of general rectifiable closed curves in the plane. Our objective is also to provide as elementary and self-contained of an introduction as possible to these notions from the point of view and framework of shape analysis without requiring preliminary concepts from convex geometry. Although, in sharp contrast to the convex case, there are infinitely many rectifiable curves that share the same length measure, we will emphasize some simple geometric properties of the underlying curve that can be recovered or connected to those of its length measure. Furthermore, this approach will allow us to formulate and prove a new characterization of the unique convex curve associated to a fixed measure on $\Sp^1$ given by the Minkowski-Fenchel-Jessen theorem in the form of an isoperimetric inequality. Specifically, we show that this convex curve is the unique maximizer of the signed area among all the rectifiable curves with the same length measure. This extends some previously known results about polygons \cite{Boroczky1986} from the field of discrete geometry.

One of the fundamental problem of shape analysis is the construction of relevant metrics on spaces of shapes such as curves, surfaces, images... which is often the first and crucial step to subsequently extend statistical methods on those highly nonlinear and infinite dimensional spaces. Considering the shape space of convex curves, the one-to-one representation provided by the length measure can be particularly advantageous in that regard. Indeed, various types of generalized measure representations of shapes such as currents \cite{Glaunes2006,durrleman2009statistical}, varifolds \cite{charon2013varifold,Charon2017} or normal cycles \cite{Roussillon2016} have been used in the recent past to define notions of distances between geometric shapes from corresponding metrics on those measure spaces. A common downside of all these frameworks however is that the mapping that associates a shape to its measure is only injective but not surjective. This typically prevents those metrics to be directly associated to geodesics on the shape space without adding more constraints and/or exterior regularization, for example through deformation models. In contrast, the bijectivity of the length measure map from the space of convex curves modulo translations to the space of all measures on $\Sp^1$ satisfying only a simple linear closure constraint hints at the possibility of constructing a geodesic distance based on the length measure. In this paper, we propose an approach that relies on a constrained variant of the Wasserstein metric of optimal transport for probability measures on $\Sp^1$ which, as we show, turns the set of convex curves of length 1 modulo translations into a geodesic space. We also adapt and implement a primal-dual algorithm to numerically estimate the distance and geodesics. 

The paper is organized as follows. In Section \ref{sec:length_measures_Lipschitz}, we define the length measure of a Lipschitz regular closed oriented curve of the plane and provide a few general geometric and approximation properties. In Section \ref{sec:convex_curves}, we discuss specifically the case of convex curves and recap some of the fundamental connections between convex geometry and length measures, in particular the Minkowski-Fenchel-Jessen theorem. Section \ref{sec:isoperimetric_theorem} is dedicated to the statement and proof of an isoperimetric inequality for curves of prescribed length measure, from which we can also find again the classical isoperimetric inequality for planar curves. Section \ref{sec:metric_measure_convex} focuses on the construction of geodesic distances on the space of convex curves and in particular on our proposed constrained Wasserstein distance and its numerical computation. We also present and compare a few examples of estimated geodesics. Finally, Section \ref{sec:future_directions} concludes the paper by discussing some current limitations of this work and potential avenues for future improvements and extensions.

\section{Length measures of Lipschitz closed curves}
\label{sec:length_measures_Lipschitz}
\subsection{Definitions}
In all the paper, we will identify the 2D plane with the space of complex numbers $\C$. We start by introducing the space of closed, immersed oriented Lipschitz parametrized curves in the plane which we define as:
\begin{equation*}
    \mathcal{C} \doteq \{c\in \text{Lip}(\Sp^1,\C) \ | \ c'(\theta) \neq 0 \ for \ a.e. \ \theta \in \Sp^1\}.
\end{equation*}
Depending on the context, we will identify $\Sp^1$ either with the interval $[0,2\pi)\subset \R$ or the circle $\{e^{i\theta} \ | \ \theta \in \R \} \subset \C$. We recall that Lipschitz continuous curves are indeed differentiable almost everywhere and that the derivative is integrable on $\Sp^1$, which implies that the length $L(c) = \int_{\Sp^1} |c'(\theta)| d \theta$ is always finite. Note that we do not assume a priori that curves are simple. In anticipation to what follows, we also introduce the space of unparametrized oriented curves up to translations which we define as the quotient space 
\begin{equation*}
 \tilde{\mathcal{C}} \doteq \mathcal{C}/\sim
\end{equation*}
the equivalence relation being that for $c_1,c_2 \in \mathcal{C}$, $c_1 \sim c_2$ if and only if there exists $z_0 \in \C$ such that $\text{Im}(c_2)=\text{Im}(c_1) + z_0$ and the orientations of $\text{Im}(c_2)$ and $\text{Im}(c_1) + z_0$ coincide. This will constitute the actual space of \textit{shapes}, since objects in $\tilde{\mathcal{C}}$ are essentially planar curves modulo reparametrizations and translations or in other words oriented \textit{rectifiable} closed curves modulo translations in the plane. Furthermore, for any parametrization $c$ of a curve in $\mathcal{C}$ or $\tilde{\mathcal{C}}$, we let $T_c: \Sp^1 \rightarrow \Sp^1$ be the Gauss map, i.e. for all $\theta$, $T_c(\theta) = c'(\theta)/|c'(\theta)|$ which is well-defined for almost all $\theta \in \Sp^1$. Lastly, we will write $ds_c = |c'(\theta)| d \theta$ the arc-length measure of $c$, which is a positive measure on $\Sp^1$. In all this paper, we will denote by $\mathcal{M}^+(\Sp^1)$ the set of all positive finite Radon measures on $\Sp^1$ and $\mathcal{M}(\Sp^1)$ the set of signed finite Radon measures on $\Sp^1$. Also, we recall that the pushforward of a measure $\mu \in \mathcal{M}^{+}(\Sp^1)$ by a mapping $\tau: \Sp^1 \rightarrow \Sp^1$ is the measure denoted $\tau_*\mu$ such that $(\tau_*\mu)(B) = \mu(\tau^{-1}(B))$ for all Borel subset $B \subset \Sp^1$. Finally, we will denote by $\lambda^n$ the Lebesgue measure on $\R^n$. We can now introduce the notion of length measure which is the focus of this work.

\begin{definition}[Length measure of a curve]
\label{def:length_measure}
  For $c \in \mathcal{C}$, we define the length measure of $c$ that we write $\mu_{c}$, as the positive Radon measure on $\Sp^1$ obtained as the pushforward of $ds_c$ by the Gauss map $T_c$ i.e. $\mu_{c} \doteq (T_c)_{*}ds_c$. In other words, for all Borel set $B$ of $\Sp^1$:
  \begin{equation}
  \label{eq:def_length_meas}
      \mu_c(B) = ds_c(T_c^{-1}(B)) = \int_{T_c^{-1}(B)} |c'(\theta)| d \theta.
  \end{equation}
  For any $c_1,c_2$ such that $[c_1]=[c_2]$ in $\tilde{\mathcal{C}}$, one has $\mu_{c_1} = \mu_{c_2}$ leading to the well-defined mapping $M:c \mapsto \mu_c$ from $\tilde{\mathcal{C}}$ to the space of positive measures on $\Sp^1$. 
\end{definition}
The invariance of $\mu_c$ to the equivalence relation $\sim$ in Definition \ref{def:length_measure} is immediate from the fact that the arclength measure and Gauss map both only depend on the geometric image of the curve and are invariant to translations in the plane. In fact, the length measure of $[c] \in \tilde{\mathcal{C}}$ can be also interpreted as the pushforward of the Hausdorff measure $\mathcal{H}^1$ on the plane by the mapping $T: \text{Im}(c) \rightarrow \Sp^1$ such that for any $x \in \text{Im}(c)$, $T(x)$ is the direction of the tangent vector of the curve at $x$. Note however that $\mu_c$ does depend on the orientation of $c$: more specifically, if the orientation is reversed, the associated length measure of the resulting $\check{c}$ is the reflection of $\mu_{c}$ that is $\mu_{\check{c}}(B) = \mu(-B)$ for all Borel set $B$ of $\Sp^1$ (or equivalently if $\Sp^1$ is identified with $[0,2\pi)$, $\mu_{\check{c}}(B) = \mu((B+\pi)\ \text{mod} \ 2\pi)$). Thus $\mu_c$ is a geometric quantity associated to unparametrized oriented curves modulo translations of the plane. For simplicity, we will still write $\mu_c$ to denote the length measure associated to whole equivalence class $[c] \in \tilde{\mathcal{C}}$. 

\begin{remark}
The measure $\mu_c$ can be intuitively understood as the distribution of length along the different directions of tangents to the curve $c$, as we will further illustrate below. We point out that our definition of length measure departs slightly from the traditional concept of length measure (or perimetric measure as it is sometimes called) introduced initially for convex objects by Alexandrov, Fenchel and Jessen \cite{alexandroff1937,fenchel1938}, in that we consider here the direction of unit tangent vectors rather than unit normal vectors to the planar curve. Note that these two definitions only differ by a simple global rotation of the measure by an angle of $\pi/2$ and so have a straightforward relation to one another. The reason for choosing this alternative convention is that this work focuses on length measures of planar curves and not area measures for general hypersurfaces, and our presentation will become a little simpler in this setting.            
\end{remark}

Lastly, thanks to the classical Riesz-Markov-Kakutani representation theorem, we also recall that signed measures on $\Sp^1$ can be alternatively interpreted as elements of the dual to the space of continuous functions on $\Sp^1$. In the case of a length measure, the measure $\mu_c$ acts on any $f \in C(\Sp^1)$ as follows:
\begin{equation}
\label{eq:dual_rep_length_meas}
 (\mu_c | f ) = \int_{\Sp^1} f \left(\frac{c'(\theta)}{|c'(\theta)|} \right) |c'(\theta)| d \theta.
\end{equation}
By straightforward extension, we can also consider the action of $\mu_c$ on continuous complex-valued functions given by the same expression as in \eqref{eq:dual_rep_length_meas}. Before looking into the properties of length measures, we add a final remark on the definition itself.

\begin{remark}
 Length measures can be also connected to some other concept of geometric measure theory namely the 1-varifolds as introduced in \cite{allard1972first} and more precisely the oriented varifold representation of curves which is discussed and analyzed for instance in \cite{Charon2017}. Indeed, the oriented varifold $V_c$ associated to a curve $c \in \mathcal{C}$ is by definition the positive Radon measure on the product space $\mathbb{R}^2 \times \Sp^1$ given for all continuous compactly supported function $\omega$ on $\mathbb{R}^2 \times \Sp^1$ by:
 \begin{equation*}
  (V_c | \omega) \doteq \int_{\Sp^1} \omega\left(c(\theta), \frac{c'(\theta)}{|c'(\theta)|}\right) |c'(\theta)| d \theta.
 \end{equation*}
and $V_c$ is once again independent of the choice of $c$ in the equivalence class $[c]\in \tilde{\mathcal{C}}$. By comparison to \eqref{eq:dual_rep_length_meas}, the 1-varifold $V_c$ can be thought as a spatially localized version of the length measure $\mu_c$, that is a distribution of unit directions at different positions in the plane. Equivalently, the length measure is obtained by marginalizing $V_c$ with respect to its spatial component. This loss in spatial localization explains the lack of injectivity of the length measure representation (even after quotienting out translations) that is discussed below.     
\end{remark}

\subsection{Basic geometric properties}
Let us now examine more closely the most immediate properties of the length measure, in particular how it relates to various geometric quantities and transformations of the underlying curve. For a general smooth mapping $\phi: \C \rightarrow \C$, we will write $\phi \cdot c$ the curve $\theta \in \Sp^1 \mapsto \phi(c(\theta))$. 
\begin{prop}
\label{prop:basic_properties_mu}
 Let $c \in \mathcal{C}$ and $\mu_c$ its length measure. Then
 \begin{enumerate}
  \item For all $\theta_1,\theta_2 \in \Sp^1 = [0,2\pi)$:
\begin{equation*}
    \mu_c([\theta_1,\theta_2]) = \text{Length}\left(\left\{c(\theta) \ | \ \theta_1 \leq \text{angle}(c'(\theta))\leq \theta_2\right\}\right).
\end{equation*}
In particular, the total length of $c$ is $L(c) = \mu_c(\Sp^1)$.
  \item For any rotation $R_\theta \in SO(2)$ of angle $\theta \in \Sp^{1}$, $\mu_{R_\theta \cdot c} = (R_\theta)_* \mu_{c}$ namely for all $B\subset \Sp^1$, $\mu_{R_\theta \cdot c}(B) = \mu_c((B-\theta)\ \text{mod} \ 2\pi)$.
  \item For any $\lambda >0$, $\mu_{\lambda c} = \lambda \mu_{c}$, where $\lambda c$ denotes the rescaling of $c$ by a factor $\lambda$.
 \end{enumerate}
\end{prop}

Note that property (1) above implies in particular that the length measures of curves of length one are probability measures on $\Sp^1$. Combined with property (3) on the action of scalings, one could further define a mapping from $\tilde{\mathcal{C}}/\text{Scal}$, the space of curves modulo translation and scaling, into the space of probability measures on $\Sp^1$ by essentially renormalizing curves to have length one. In all cases, by identifying $\Sp^1$ with $[0,2\pi)$, a convenient way to represent and visualize $\mu_c$ is through its \textit{cumulative distribution function} (cdf): 
\begin{align*}
    F_{\mu_c}:[0,2\pi) &\longrightarrow [0,L_c) \\
              \theta &\longmapsto \mu_c([0,\theta])
\end{align*}
which is always a non-decreasing and right-continuous function.  

\begin{figure}
    \begin{tabular}{ccc}
    \includegraphics[width=5cm]{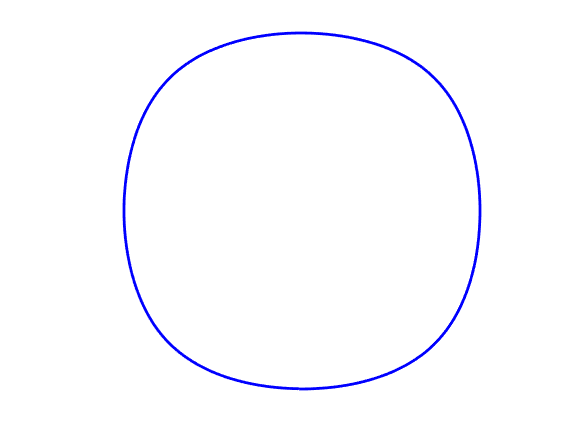} 
    &\includegraphics[width=5cm]{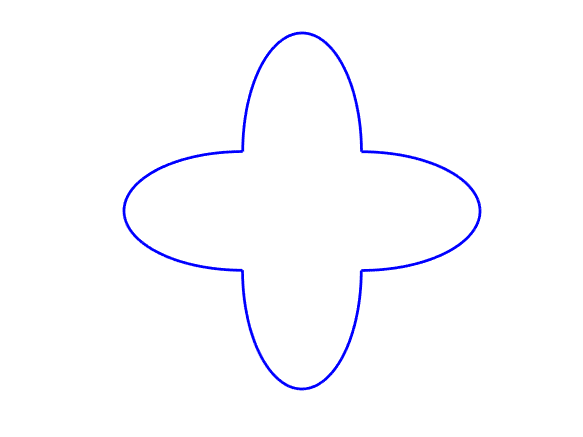} 
    &\includegraphics[width=5cm]{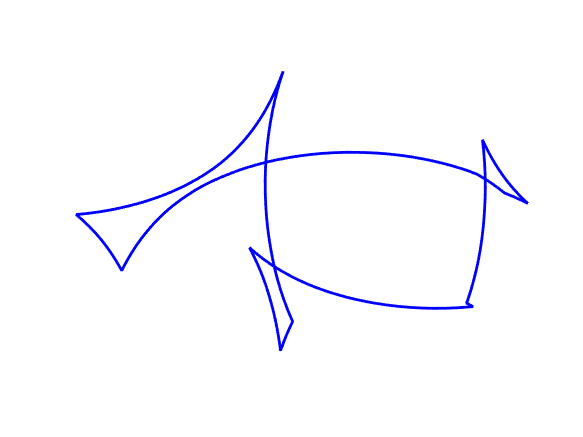} 
    \end{tabular}
 \includegraphics[width=5cm]{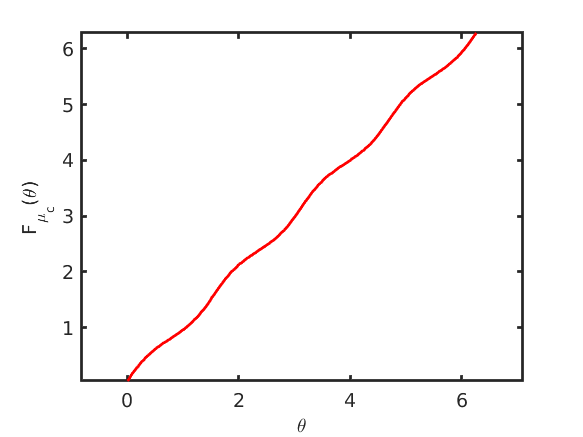} 
    \caption{Three planar curves with the same length measure $\mu_c$ which c.d.f is plotted on the second row.}
    \label{fig:smooth_curves}
\end{figure}

There are some further constraints that length measures must satisfy. Most notably, since the curve $c$ is closed, we obtain from \eqref{eq:dual_rep_length_meas} that:
\begin{align*}
   \left(\mu_c | e^{i\theta} \right) = \int_{\Sp^1} e^{i\theta} d\mu_c(\theta) = \int_{\Sp^1} e^{i\text{angle}(c'(\theta))} |c'(\theta)| d\theta = \int_{\Sp^1} c'(\theta) d\theta = 0. 
\end{align*}
In other words, all length measures are such that the expectation of $e^{i\theta}$ on $\Sp^1$ vanishes. Together with the above, this leads us to introduce the following subset of $\mathcal{M}^+(\Sp^1)$:
\begin{definition}
\label{def:length_meas_map}
 We denote by $\mathcal{M}_0^+(\Sp^1)$ (resp. $\mathcal{M}_0(\Sp^1)$) the space of all positive (resp. signed) Radon measures $\mu$ on $\Sp^1$ which are such that:
 \begin{equation*}
  \int_{\Sp^1} e^{i\theta} d\mu(\theta) = 0 \Leftrightarrow \ \int_{\Sp^1} \cos(\theta) d\mu(\theta) = \int_{\Sp^1} \sin(\theta) d\mu(\theta) =0.
 \end{equation*}
 We then define the mapping $M: \ \tilde{\mathcal{C}} \rightarrow \mathcal{M}^+_0$ by $M([c]) = \mu_c$ which does not depend on the choice of $c$ in the equivalence class $[c]$. 
\end{definition}
Again, we will often replace an element of the quotient $[c]\in \tilde{\mathcal{C}}$ by one of its representant $c \in \mathcal{C}$ and then write $M(c)$ instead of $M([c])$. 

\begin{remark}
\label{rem:symmetrization}
Note that any positive measure $\mu$ on $\Sp^1$ such that $\mu(-B)=\mu(B)$ for all Borel subsets of $\Sp^1$ belongs to $\mathcal{M}^+_0$. Indeed the assumption implies that:
\begin{equation*}
    \int_0^{2\pi} e^{i\theta} d\mu(\theta) = \int_0^{\pi} e^{i\theta} d\mu(\theta) + \int_{0+\pi}^{\pi+\pi} e^{i\theta} d\mu(\theta) = \int_0^{\pi} e^{i\theta} d\mu(\theta) - \int_0^{\pi} e^{i\theta} d\mu(\theta) =0
\end{equation*}
where the third equality follows from the change of variable $\theta \leftarrow \theta +\pi$ and the fact that $d\mu(\theta+\pi)=d\mu(\theta)$ by assumption. As a consequence, given any positive measure $\mu$ on $\Sp^1$, one can always symmetrize it by defining $\bar{\mu}(B) = \mu(B) + \mu(-B)$ and obtain a measure of $\mathcal{M}^+_0$. As a side note, convex curves for which the length measure satisfy $\mu_c(-B)=\mu_c(B)$ are called central-symmetric and are an important class of objects in convex geometry.  
\end{remark}

It is obvious that the mapping $M$ in Definition \ref{def:length_meas_map} cannot be injective on $\tilde{\mathcal{C}}$. In fact, given $c \in \tilde{\mathcal{C}}$, there are infinitely many other curves that share the same length measure as $c$. In Figure \ref{fig:smooth_curves}, we show several examples of curves having the same length measure which c.d.f is plotted underneath. On the other hand, $M$ is surjective as we shall see in the next section. For now, we will just illustrate the different types of length measures depending on he nature of the underlying curve with a few examples.

\begin{figure}
    \begin{tabular}{cc}
    \includegraphics[width=6cm]{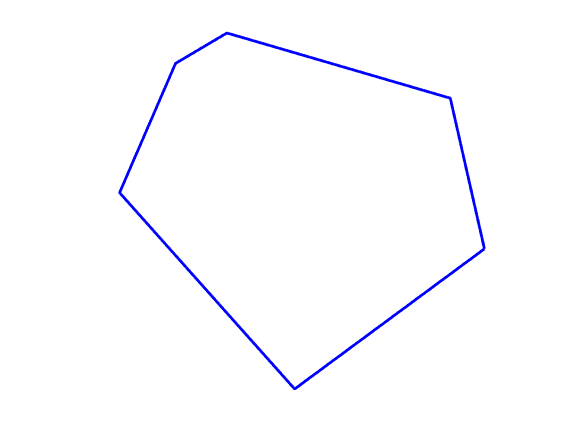} 
    &\includegraphics[width=6cm]{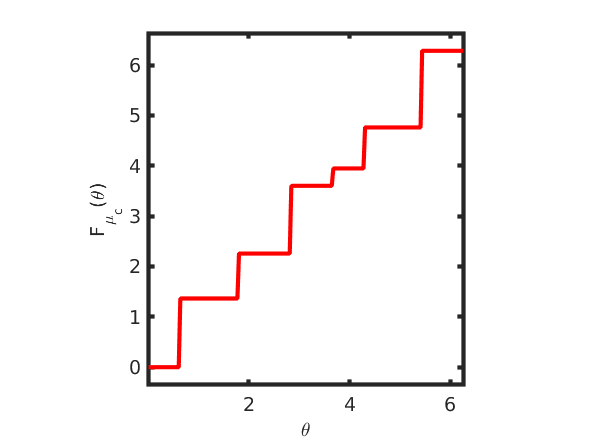} \\
    \includegraphics[width=6cm]{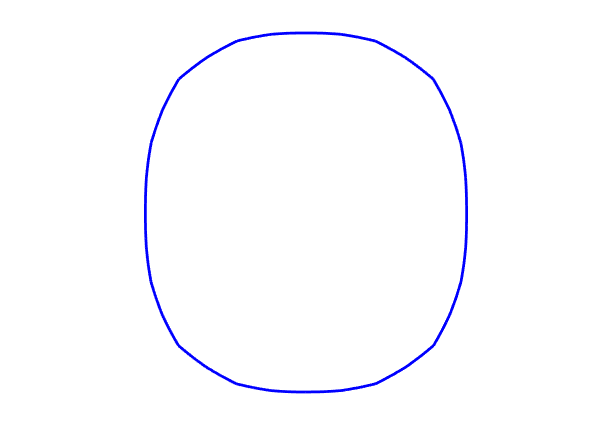} 
    &\includegraphics[width=6cm]{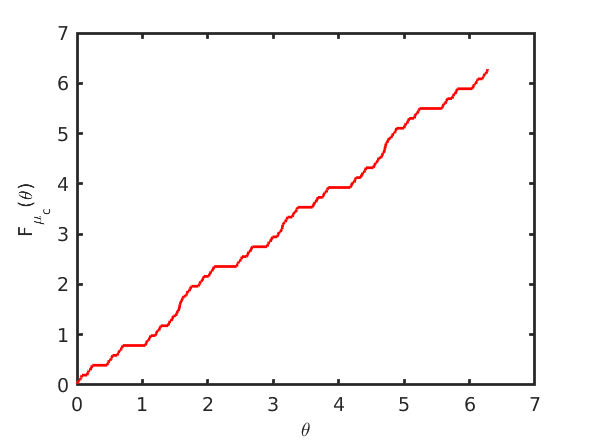}    
    \end{tabular}
    \caption{Examples of closed curves (left) with their corresponding c.d.f (right).}
    \label{fig:polygon_mu}
\end{figure}

\begin{example}
Assume that $c(\theta) = e^{i\theta}$ for $\theta \in [0,2\pi)$ is the unit circle in $\C$. Then $ ds_c(\theta) = d\theta$ and $T_c(\theta) = \theta$ thus $\mu_c = d\theta$ is the uniform measure on $\Sp^1$. 
\end{example}

\begin{example}
\label{ex:poylgon}
Consider a closed polygon with $n$ vertices $z_1,z_2,\ldots,z_n$ and the $n$ faces $[z_1,z_2],[z_2,z_3],$ $\ldots,[z_n,z_0]$. For all $j=1,\ldots,n$, let $\alpha_j= \text{angle}(z_{j+1}-z_j)$, $l_j= |z_{j+1}-z_j|$ (by convention $z_{n+1}=z_0$) and $L=\sum_{j=1}^{n}l_j$. Then, one can construct a parametrization $c \in \mathcal{C}$ of the polygon as follows. We define $\theta_j=2\pi (j-1)/n$ for $j=1,\ldots,n+1$. Then, on $[\theta_j,\theta_{j+1}]$, $j=1,\ldots,n$, take $c(\theta)\doteq (\theta-\theta_j)\frac{n-1}{2\pi}l_j e^{i\alpha_j}$. It follows that $c'(\theta) = \frac{n-1}{2\pi}l_j e^{i\alpha_j}$ for $\theta \in [\theta_j,\theta_{j+1}]$ and one can easily check that it leads to $\mu_c = \sum_{j=1}^{n} l_j \delta_{e^{i\alpha_j}}$. The length measure of a polygon is thus always a sum of Dirac masses, c.f Figure \ref{fig:polygon_mu} for an illustration.
\end{example}

\begin{example}
Lastly, we consider an example of a singular continuous length measure. To construct it, let's introduce the standard Cantor distribution that we rescale to the interval $[0,2\pi]$ and denote it by $\sigma$. Its support is the Cantor set and its cumulative distribution function $F_{\sigma}$ is the well-known devil's staircase function on $[0,2\pi]$. The measure $\sigma$ is not in $\mathcal{M}^+_0$ but following Remark \ref{rem:symmetrization}, we can consider instead its symmetrization $\bar{\sigma} \in \mathcal{M}_0$ which c.d.f is given by $F_{\bar{\sigma}}(\theta) = F_{\sigma}(\theta) + \sigma(([0,\theta] + \pi) \ \text{mod} \ 2\pi)$. Now, letting $F_{\bar{\sigma}}^{(-1)}$ being the pseudo-inverse of $F_{\bar{\sigma}}$, such that for all $\theta \in [0,2\pi]$ by $F_{\bar{\sigma}}^{(-1)}(\theta) = \inf\{\theta' \ | \  F_{\bar{\sigma}}(\theta')\geq \theta\}$, we define the curve $c:[0,2\pi)\rightarrow \C$:
\begin{equation*}
    c(\theta) = \int_0^{\theta} e^{i F_{\bar{\sigma}}^{(-1)}(\theta')} d\theta' 
\end{equation*}
which is $C^1$ and satisfies $|c'(\theta)|=1$ and $T_c(\theta) = e^{i F_{\bar{\sigma}}^{(-1)}(\theta)}$. As we will show more in details and in general in the proof of Theorem \ref{thm:tangent_length_map}, it follows that we have $F_{\mu_c} = F_{\bar{\sigma}}$ and therefore $\mu_c = \bar{\sigma}$ is the above symmetrized Cantor distribution $\bar{\sigma}$. We simulated such a curve using an approximation at level 12 of the Cantor measure, which is shown in Figure \ref{fig:polygon_mu}. 
\end{example}

Those examples show that length measures of curves in $\mathcal{C}$ may have density with respect to the Lebesgue measure on $\Sp^1$, be singular discrete measures but also singular continuous measures as the last example shows. Furthermore, we have the following connection between the length measure density of a sufficiently regular curve and its curvature:

\begin{prop}
\label{prop:length_measure_dens}
If $c \in \mathcal{C}$ is in addition twice differentiable with bounded and a.e non-vanishing curvature then $\mu_c = \rho(\theta) d \theta$ where the density $\rho(\theta)$ is given for a.e $\theta \in \Sp^1$ by:
\begin{equation}
\label{eq:density_curvature}
    \rho(\theta) = \sum_{u \in T_c^{-1}(\{\theta\})} \frac{1}{\kappa_c(u)} 
\end{equation}
where $\kappa_c$ is the curvature of $c$.
\end{prop}
\begin{proof}
First, we notice that $T_c'(\theta) = |c'(\theta)| \frac{d}{ds_c} T_c(\theta) = |c'(\theta)| \kappa(\theta)$. Now, let $f$ be a continuous test function $\Sp^1 \rightarrow \R$. By definition, we have:
\begin{align*}
  (\mu_c | f) = \int_{\Sp^1} f \circ T_c(\theta) |c'(\theta)| d\theta  = \int_{\Sp^1} f \circ T_c(\theta) \frac{1}{\kappa(\theta)} T_c'(\theta) d\theta.
\end{align*}
Since $T_c$ is a Lipschitz function, by the coarea formula (\cite{Ambrosio2000} Theorem 2.93), the above integral can be rewritten as:
\begin{equation*}
    (\mu_c | f) = \int_{\Sp^1} \left(\int_{u \in T_c^{-1}(\{\theta\})} f(\theta) \frac{1}{\kappa(u)} d\mathcal{H}^0(u) \right) d\theta = \int_{\Sp^1} \left(\int_{u \in T_c^{-1}(\{\theta\})} \frac{1}{\kappa(u)} d\mathcal{H}^0(u) \right) f(\theta) d\theta. 
\end{equation*}
and furthermore for almost all $\theta \in \Sp^1$, $T_c^{-1}(\{\theta\})$ is a finite set so we obtain:
\begin{equation*}
    (\mu_c | f) = \int_{\Sp^1} \underbrace{\left(\sum_{u \in T_c^{-1}(\{\theta\})} \frac{1}{\kappa_c(u)}\right)}_{=\rho(\theta)} f(\theta) d\theta .
\end{equation*}
\end{proof}
A consequence is that the c.d.f of the length measure of a curve satisfying the assumptions of Proposition \ref{prop:length_measure_dens} is given by $F_{\mu_c}(\theta) = \int_0^\theta \rho(\theta') d\theta'$ and thus does not have any jumps. Such jumps can only occur with the presence of flat faces in the curve, in particular for polygons as in Example \ref{ex:poylgon}. Note also that if the curve is smooth and strictly convex, the curvature is non-vanishing everywhere and the mapping $T_c$ is a bijection which implies in this case that $\rho(\theta) = \frac{1}{\kappa_c(T_c^{-1}(\theta))}$.

\subsection{Convergence of length measures}
It will be useful for the rest of the paper to examine the topological properties of the mapping $M:c \mapsto \mu_c$. Specifically, we want to determine under what notion of convergence in the space of curves, we can recover convergence of the corresponding length measures. We remind that a sequence of measures $\mu_n \in \mathcal{M}(\Sp^1)$ is said to converge weakly to $\mu \in \mathcal{M}(\Sp^1)$, which will be written $\mu_n \rightharpoonup \mu$, when for all $f \in C(\Sp^1)$, $(\mu_n | f) \rightarrow (\mu | f)$ as $n \rightarrow +\infty$. 

First, it is clear that uniform convergence $c_n \rightarrow c$ in $\mathcal{C}$ (or equivalently convergence in Hausdorff distance of the unparametrized curves in $\tilde{C}$) does not imply that $\mu_{c_n}$ converges to $\mu_c$ even weakly. Indeed it suffices to consider a sequence of staircase curves as the one displayed in Figure \ref{fig:steps_nonconv} which converges in Hausdorff distance to the diamond curve in blue; however, for all $n$, the length measure is identical and equal to $\mu_{c_n} = 2(\delta_{0} + \delta_{\pi/2} + \delta_{\pi} + \delta_{3\pi/2})$ whereas $\mu_c = \sqrt{2}(\delta_{\pi/4} + \delta_{-\pi/4} + \delta_{3\pi/4} + \delta_{5\pi/4})$. 

\begin{figure}
    \includegraphics[width=7cm]{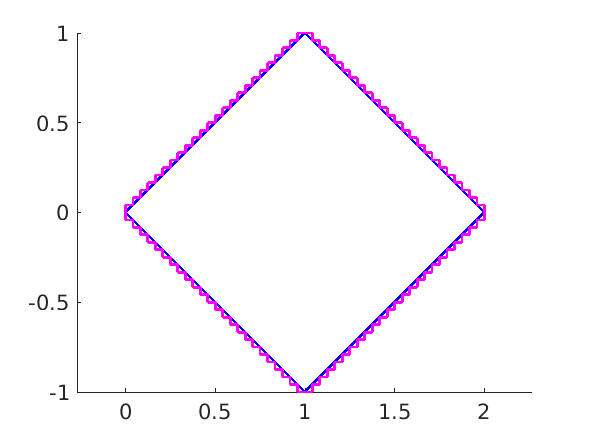} 
    \caption{Two curves close in Hausdorff distance but for which the length measures are respectively $\sqrt{2}(\delta_{\pi/4} + \delta_{-\pi/4} + \delta_{3\pi/4} + \delta_{5\pi/4})$ (blue curve) and $2(\delta_{0} + \delta_{\pi/2} + \delta_{\pi} + \delta_{3\pi/2})$ (red curve).}
    \label{fig:steps_nonconv}
\end{figure}

However, assuming in addition some convergence of the derivatives, we have the following:
\begin{prop}
\label{prop:convergence_length_meas}
Let $(c_n)$ be a sequence of $\mathcal{C}$ with uniformly bounded Lipschitz constant such that there exists $c \in \mathcal{C}$ for which $c_n'(\theta) \rightarrow c'(\theta)$ for almost every $\theta \in \Sp^1$. Then $(\mu_{c_n})$ converges weakly to $\mu_c$ and for all Borel set $B \subset \Sp^1$ with $\mu_c(\partial B) =0$ one has $\mu_{c_n}(B) \rightarrow \mu_{c}(B)$.
\end{prop}
\begin{proof}
Let $f$ be a continuous real-valued function on $\Sp^1$. 
\begin{equation*}
    (\mu_{c_n} | f) = \int_{\Sp^1} f \circ T_{c_n}(\theta) |c_n'(\theta)| d\theta = \int_{\Sp^1} f\left(\frac{c_n'(\theta)}{|c_n'(\theta)|} \right) |c_n'(\theta)| d\theta. 
\end{equation*}
By the assumptions above, outside a set of measure $0$ in $\Sp^1$, we have $c_n'(\theta) \rightarrow c'(\theta)$ with $c'(\theta)\neq 0$ and $c_n'(\theta) \neq 0$ for all $n$. Thus for almost all $\theta$ in $\Sp^1$, $c_n'(\theta)/|c_n'(\theta)| \rightarrow c'(\theta)/|c'(\theta)|$ and thus 
\begin{equation*}
   f\left(\frac{c_n'(\theta)}{|c_n'(\theta)|} \right) |c_n'(\theta)| \xrightarrow[n \rightarrow +\infty]{} f\left(\frac{c'(\theta)}{|c'(\theta)|} \right) |c'(\theta)| 
\end{equation*}
since $f$ is continuous. Furthermore, as $f$ is bounded, there exists $\alpha>0$, such that $f(c_n'(\theta)/|c_n'(\theta)|)\leq \alpha$ for all $n$ and by assumption, there also exists $\beta>0$ such that $|c_n'(\theta)|\leq \beta$. It results from Lebesgue dominated convergence theorem that:
\begin{equation*}
    (\mu_{c_n} | f) \xrightarrow[n \rightarrow +\infty]{} \int_{\Sp^1} f\left(\frac{c'(\theta)}{|c'(\theta)|} \right) |c'(\theta)| d\theta = (\mu_{c} | f).
\end{equation*}
and so $\mu_{c_n} \rightharpoonup \mu_{c}$. The second statement is a classical consequence of weak convergence of Radon measures (see \cite{Evans2015} Theorem 1.40).
\end{proof}
Another important question is whether polygonal approximation of a curve leads to consistent length measures which we answer in the particular case of piecewise smooth curves:
\begin{prop}
\label{prop:convergence_polygons}
Let $c$ be a curve of $\mathcal{C}$ which is assumed in addition to be piecewise twice differentiable with bounded second derivative. If $(c_n)$ is a sequence of polygonal curves with $k_n$ vertices given by $c(\theta_{n,0}),c(\theta_{n,1}),\ldots,c(\theta_{n,k_n})$ with $0=\theta_{n,0} < \theta_{n,1} < \ldots < \theta_{n,k_n}=2\pi$ and $\max_{i} \{\theta_{i+1,n} - \theta_{i,n}\} \rightarrow 0$ as $n\rightarrow +\infty$, then $(\mu_{c_n})$ converges weakly to $\mu_c$.    
\end{prop}
\begin{proof}
Using Proposition \ref{prop:convergence_length_meas}, we simply need to show that $c_n'$ converges pointwise to $c'$ a.e and bound the Lipschitz constant of $c_n$ uniformly. As $c$ is piecewise smooth, we can treat each of the intervals separately and fix $\theta \in \Sp^1$ with $c$ being twice differentiable at $\theta$. For $n\geq 1$, let us denote $h_n = \max_{i} \{\theta_{i+1,n} - \theta_{i,n}\}$. For $n$ large enough,  let $\theta_{j,n} \leq \theta < \theta_{j+1,n}$ with $c$ twice differentiable on $[\theta_{j,n},\theta_{j+1,n})$. Then we have by definition $c_n'(\theta) = \frac{c(\theta_{j+1,n}) - c(\theta_{j,n})}{\theta_{j+1,n}-\theta_{j,n}} = c'(\xi)$ for a certain $\xi \in (\theta_{j,n},\theta_{j+1,n})$ by Taylor's theorem which, using the boundedness of the second derivative of $c$, gives
\begin{equation*}
    |c_n'(\theta) - c'(\theta)| \leq \|c''\|_{L^{\infty}} |\xi-\theta| \leq \|c''\|_{L^{\infty}} h_n \xrightarrow[n \rightarrow +\infty]{} 0
\end{equation*}
Thus $c_n'$ converges to $c'$ pointwise except maybe on the finite set of points where $c$ is not twice differentiable. Furthermore, the above equation also implies that $\text{Lip}(c_n) \leq \|c''\|_{L^{\infty}} h_n + \text{Lip}(c)$ which is uniformly bounded in $n$ and therefore Proposition \ref{prop:convergence_length_meas} leads to the conclusion.  
\end{proof}
Note that the above approximation property requires more regularity on $c$. We will see in the next section that the result holds for all convex curves as well.   

\section{Convex curves and length measures}
\label{sec:convex_curves}
As already pointed out, the length measure $\mu_c$ does not characterize the curve itself since there are in fact infinitely many curves of $\tilde{\mathcal{C}}$ in the fiber $M^{-1}(\{\mu_c\})$ for any given $c$. Yet, remarkably, it is the case if one restricts to convex curves in $\tilde{\mathcal{C}}$. This follows from the fundamental Minkowski-Fenchel-Jessen theorem (also in part due to Alexandrov) established in \cite{fenchel1938,alexandroff1937} that shows in general dimension the uniqueness of a convex domain associated to any given \textit{area measure}. For the specific case of planar curves which is the focus of this paper, we provide a more direct and constructive proof of this result in the following section before highlighting a few other well-known connections with convex geometry.

\subsection{Characterization of convex curves by their length measure}
In all the following, we will denote by $\tilde{\mathcal{C}}_{conv}$ the set of curves in $\tilde{\mathcal{C}}$ which are simple, convex and positively oriented. For technical reasons that will appear later, will adopt the convention that degenerate convex curves made of two opposite segments (which length measures are of the form $r(\delta_{\theta_0} + \delta_{\theta_0+\pi})$) belong to $\tilde{\mathcal{C}}_{conv}$. For simplicity, we shall still write $M: \tilde{\mathcal{C}}_{conv} \rightarrow \mathcal{M}_0$ for the restriction of the previous length measure mapping to $\tilde{\mathcal{C}}_{conv}$. Before stating the main connection between curves in $\tilde{\mathcal{C}}_{conv}$ and length measures, let us start with the following lemma.   
\begin{lemma}
\label{lemma:proof_convex}
 Let $c \in \mathcal{C}$ such that there exists $0\leq \theta_1 < \theta_2 <2\pi$ with $c(\theta_1) = c(\theta_2)$. Then there exist $\theta,\tilde{\theta} \in [\theta_1,\theta_2]$ such that $|\text{angle}(c'(\theta))-\text{angle}(c'(\tilde{\theta}))|\geq \pi$ with strict equality unless $c([\theta_1,\theta_2])$ is a segment.
\end{lemma}
\begin{proof}
 By contradiction, let's assume that given $\theta_1,\theta_2$ as above, we have $|\text{angle}(c'(\theta))-\text{angle}(c'(\tilde{\theta}))|< \pi$ for all $\theta,\tilde{\theta} \in [\theta_1,\theta_2]$ where $c'(\theta)$ is defined. Up to a rotation and translation of the curve, we may assume that $c(\theta_1)=c(\theta_2)=0$ and that $\text{angle}(c'(\theta)) \in [0,\pi)$ for $\theta \in [\theta_1,\theta_2]$. Assuming that $\text{angle}(c'(\theta)) =0$ a.e on $[\theta_1,\theta_2]$ would lead to $c(\theta_2) = \int_{\theta_1}^{\theta_2} |c'(\theta)| d\theta >0$ which is impossible. On the other hand, if $\text{angle}(c'(\theta))\neq 0$ on a subset of $[\theta_1,\theta_2]$ of non-zero measure, then we would have:
 \begin{equation}
 \label{eq:lemma_imag_part}
  \text{Imag}(c(\theta_2)) = \int_{\theta_1}^{\theta_2} |c'(\theta)| \sin(\text{angle}(c'(\theta))) d \theta 
 \end{equation}
strictly positive which is again a contradiction. It results that we can find $\theta,\tilde{\theta} \in [\theta_1,\theta_2]$ such that $|\text{angle}(c'(\theta))-\text{angle}(c'(\tilde{\theta}))|\geq \pi$. Furthermore, if $\max_{\theta \in [\theta_1,\theta_2]} |\text{angle}(c'(\theta))-\text{angle}(c'(\tilde{\theta}))| = \pi$, it can be easily seen from a similar reasoning as above and \eqref{eq:lemma_imag_part} that in this case, one must have $\text{angle}(c'(\theta)) = 0$ or $\pi$ a.e on $[\theta_1,\theta_2]$ and therefore the curve is a subset of the horizontal line.   
\end{proof}

The following result is a reformulation of the Minkowski-Fenchel-Jessen theorem in the special setting of this paper. The usual proof of this theorem in general dimensions is quite involved and requires many preliminary results and notions from convex geometry. We propose here an alternative proof specific to the case of curves which is more elementary and constructive.      
\begin{theorem}
\label{thm:tangent_length_map}
 The length measure mapping $M$ is a bijection from $\tilde{C}_{conv}$ to $\mathcal{M}^+_0$. 
\end{theorem}
\begin{proof}
 1. We first show the surjectivity of $M$. Thus, taking $\mu \in \mathcal{M}^+_0$, we want to construct $c \in \mathcal{C}$, a parametrization of a curve in $\tilde{\mathcal{C}}_{conv}$ such that $M(c) = \mu_c = \mu$. To do so, let us first assume, without loss of generality thanks to Proposition \ref{prop:basic_properties_mu} on the action of rescaling on length measures, that $\mu(\Sp^1)=2\pi$. Then let $F_{\mu}:[0,2\pi] \rightarrow [0,2\pi]$ be the c.d.f of $\mu$ as we defined previously and let's introduce its pseudo-inverse $F_{\mu}^{(-1)}: [0,2\pi] \rightarrow [0,2\pi]$:
 \begin{equation*}
  F_{\mu}^{(-1)}(s) = \inf\{\theta \in [0,2\pi] \ | \  F_{\mu}(\theta)\geq s\}.
 \end{equation*}
 Note that, as $F_{\mu}$, the pseudo-inverse $F_{\mu}^{(-1)}$ is a non-decreasing function with $F_{\mu}^{(-1)}(0)=0$. We now define $c:[0,2\pi] \rightarrow \C$ as follows:
 \begin{equation}
  \label{eq:def_c_mu}
  c(\theta) = \int_{0}^{\theta} e^{i F_{\mu}^{(-1)}(s)} ds.
 \end{equation}
 We first see that $c \in \mathcal{C}$. Indeed, by construction, $c$ is differentiable almost everywhere with $c'(\theta) = e^{i F_{\mu}^{(-1)}(\theta)}$ giving $|c'(\theta)|=1$ thus $c \in \text{Lip}(\Sp^1,\C)$ and $c$ is a Lipschitz immersion. Next, we obtain that $\mu_c = \mu$ by checking the equality $F_{\mu_c} = F_{\mu}$. This is simply a consequence of the fact that $F_{\mu_c}(\theta) = \lambda^1(\{\alpha \ | \ 0 \leq F_{\mu}^{(-1)}(\alpha) \leq \theta\})$ and the easy verification that $F_{\mu}^{(-1)}(\alpha) \leq \theta \Leftrightarrow \alpha \leq F_{\mu}(\theta)$ so that 
 \begin{equation*}
  F_{\mu_c}(\theta) = \lambda^1(\{\alpha \ | \ 0 \leq \alpha \leq F_{\mu}(\theta)\}) = F_{\mu}(\theta).
 \end{equation*}
Incidentally, this also shows that the curve is indeed closed as:
\begin{equation*}
 c(2\pi)-c(0) = \int_0^{2\pi} c'(\theta) d\theta = \int_{0}^{2\pi} e^{i\theta} d\mu_c(\theta) = \int_{0}^{2\pi} e^{i\theta} d\mu(\theta) =0.
\end{equation*}
One still needs to verify that the image of $c$ belongs to $\tilde{\mathcal{C}}_{conv}$. As $c'(\theta) = e^{i F_{\mu}^{(-1)}(\theta)}$ and $F_{\mu}^{(-1)}$ is non-decreasing, $c$ is locally convex and we only need to show that $c$ is simple or is supported by a straight segment. By contradiction, let's assume that there exists $0\leq \theta_1 < \theta_2<2\pi$ such that $c(\theta_1)=c(\theta_2)$. We have to consider the two following cases: 
\begin{itemize}
 \item If $c([0,\theta_1])$ is not a segment then, since $\text{angle}(c'(\theta))$ is non-decreasing on $[0,\theta_1]$, there is a limit $\alpha = \lim_{\theta \uparrow \theta_1} \text{angle}(c'(\theta)) \in [0,2\pi)$ and we have that $c(0)$ is strictly above the line passing by $c(\theta_1)$ and directed by $e^{i\alpha}$. Furthermore, using Lemma \ref{lemma:proof_convex}, we deduce that for all $\theta \geq \theta_2$, we have $\alpha +\pi \leq \text{angle}(c'(\theta)) <2\pi$ which implies that for all $\theta \geq \theta_2$, $c(\theta)$ is below the line passing by $c(\theta_2)=c(\theta_1)$ directed by $e^{i\alpha}$ which contradicts the fact that $\lim_{\theta\rightarrow 2\pi^-} c(\theta) = c(0)$. 
 \item If $c([0,\theta_1])$ is a segment, then by Lemma \ref{lemma:proof_convex} and the fact that $\theta \mapsto \text{angle}(c'(\theta))$ is non-decreasing, we have that $\text{angle}(c'(\theta))\geq \text{angle}(c'(0))+\pi$ for all $\theta_2 \leq \theta < 2\pi$ with even strict inequality if $c([\theta_1,\theta_2])$ is not a segment of the same direction as $c([0,\theta_1])$. The latter case is not possible since we would then find that for $\theta>\theta_2$, $c(\theta)$ is strictly below the line containing the segment $c([0,\theta_1])$ and thus $\lim_{\theta \rightarrow 2\pi^{-}} c(\theta) \neq c(0)$. Finally, by a similar argument, we find that the image of $c$ on $[\theta_2,2\pi)$ is again a segment of the same direction, which shows that $\text{Im}(c)$ is eventually a segment.   
\end{itemize}
Thus the curve is either simple, convex and positively oriented or is a segment, and in all cases belong to $\tilde{\mathcal{C}}_{conv}$.
\vskip1ex
2. Let us now prove the injectivity of $M$. Specifically, we show that if $c$ is the arclength parametrization of a curve in $\tilde{C}_{conv}$, which we assume once again to have length $2\pi$, and $\mu_c = \mu \in \mathcal{M}^+_0$ then for almost all $s\in[0,2\pi)$, $c'(s)=e^{i F_{\mu}^{(-1)}(s)}$. By assumption, we have $|c'(s)|=1$ and since the curve is convex and positively oriented, up to a shifting of the parameter, we may assume that $s \mapsto \text{angle}(c'(s))$ is non-decreasing from $[0,2\pi)$ to $[0,2\pi)$. Let us denote $\nu(s) \doteq \text{angle}(c'(s)) \in [0,2\pi)$. From $F_{\mu_c}=F_{\mu}$ we get that $F_{\mu}(\theta) = \lambda^1(\{s \ | \ 0\leq \nu(s) \leq \theta\})$. We need to show that $\nu(s)=F_{\mu}^{(-1)}(s) = \inf\{\theta \ | \ F_{\mu}(\theta)\geq s\}$. On the one hand, we have $F_{\mu}(\nu(s))=\lambda^1(\{\tilde{s} \ | \ 0\leq \nu(\tilde{s}) \leq \nu(s)\})\geq s$ since for any $\tilde{s}\in[0,s]$ we have $0\leq \nu(\tilde{s}) \leq \nu(s)$. This leads to $F_{\mu}^{(-1)}(s)\leq \nu(s)$. On the other hand, for any $\theta < \nu(s)$ we have that $F_{\mu}(\theta)<s$ by definition of $\nu(s)$ and consequently $F_{\mu}^{(-1)}(s)\geq \theta$ for all $\theta<\nu(s)$ leading to $F_{\mu}^{(-1)}(s)\geq \nu(s)$.
\end{proof} 
\begin{remark}
\label{rem:reconstruction_convex_curve}
 Note that from the above proof, one can technically reconstruct a convex curve up to translation from its length measure directly based on \eqref{eq:def_c_mu}. When the measure is discrete i.e. $\mu=\sum_{i=1}^{N} l_j \delta_{\alpha_j}$ where $l_j>0$ and $\alpha_j \in [0,2 \pi)$ for all $j=1,\ldots,N$, the reconstruction becomes particularly simple. In this case, one can see that the corresponding convex polygon is obtained by selecting an initial vertex (e.g. at the origin) and ordering the edges $l_{j_1} e^{i\alpha_{j_1}},\ldots,l_{j_N} e^{i\alpha_{j_N}}$ such that the angles $0\leq \alpha_{j_1} \leq \ldots \leq \alpha_{j_N}<2\pi$ are in ascending order, which is a well-known algorithm for convex planar objects. However, this reconstruction is a significantly more difficult problem for area measures in higher dimensions, c.f. the discussion in Section \ref{sec:future_directions}.
\end{remark}

\subsection{Length measures and Minkowski sum of convex sets}
\label{ssec:area_measures}
The above correspondence between convex shapes and length measures has many interesting consequences and applications, in particular for the study of mixed areas and Brunn-Minkowski theory, as developed for example in \cite{Letac1983,Heijmans1998,schneider_2013}. We will not go over all of these in detail but only recap in this section a few results which shall be relevant for the rest of the paper.

First, in the category of planar convex curves, one has the following stronger version of the approximation result of Proposition \ref{prop:convergence_polygons}:
\begin{prop}
\label{prop:convergence_convex_polygons}
 Let $C$ be a planar convex domain with boundary $c=\partial C \in \tilde{C}_{conv}$ and $(P_n)$ a sequence of convex polygons of boundary $p_n = \partial P_n$ that converges in Hausdorff distance to $C$. Then $\mu_{p_n} \rightharpoonup \mu_{c}$ as $n \rightarrow +\infty$. Furthermore, the area of $P_n$ converges to the area of $C$ i.e. $\lambda^2(P_n) \rightarrow \lambda^2(C)$.
\end{prop}
This is classical property of convex sets and length measures which proof can be found in \cite{schneider_2013} (Theorem 1.8.16 and Theorem 4.1.1).

We now recall the definition of the Minkowski sum. If $C_1$ and $C_2$ are two convex planar domains, their Minkowski sum (also known as dilation in mathematical morphology) is defined by $C_1 + C_2 = \{x_1+x_2 \ | \ x_1 \in C_1, x_2 \in C_2\}$ which is also a convex planar domain. More generally, one can define the Minkowski combination $a_1 C_1 + a_2 C_2$ for $a_1, a_2 \geq 0$ as $a_1 C_1 + a_2 C_2 = \{a_1 x_1 + a_2 x_2 \ | \ x_1 \in C_1, x_2 \in C_2\}$. This allows to view the set of all convex domains as a convex cone for this Minkowski addition. The length measure mapping $M$ has the following interesting property (\cite{Letac1983} Theorem 3.2):
\begin{prop}
 Let $C_1$ and $C_2$ be two convex domains and $c_1, c_2$ their boundary curves. Then, denoting $c \in \mathcal{C}$ a parametrization of the boundary of $C_1 + C_2$, it holds that $\mu_{c} = \mu_{c_1} + \mu_{c_2}$. 
\end{prop}
\begin{proof}
 This is easily shown by approximating the convex domains $C_1$ and $C_2$ by sequences of convex polygons and using Proposition \ref{prop:convergence_convex_polygons}. The fact that the result holds for two convex polygons is well-known and is actually used algorithmically for the computation of the Minkowski sum of polygons in the plane in linear time, c.f. for example \cite{van2000computational} (Chap. 13).   
\end{proof}

By combining the above with Theorem \ref{thm:tangent_length_map} and Proposition \ref{prop:basic_properties_mu} (3), we can summarize the properties of the length measure mapping $M$ as follows:
\begin{corollary}
\label{cor:isometry_convex_cones}
 The map $M:C \mapsto \mu_{\partial C}$ is an isomorphism of convex cones between $\tilde{C}_{conv}$ and $\mathcal{M}_0$. 
\end{corollary}
\begin{figure}
    \begin{tabular}{cccc}
    \includegraphics[trim = 10mm 10mm 10mm 5mm ,clip,width=4cm]{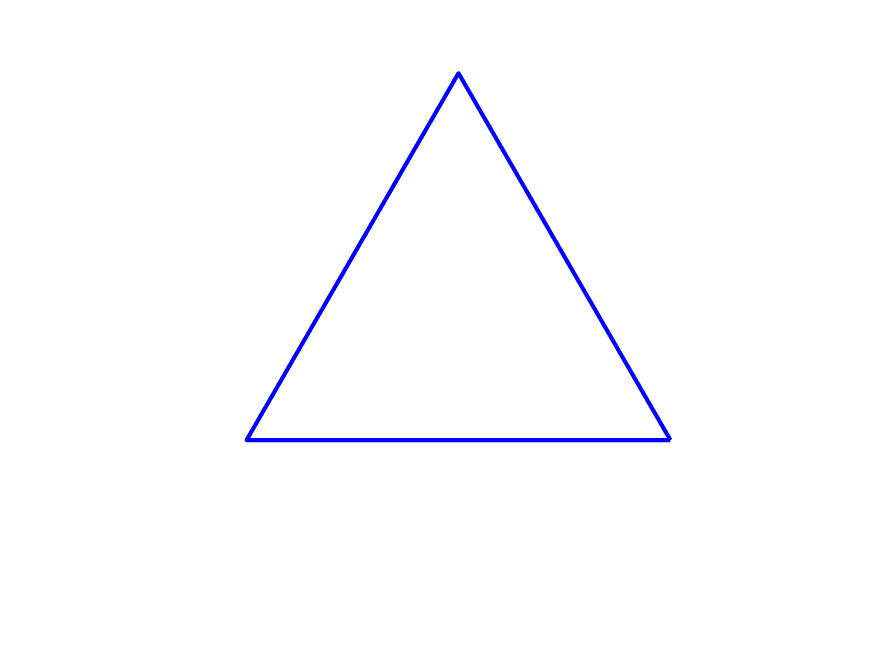} 
    &\includegraphics[trim = 10mm 10mm 10mm 5mm ,clip,width=4cm]{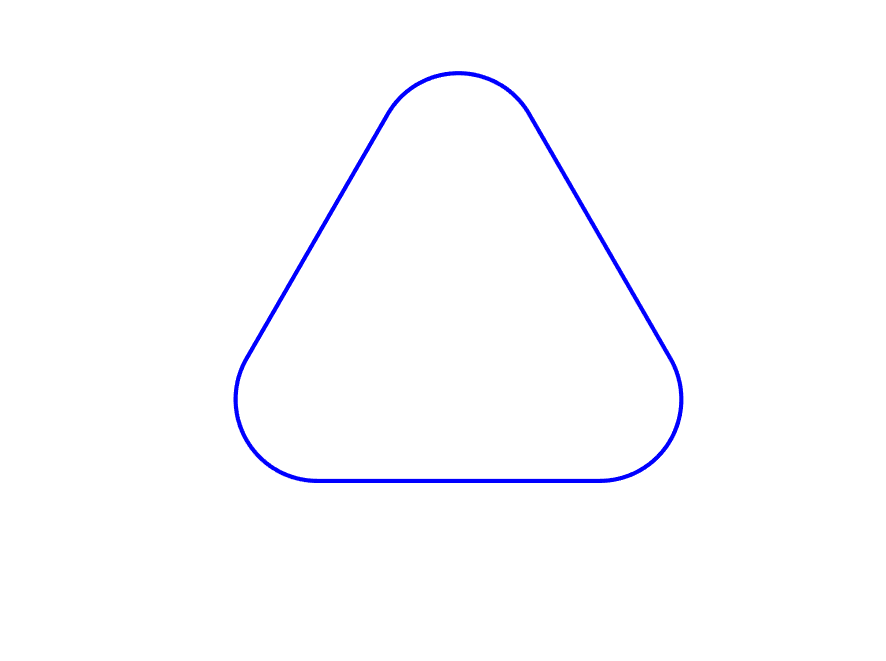}
    &\includegraphics[trim = 10mm 10mm 10mm 5mm ,clip,width=4cm]{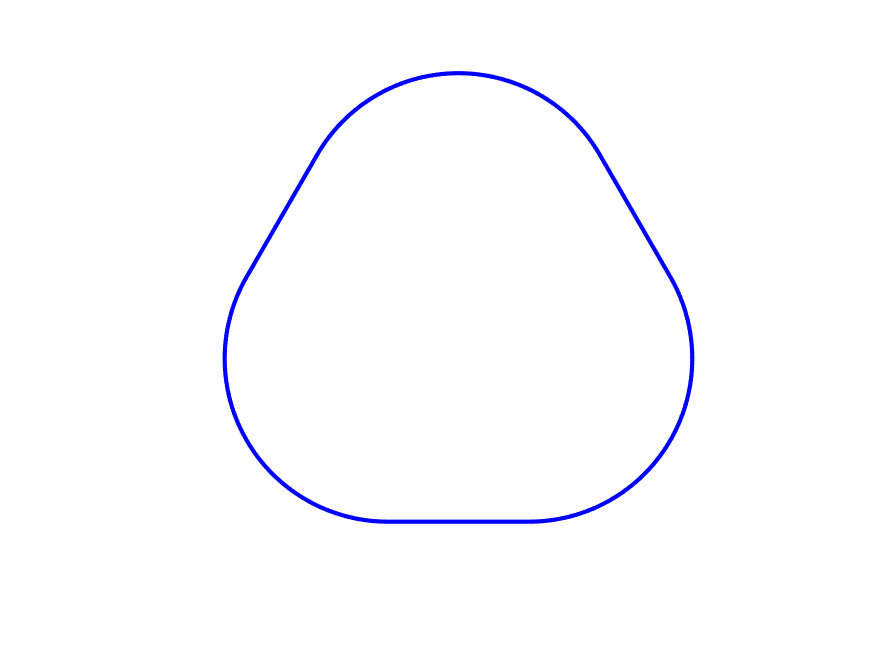} 
    &\includegraphics[trim = 10mm 10mm 10mm 5mm ,clip,width=4cm]{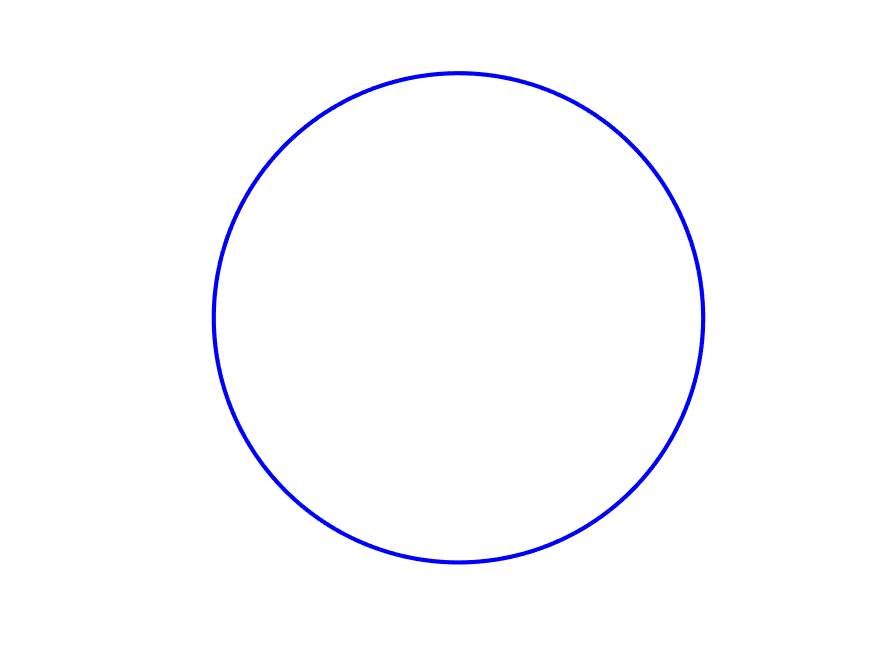}\\
    $\lambda=0$ & $\lambda=1/3$ & $\lambda=2/3$ & $\lambda=1$
    \end{tabular}
    \caption{Minkowski sum $(1-\lambda)C_1+\lambda C_2$ for different values of $\lambda$ of a triangle and a disk computed based on the addition of their length measures.}
    \label{fig:minkowski_combination}
\end{figure}
This implies that if $C_1$ and $C_2$ are two convex domains with length measures $\mu_{\partial C_1}$ and $\mu_{\partial C_2}$, their Minkowski sum $C= C_1+C_2$ is such that $\partial C = M^{-1}(\mu_{\partial C_1} + \mu_{\partial C_2})$ which could be directly computed using the inversion formula \eqref{eq:def_c_mu} or, in the case of discrete measures and polygons, by adequately sorting the Diracs appearing in $\mu_{\partial C_1} + \mu_{\partial C_2}$ with angles in ascending order as explained earlier in Remark \ref{rem:reconstruction_convex_curve}. We show an example of Minkowski sum computed with this approach in Figure \ref{fig:minkowski_combination}.

\section{An isoperimetric characterization}
\label{sec:isoperimetric_theorem}
The previous section showed that there is a unique convex curve of positive orientation in the preimage $M^{-1}(\{\mu\}) \subset \tilde{\mathcal{C}}$ for any measure $\mu \in \mathcal{M}^+_0$. Several previous works on length and area measures have investigated variational characterizations of these convex objects in the context of shape optimization and geometric inequalities, see for instance the survey of \cite{gardner2002brunn} and \cite{carlier2004theorem}. These are typically expressing some variational property within the set of convex shapes only. We prove here a distinct characterization which can be rather interpreted as an isoperimetric inequality in each of the fiber $M^{-1}(\{\mu\})$, namely the convex curve of $M^{-1}(\{\mu\})$ is also the one of maximal signed area among all the curves in $\tilde{\mathcal{C}}$ of length measure $\mu$. Our result extends to the whole class of Lipschitz regular curves some related results on polytopes in discrete geometry that were stated in \cite{Boroczky1986}. We will also see how the classical isoperimetric inequality on curves can be recovered as a consequence.

Let us first recall that the signed area of a curve in $\tilde{\mathcal{C}}$ with Lipschitz parametrization $c \in \mathcal{C}$ is given by (cf \cite{Younes2019} Chap 1.10): 
\begin{equation}
 \label{eq:def_signed_area}
 \text{Area}(c) = -\frac{1}{2} \int_{\Sp^1} \langle c(\theta),N_c(\theta) \rangle |c'(\theta)| d\theta = \frac{1}{2} \int_{\Sp^1} \det(c(\theta),c'(\theta)) d\theta
\end{equation}
where $N_c$ denotes the unit normal vector to the curve. Note that for a simple and positively oriented curve, \eqref{eq:def_signed_area} is the usual area enclosed by this curve. We begin by reminding a few preliminary properties of the signed area. The first one is that the signed area is additive with respect to the ``gluing'' of two cycles, namely:
\begin{lemma}
 \label{lemma:area_deccomp}
 Let $c \in \mathcal{C}$ and $0 \leq \theta_1 < \theta_2 < 2\pi$. Consider any given Lipschitz open curve $\gamma: [\theta_1, \theta_2] \rightarrow \C$ with $\gamma(\theta_1) = c(\theta_1)$ and $\gamma(\theta_2) = c(\theta_2)$ and denote $\check{\gamma}$ the same curve but with opposite orientation. Define $c_1$, $c_2$ the two closed curves in $\mathcal{C}$ obtained by respectively concatenating $c(\Sp^1 \backslash [\theta_1, \theta_2])$ with $\gamma$ and $c([\theta_1, \theta_2])$ with $\check{\gamma}$. Then:
 \begin{equation*}
  \text{Area}(c) = \text{Area}(c_1) + \text{Area}(c_2)
 \end{equation*}
\end{lemma}
\begin{proof}
This is just a direct verification from the definition \eqref{eq:def_signed_area}. Indeed, we have $\check{\gamma}(\theta) = \gamma(\theta_1 + \theta_2 -\theta)$ for $\theta \in [\theta_1,\theta_2]$ and:
\begin{align*}
 &2\text{Area}(c) = \int_{\Sp^1\backslash [\theta_1, \theta_2]} \det(c(\theta),c'(\theta)) d\theta + \int_{\theta_1}^{\theta_2} \det(c(\theta),c'(\theta)) d\theta \\
 &= \int_{\Sp^1\backslash [\theta_1, \theta_2]} \det(c(\theta),c'(\theta)) d\theta + \int_{\theta_1}^{\theta_2} \det(\gamma(\theta),\gamma'(\theta)) d\theta + \int_{\theta_1}^{\theta_2} \det(c(\theta),c'(\theta)) d\theta - \int_{\theta_1}^{\theta_2} \det(\gamma(\theta),\gamma'(\theta)) d\theta \\
 &= \underbrace{\int_{\Sp^1\backslash [\theta_1, \theta_2]} \det(c(\theta),c'(\theta)) d\theta + \int_{\theta_1}^{\theta_2} \det(\gamma(\theta),\gamma'(\theta)) d\theta}_{=2\text{Area}(c_1)} + \underbrace{\int_{\theta_1}^{\theta_2} \det(c(\theta),c'(\theta)) d\theta + \int_{\theta_1}^{\theta_2} \det(\check{\gamma}(\theta),\check{\gamma}'(\theta)) d\theta}_{=2\text{Area}(c_2)} \\
\end{align*}
which leads to the result.
\end{proof}

We also have the following well-known approximation property of the signed area:
\begin{lemma}
 \label{lemma:approx_polygons}
 Let $c \in \mathcal{C}$. There exists a sequence $(p_n)$ of polytopes (i.e. piecewise linear curves) in $\mathcal{C}$ such that $\text{Area}(p_n) \rightarrow \text{Area}(c)$ and $\mu_{p_n}$ converges weakly to $\mu_c$.  
\end{lemma}
\begin{proof}
 Let $c \in \mathcal{C}$ and by invariance to translation, let's also assume that $c(0)=0$. Then since $c$ is Lipschitz regular, we have $c' \in L^{\infty}(\Sp^1,\C)\subset L^1(\Sp^1,\C)$. Using the density of step functions in $L^1$, we can construct a sequence $(\rho_n)_{n \in \mathbb{N}}$ of step functions such that $\|\rho_n - c'\|_{L^1} \rightarrow 0$ as $n\rightarrow +\infty$. By the converse of Lebesgue's dominated convergence theorem (\cite{brezis2010functional} Theorem 4.9), up to extraction of a subsequence, we can assume that $\rho_n(\theta)$ converges to $c'(\theta)$ almost everywhere in $\Sp^1$ and that there exists $h \in L^1(\Sp^1)$ such that $\rho_n(\theta) \leq h(\theta)$ for almost all $\theta \in \Sp^1$. Let us then define $p_n = \int_0^{\theta} \rho_n(\alpha) d\alpha$ which is a piecewise linear curve in $\mathcal{C}$ such that:
 \begin{align*}
  |p_n(\theta) - c(\theta)| \leq \int_{0}^{\theta} |\rho_n(\alpha) - c'(\alpha)| d\alpha \leq \int_{\Sp^1} |\rho_n(\alpha) - c'(\alpha)| d\alpha = \|\rho_n - c'\|_{L^1}
 \end{align*}
showing that $\|p_n-c\|_{L^\infty} \rightarrow 0$ as $n\rightarrow +\infty$ and as a consequence there exists $M>0$ such that $\|p_n\|_{L^\infty} \leq M$ for all $n$. Now, by Proposition \ref{prop:convergence_length_meas}, we deduce that $\mu_{p_n}$ weakly converges to $\mu_c$. Furthermore, for any $n\in \mathbb{N}$, we have:
\begin{equation*}
 \text{Area}(p_n) = -\frac{1}{2} \int_{\Sp^1} \langle p_n(\theta),N_{p_n}(\theta) \rangle |p_n'(\theta)| d\theta = -\frac{1}{2} \int_{\Sp^1} \langle p_n(\theta),N_{p_n}(\theta) \rangle |\rho_n(\theta)| d\theta
\end{equation*}
and for almost all $\theta \in \Sp^1$, $p_n(\theta)\rightarrow c(\theta)$, $\rho_n(\theta) \rightarrow c'(\theta)$, $N_{p_n}(\theta) = R_{\pi/2}(\rho_n(\theta)/|\rho_n(\theta)|) \rightarrow N_c(\theta)$ as $n\rightarrow +\infty$. In addition, 
\begin{equation*}
 |\langle p_n(\theta),N_{p_n}(\theta) \rangle| |\rho_n(\theta)| \leq |p_n(\theta)| |\rho_n(\theta)| \leq M h(\theta).
\end{equation*}
As $h \in L^1(\Sp^1)$, Lebesgue dominated convergence theorem leads to:
\begin{equation*}
 \text{Area}(p_n) \xrightarrow[n\rightarrow +\infty]{} -\frac{1}{2} \int_{\Sp^1} \langle c(\theta),N_{c}(\theta) \rangle |c'(\theta)| d\theta = \text{Area}(c).
\end{equation*}
\end{proof}
We will also need the following which is a particular case of our main result for polytopes with fixed number of edges, which proof is adapted from the one outlined in \cite{Fary1982,Boroczky1986}. We call a polytope non-degenerate if it does not lie entirely along a single line (in other words if its length measure is not a sum of two oppositely oriented Diracs).
\begin{figure}
    \begin{tabular}{ccc}
    \includegraphics[trim = 10mm 10mm 10mm 0mm ,clip,width=5cm]{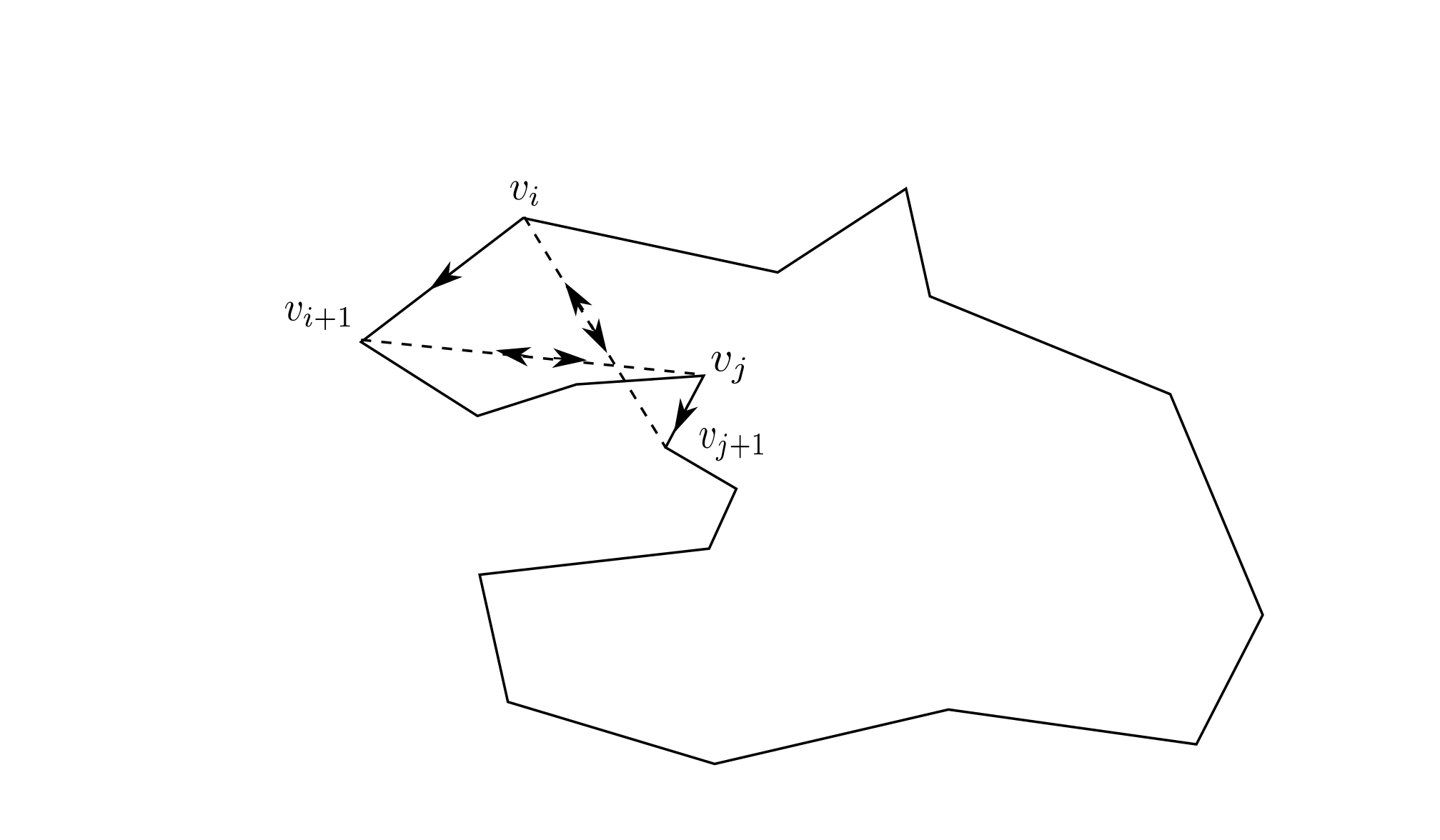} 
    &\includegraphics[trim = 10mm 10mm 10mm 0mm ,clip,width=5cm]{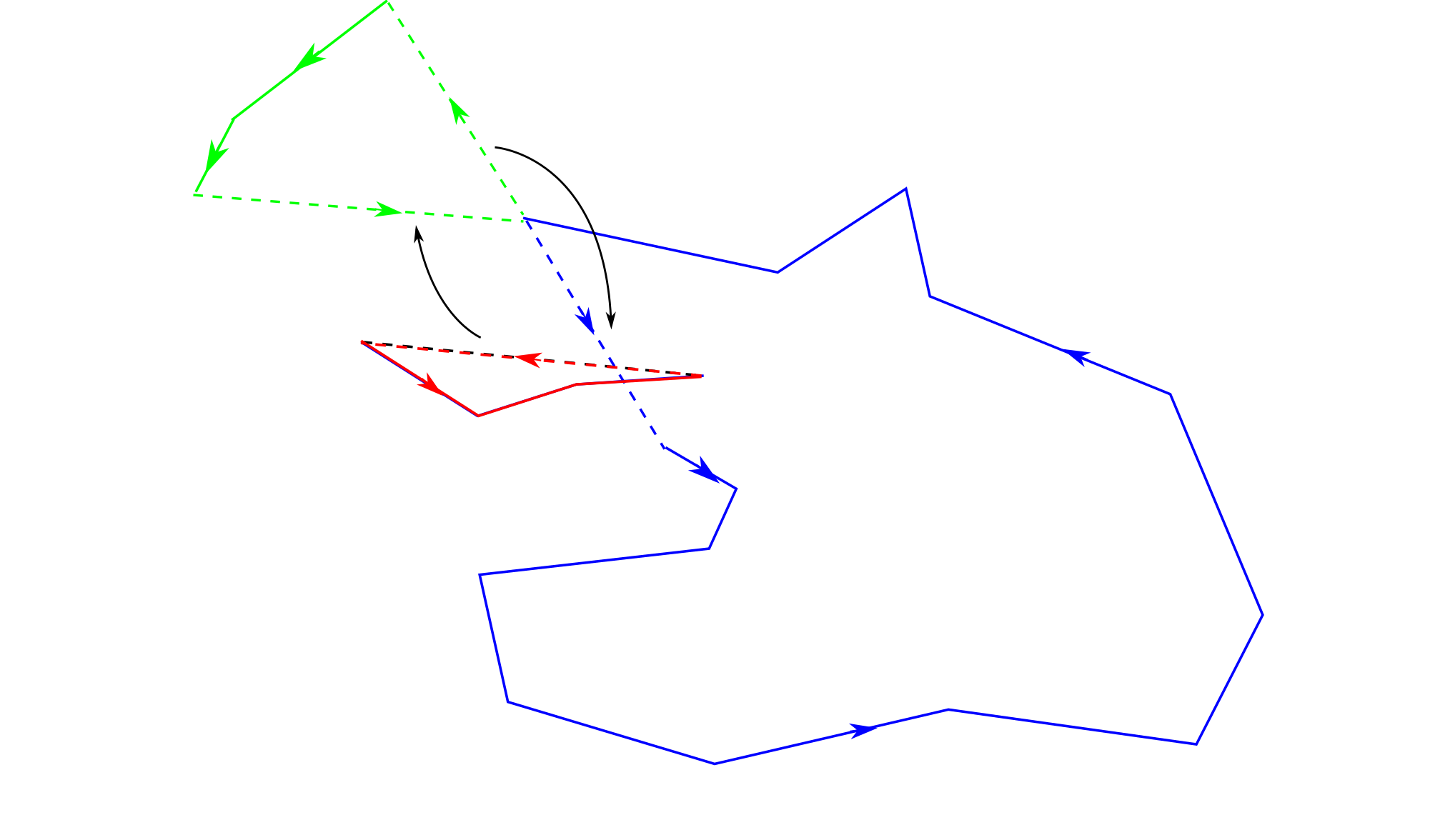}
    &\includegraphics[trim = 10mm 10mm 10mm 0mm ,clip,width=5cm]{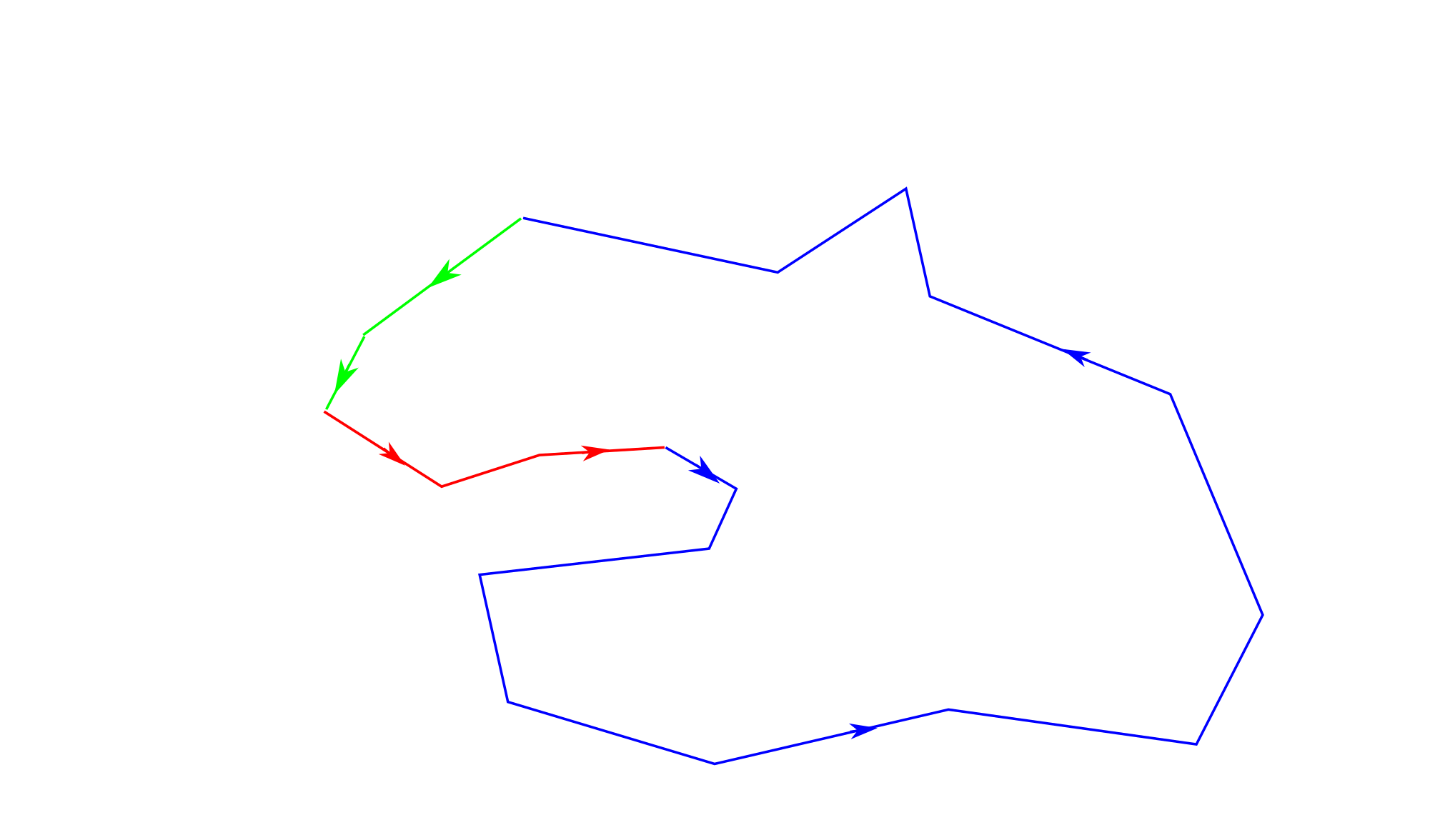}
    \end{tabular}
    \caption{Illustration of the proof of Lemma \ref{lemma:polygon_max_area}: the edges of a non-convex polytope can be rearranged to obtain a polytope of larger signed area.}
    \label{fig:polygon_area}
\end{figure}
\begin{lemma}
 \label{lemma:polygon_max_area}
 Let $p$ be a non-degenerate polytope. Then the convex polygon $\bar{p}$ such that $\mu_{\bar{p}} = \mu_{p}$ satisfies $\text{Area}(p) \leq \text{Area}(\bar{p})$.  
\end{lemma}
\begin{proof}
 Let $p$ be a polytope that we represent by its ordered list of vertices $v_1, v_2, \ldots,v_N$. Its corresponding list of edges are denoted by $\mathcal{E}=(e_1,e_2,\ldots,e_N)$ where we assume the directions of the $e_i$ to be consecutively distinct. Recall that, from Theorem \ref{thm:tangent_length_map}, there is a unique convex and positively oriented polygon $\bar{p}$ such that $\mu_{\bar{p}} = \mu_{p}$. Let us define $S_{\mathcal{E}}$ as the set of polygonal curves obtained by all the different permutations of the edges in $\mathcal{E}$. Note that for any polygonal curve $\tilde{p} \in S_{\mathcal{E}}$, we have $\mu_{\tilde{p}} = \mu_{p}$.  Let us show that:
 \begin{equation}
 \label{eq:max_area_pbar}
  \text{Area}(\bar{p}) = \max_{\tilde{p} \in S_{\mathcal{E}}} \ \ \text{Area}(\tilde{p})
 \end{equation}
which will imply in particular that $\text{Area}(p) \leq \text{Area}(\bar{p})$. Since $S_{\mathcal{E}}$ is finite, the maximum is attained: let $\hat{p}$ be a curve in  $S_{\mathcal{E}}$ of maximal signed area. By contradiction, assume that $\hat{p}$ is not the convex positively oriented polygon $\bar{p}$ (modulo translations). 

Let us first consider the trivial cases. When $N=3$, there are only two triangles up to translation in $S_{\mathcal{E}}$ and thus $\hat{p}$ is the negatively oriented one which signed are is clearly strictly smaller than the one of $\bar{p}$. For $N=4$, the polytope $\hat{p}$ being distinct from $\bar{p}$, one can see that, up to a cyclic permutation of its edges, we can assume without loss of generality that $\det(e_1,e_2) < 0$, i.e. the angle is decreasing between the first and second edge. Now the signed area \eqref{eq:def_signed_area} of the piecewise linear curve $\hat{p}$ can be written more simply as (cf \cite{Younes2019} Chap 1.10):
\begin{equation*}
 \text{Area}(\hat{p}) = \frac{1}{2} \sum_{k=1}^{N} \det(\overrightarrow{v v_k},e_k)
\end{equation*}
where $v_1=O, v_2=e_1, v_3=e_1 + e_2, v_4 =e_1+e_2+e_3$ are the consecutive vertices of $\hat{p}$ and $v$ is any reference point in the plane. The above expression is independent of the choice of this reference point, thus choosing $v=v_1$, we obtain after simplifications:
\begin{equation*}
 \text{Area}(\hat{p}) = \frac{1}{2} [\det(e_1,e_2) + \det(e_1,e_3) + \det(e_2,e_3)].
\end{equation*}
Then we can define the quadrilateral $\tilde{p}$ in which the ordering of the edges $e_1$ and $e_2$ is switched. This is still a polytope of $S_{\mathcal{E}}$ and we have:
\begin{align*}
 \text{Area}(\tilde{p}) &= \frac{1}{2} [\det(e_2,e_1) + \det(e_2,e_3) + \det(e_1,e_3)] \\
 &=\frac{1}{2} [-\det(e_1,e_2) + \det(e_1,e_3) + \det(e_2,e_3)] \\
 &>\frac{1}{2} [\det(e_1,e_2) + \det(e_1,e_3) + \det(e_2,e_3)] = \text{Area}(\hat{p}).
\end{align*}
which contradicts the fact that $\text{Area}(\hat{p})$ is maximal.


Let us now assume $N\geq 5$. Then, by the assumption on $\hat{p}$, we can find four vertices $v_i,v_{i+1},v_j,v_{j+1}$ ($i<j-1$) such that the corresponding quadrilateral with edges $\overrightarrow{v_i v_{i+1}},\overrightarrow{v_{i+1} v_j},\overrightarrow{v_j v_{j+1}},\overrightarrow{v_{j+1} v_i}$ is not both convex and positively oriented. Let us write $p_1$ this quadrilateral and introduce in addition the polytope $p_2$ with successive vertices $v_1,\ldots,v_i,v_{j+1},v_{j+2},\ldots,v_N$ and $p_3$ the polytope with vertices $v_{i+1},v_{i+2},\ldots v_j$. This amounts in decomposing the original polygon into three distinct cycles as depicted in Figure \ref{fig:polygon_area} (left). By the additivity property of the signed area from Lemma \ref{lemma:area_deccomp}, this implies that $\text{Area}(\hat{p}) = \text{Area}(p_1) + \text{Area}(p_2) + \text{Area}(p_3)$. Now, from the case $N=4$ treated above, we can find a permutation of the edges of $p_1$ giving the convex positively oriented quadrilateral $\bar{p}_1$ which satisfies $\text{Area}(\bar{p}_1) > \text{Area}(p_1)$. We then obtain the three polytopes $\bar{p}_1$, $p_2$ and $p_3$ shown respectively in green, blue and red in the example of Figure \ref{fig:polygon_area}. By translation of $p_3$, we can superpose the edge $\overrightarrow{v_{i+1} v_j}$ in $\bar{p}_1$ with the edge $\overrightarrow{v_j v_{i+1}}$ in $p_3$. Similarly, by translation of $p_1$, we match the edge $\overrightarrow{v_{j+1} v_i}$ of $p_1$ on the edge $\overrightarrow{v_i v_{j+1}}$ of $p_2$. Then, removing the trivial back and forth edges resulting from this superposition, one obtains a new polytope $\tilde{p}$ as shown in the right image in Figure \ref{fig:polygon_area}. By construction, this polytope has the same list of edge vectors as $\hat{p}$ and thus belongs to $S_{\mathcal{E}}$. Moreover, its signed area is:
\begin{equation*}
 \text{Area}(\tilde{p}) = \text{Area}(\bar{p}_1) + \text{Area}(p_2) + \text{Area}(p_3) > \text{Area}(p_1) + \text{Area}(p_2) + \text{Area}(p_3) = \text{Area}(\hat{p})
\end{equation*}
which contradicts the maximality of $\text{Area}(\hat{p})$ among polytopes of $S_{\mathcal{E}}$.
\end{proof}
The convex polygon $\bar{p}$ is sometimes referred to as the \textit{convexification} of $p$ (which is distinct from the convex hull of $p$). We now state and prove the main result of this section.
\begin{theorem}
\label{thm:max_area_convex}
Let $\mu \in \mathcal{M}^+_0$ such that the support of $\mu$ is not of the form $\{\theta,\theta+\pi\}$ for some $\theta \in \Sp^1$. Among all curves in $M^{-1}(\{\mu\})$, the convex positively oriented curve is the unique maximum of the signed area and this maximum is given by:
\begin{equation}
\label{eq:expression_max_area}
A_{max}(\mu) = \frac{1}{4} \int_0^{2\pi} \int_0^{2\pi} \sin |\theta - \alpha| d\mu(\theta) d\mu(\alpha).
\end{equation}
\end{theorem}
\begin{proof}
 Let $\mu$ be a measure in $\mathcal{M}^+_0$ satisfying the above assumption on its support. This means that for any curve $c$ in $M^{-1}(\{\mu\})$, the image of $c$ does not lie within a single straight line. Furthermore, we have $\text{Length}(c) = \mu(\Sp^1)\doteq L$. By the standard isoperimetric inequality, this implies that the area of any simple closed curve in $M^{-1}(\{\mu\})$ is bounded by $\frac{L^2}{4\pi}$ and thus the supremum of these areas $A_{max}$ is indeed finite. Let us first show that this maximal area is achieved and by the convex positively oriented curve of $M^{-1}(\{\mu\})$. We let $(c_n)_{n\in \mathbb{N}}$ be a maximizing sequence i.e $c_n \in M^{-1}(\{\mu\})$ and $\text{Area}(c_n) \rightarrow A_{max}$ as $n\rightarrow +\infty$.    
 \vskip1ex
 1. As a first step, we want to replace this maximizing sequence by a sequence of piecewise linear curves. For a fixed $n \in \mathbb{N}$, by Lemma \ref{lemma:approx_polygons}, we can construct a sequence of polytopes $(\tilde{p}_{n,m})_{m\in \mathbb{N}}$ such that $\text{Area}(\tilde{p}_{n,m}) \rightarrow \text{Area}(c_n)$ and $\mu_{\tilde{p}_{n,m}}$ weakly converges to $\mu_{c_n}=\mu$ as $m \rightarrow + \infty$. From this, let us construct a sequence of polytopes $p_n$ such that $\mu_{p_n} \rightharpoonup \mu_c$ and $\text{Area}(p_n) \rightarrow A_{max}$ as $n \rightarrow +\infty$. Indeed, recall that the weak convergence of finite measures of $\Sp^1$ is metrizable, for instance by the bounded Lipschitz distance $d^{BL}(\mu,\nu) \doteq \sup_{f \in \text{Lip}^1(\Sp^1)} |(\mu|f) - (\nu|f)|$ (\cite{villani2008optimal} Chap. 6). Thus, we have for all $n \in \mathbb{N}$, $d^{BL}(\mu_{\tilde{p}_{n,m}},\mu) \rightarrow 0$ as $m \rightarrow +\infty$. Since $\text{Area}(c_n) \rightarrow A_{max}$, for any $n\in \mathbb{N}$, we can find an increasing function $\phi: \mathbb{N} \rightarrow \mathbb{N}$ such that $|\text{Area}(c_{\phi(n)}) - A_{max}|\leq\frac{1}{n}$. In addition, we also obtain an increasing function $\psi: \mathbb{N} \rightarrow \mathbb{N}$ such that for any $n \in \mathbb{N}$, we have $d^{BL}(\mu_{\tilde{p}_{n,\psi(n)}},\mu) \leq 1/n$ and $|\text{Area}(\tilde{p}_{n,\psi(n)}) - \text{Area}(c_n)|\leq 1/n$. We then set $p_n \doteq \tilde{p}_{\phi(n),\psi(\phi(n))}$ which gives on the one hand:
 \begin{equation*}
  d^{BL}(\mu_{\tilde{p}_{n}},\mu) = d^{BL}(\mu_{\tilde{p}_{\phi(n),\psi(\phi(n))}},\mu) \leq \frac{1}{\phi(n)}\leq \frac{1}{n}
 \end{equation*}
and thus $\mu_{p_n} \rightharpoonup \mu$. And on the other hand:
 \begin{align*}
  |\text{Area}(p_n) - A_{max}| &\leq  + |\text{Area}(\tilde{p}_{\phi(n),\psi(\phi(n))}) - \text{Area}(c_{\phi(n)})| + |\text{Area}(c_{\phi(n)}) - A_{max}| \\
  &\leq \frac{1}{\phi(n)} + \frac{1}{n} \leq \frac{2}{n}
 \end{align*}
 and therefore $\text{Area}(p_n) \rightarrow A_{max}$.
 \vskip1ex
 2. Now, using Lemma \ref{lemma:polygon_max_area}, we obtain a sequence of convex, positively oriented polygons $(\bar{p}_n)$ such that $\mu_{\bar{p}_n} = \mu_{p_n} \rightharpoonup \mu$ and for all $n$, $\text{Area}(\bar{p}_n)\geq \text{Area}(p_n)$. Using once again the invariance to translation, we can further assume that each $\bar{p}_n$ passes through the origin. We then point out that the length $L(\bar{p}_n) = \mu_{\bar{p}_n}(\Sp^1) = \mu_{p_n}(\Sp^1)$ converges to $\mu(\Sp^1)<+\infty$ as $n\rightarrow +\infty$ since $\mu_{p_n} \rightharpoonup \mu$. This implies that $L(\bar{p}_n)$ is bounded uniformly in $n$ by some constant $M>0$. If we denote by $\bar{P}_n$ the polygonal domain delimited by $\bar{p}_n$ i.e. $\partial \bar{P}_n = \bar{p}_n$, it is then easy to see that for all $n\in \mathbb{N}$, $\bar{P}_n$ is included in the fixed ball $B(0,M)$. In other words, the sequence $(\bar{P}_n)$ is bounded in the space of compact subsets of the plane equipped with the Hausdorff metric.    
 \vskip1ex
 3. We can therefore apply Blaschke selection theorem \cite{Blaschke1916} which allows to assume, up to extraction of a subsequence, that the sequence of convex polygonal domains $\bar{P}_n$ converges in Hausdorff distance to a convex domain $C$ which oriented boundary curve we write $\partial C = c \in \tilde{C}_{conv}$. Then by Proposition \ref{prop:convergence_convex_polygons}, we deduce that $\mu_{\bar{p}_n}$ weakly converges to $\mu_c$. As a consequence, $\mu_c = \mu$ and $c$ is thus the (unique) convex curve associated to $\mu$ given by Theorem \ref{thm:tangent_length_map}. In addition, still from Proposition \ref{prop:convergence_convex_polygons}, we get that $\lambda^2(\bar{P}_n) \rightarrow \lambda^2(C)$ as $n \rightarrow +\infty$ and since the boundary curves $\bar{p}_n$ and $c$ are simple and positively oriented it follows that $\text{Area}(\bar{p}_n) \rightarrow \text{Area}(c)$ as $n \rightarrow +\infty$. Now, since for all $n\in \mathbb{N}$, $\text{Area}(\bar{p}_n)\geq \text{Area}(p_n) \rightarrow A_{max}$, we conclude that $\text{Area}(c) \geq A_{max}$ and therefore $c$ achieves the maximal area among all curves in $M^{-1}(\{\mu\})$.    
 \vskip1ex
\begin{figure}
    \begin{tabular}{ccc}
    \includegraphics[trim = 30mm 30mm 30mm 30mm ,clip,width=7.5cm]{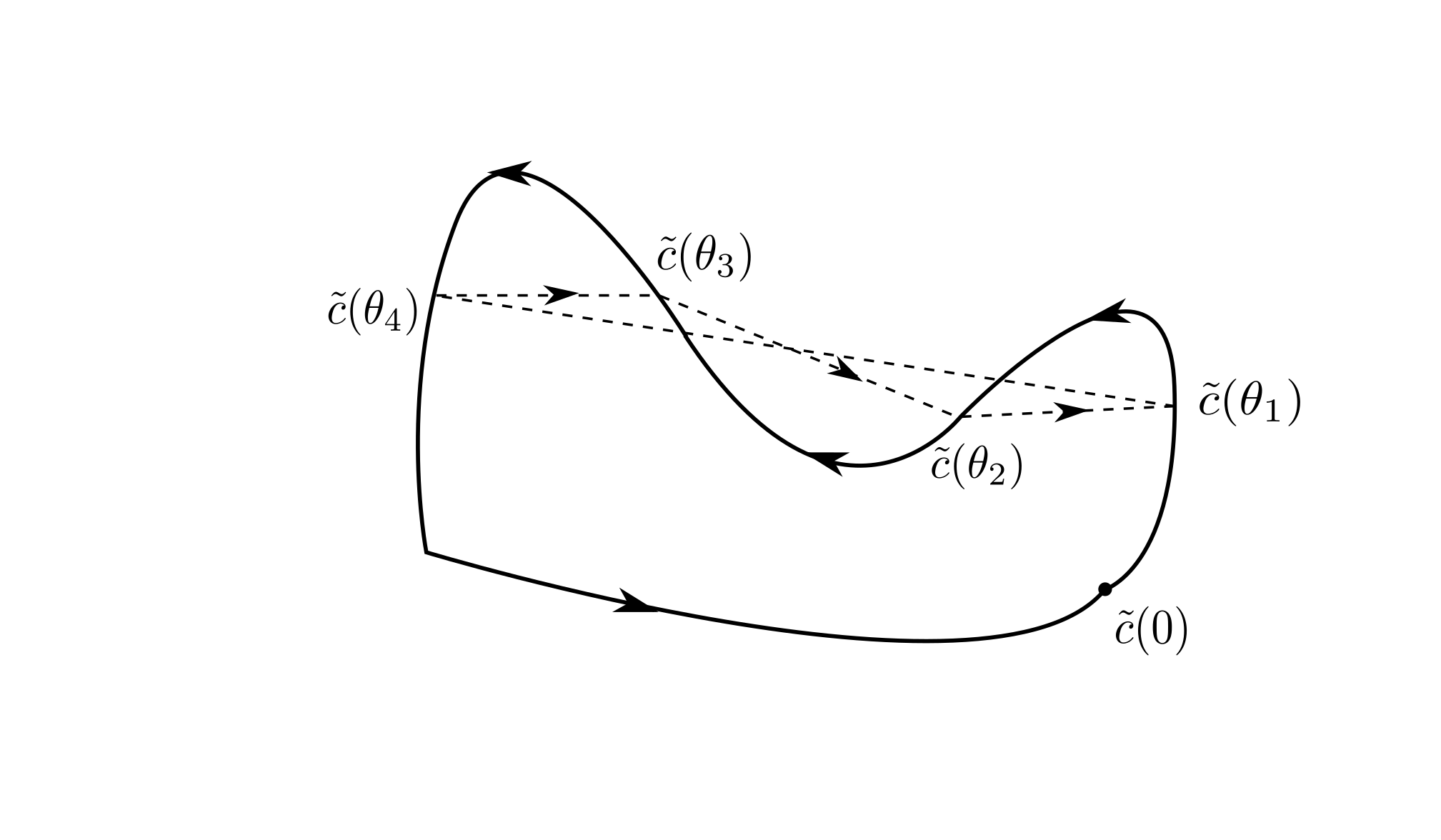} 
    & 
    &\includegraphics[trim = 30mm 30mm 30mm 30mm ,clip,width=7.5cm]{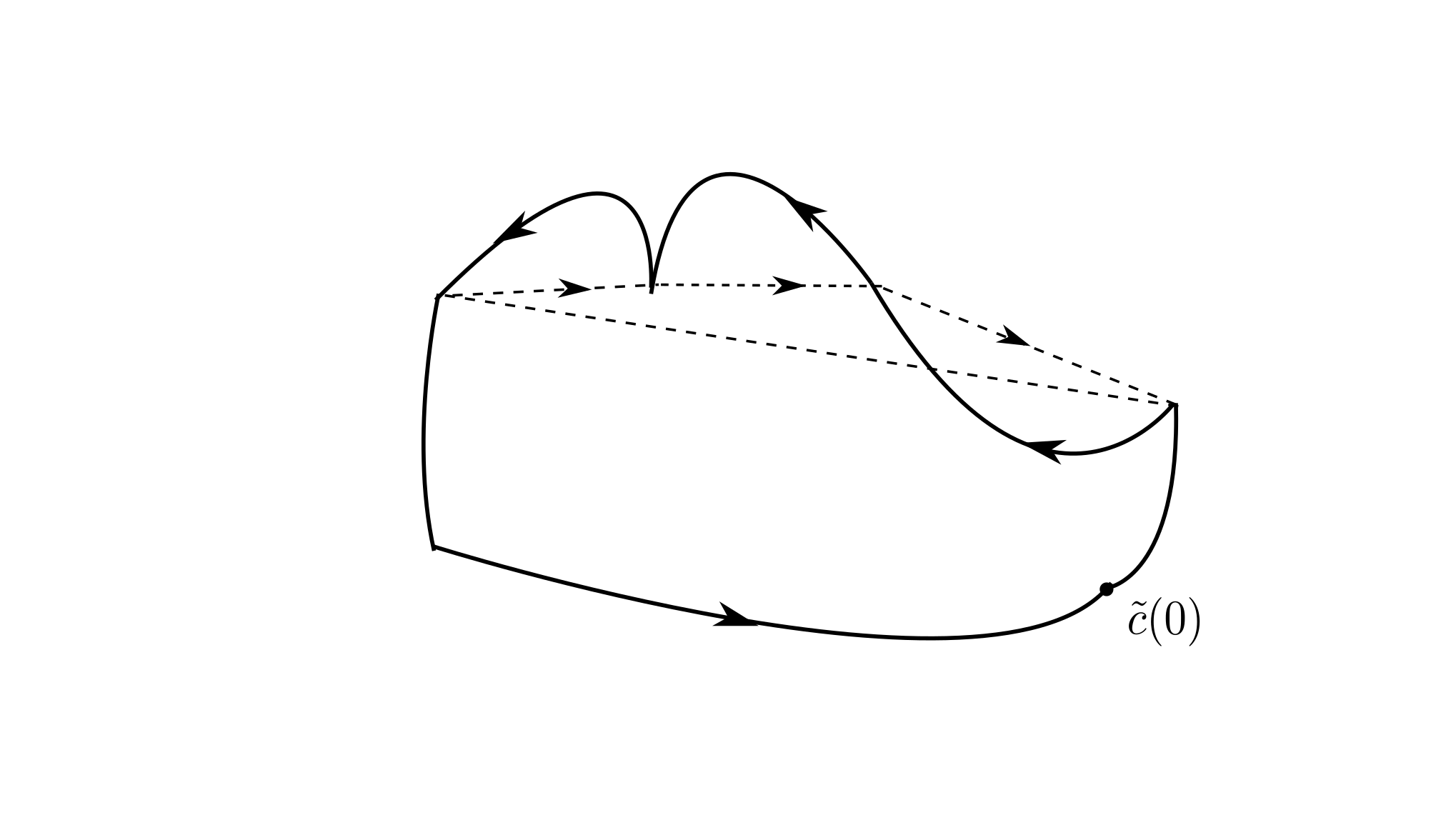}
    \end{tabular}
    \caption{Illustration of step 4 in the proof of Theorem \ref{thm:max_area_convex}. By rearranging the different subcycles in the curve $\tilde{c}$, one obtains a new curve in $M^{-1}(\{\mu\})$ of larger signed area.}
    \label{fig:uniqueness_maximizer}
\end{figure}  
 
 4. The uniqueness of the maximum can be now showed by essentially adapting the argument used in the proof of Lemma \ref{lemma:polygon_max_area} for polytopes. Assume by contradiction that $\tilde{c} \in M^{-1}(\mu)$ is non-convex and that $\text{Area}(\tilde{c})= A_{max}(\mu)$. Then we can find $0 \leq \theta_1 < \theta_2 < \theta_3 < \theta_4 <2\pi$ such that the quadrilateral $Q$ with successive vertices $\tilde{c}(\theta_1),\tilde{c}(\theta_2),\tilde{c}(\theta_3),\tilde{c}(\theta_4)$ is not convex positively oriented and is not degenerate. Let us further introduce $c_1$ defined as the concatenation of the portion of curve $\tilde{c}([0,\theta_1])$, the oriented segment $[c(\theta_1), c(\theta_4)]$ and $\tilde{c}([\theta_4,2\pi])$. In addition, let $c_2$ be the concatenation of $\tilde{c}([\theta_1,\theta_2])$ and the segment $[c(\theta_2), c(\theta_1)]$, $c_3$ the curve obtained by concatenating $\tilde{c}([\theta_2,\theta_3])$ and the segment $[\tilde{c}(\theta_3), \tilde{c}(\theta_2)]$ and $c_4$ the concatenation of $\tilde{c}([\theta_3,\theta_4])$ and $[\tilde{c}(\theta_4), \tilde{c}(\theta_3)]$. See Figure \ref{fig:uniqueness_maximizer} for an illustration. Now from Lemma \ref{lemma:area_deccomp}, we deduce that:
 \begin{equation*}
  \text{Area}(\tilde{c}) = \text{Area}(Q) + \text{Area}(c_1) + \text{Area}(c_2) + \text{Area}(c_3) + \text{Area}(c_4).
 \end{equation*}
 By reordering of the edges in $Q$, we obtain the convexified quadrilateral $\bar{Q}$ which, with the same argument used in the proof of Lemma \ref{lemma:polygon_max_area}, is such that $\text{Area}(Q) < \text{Area}(\bar{Q})$. We can then rearrange the different cycles introduced above accordingly, as shown in Figure \ref{fig:uniqueness_maximizer} (right). This is done by translating $c_1,c_2,c_3,c_4$ to match the corresponding edges in $\bar{Q}$. As a result, we obtain a new closed curve $\bar{c}$ which by construction still belongs to $M^{-1}(\{\mu\})$ since it is just obtained by interchanging sections of $\tilde{c}$. Then applying again Lemma \ref{lemma:area_deccomp}:
 \begin{equation*}
  \text{Area}(\bar{c}) = \text{Area}(\bar{Q}) + \text{Area}(c_1) + \text{Area}(c_2) + \text{Area}(c_3) + \text{Area}(c_4) > \text{Area}(\tilde{c}).
 \end{equation*}
This contradicts the fact that $\text{Area}(\tilde{c})= A_{max}(\mu)$.
 \vskip1ex
 5. Finally, the expression of the area $A_{max}=\text{Area}(c)$ of the convex curve $c$ with respect to its length measure is well-known and can be recovered for example from \cite{Letac1983}, although the definition and presentation of length measures is slightly different than in the present paper. For completeness, we provide a direct proof. It suffices to show \eqref{eq:expression_max_area} for a convex polygon and the general case will directly follow from the above approximation argument. Let $\bar{p}$ be a convex polygon with ordered list of edges $\mathcal{E}=(e_1,e_2,\ldots,e_N)$ and vertices $v_1=O, v_2=e_1,\ldots, v_N =e_1+e_2+\ldots+e_{N-1}$. Its length measure is then $\mu = \sum_{i=1}^{N} l_i \delta_{\alpha_i}$ where $l_i = \|e_i\|$ and $\alpha_i = \frac{e_i}{\|e_i\|} \in \Sp^1$. Since the polygon is convex and positively oriented, the successive angles of the edges are non-decreasing so we may assume without loss of generality that $0 \leq \alpha_1 < \alpha_2 < \ldots < \alpha_N < 2\pi$. Therefore the right hand side of \eqref{eq:expression_max_area} is:
 \begin{equation*}
  \frac{1}{4} \int_0^{2\pi} \int_0^{2\pi} \sin |\theta - \alpha| d\mu(\theta) d\mu(\alpha) = \frac{1}{4} \sum_{i=1}^{N} \sum_{j=1}^{N} l_i l_j \sin |\alpha_i - \alpha_j|.  
 \end{equation*}
By splitting the inner sum into two sums from $j=1,\ldots,i$ and $j=i+1,\ldots,N$, after an easy calculation, we obtain:
 \begin{equation*}
  \frac{1}{4} \int_0^{2\pi} \int_0^{2\pi} \sin |\theta - \alpha| d\mu(\theta) d\mu(\alpha) = \frac{1}{2} \sum_{i=1}^{N} \sum_{j=1}^{i-1} l_i l_j \sin |\alpha_i - \alpha_j| = \frac{1}{2} \sum_{i=1}^{N} \sum_{j=1}^{i-1} l_i l_j \sin (\alpha_i - \alpha_j) 
 \end{equation*}
 since $\alpha_i >\alpha_j$ for all $j=1,\ldots,i-1$. This is in turn leads to:
  \begin{equation}
  \label{eq:proof_area_convex_pol1}
  \frac{1}{4} \int_0^{2\pi} \int_0^{2\pi} \sin |\theta - \alpha| d\mu(\theta) d\mu(\alpha) = \frac{1}{2} \sum_{i=1}^{N} \sum_{j=1}^{i-1} \det(e_j,e_i) = \frac{1}{2} \sum_{i=1}^{N} \det(Ov_i, e_i)
 \end{equation}
 since $\sum_{j=1}^{i-1} e_j = \overrightarrow{Ov_i}$. Now the right hand side of \eqref{eq:proof_area_convex_pol1} is exactly the signed area of $\bar{p}$. 
\end{proof}

Another useful expression of the maximal area $A_{max}(\mu)$ is through the Fourier coefficients of the measure $\mu$. Let us define:
\begin{equation*}
 \hat{\mu}(n) = \frac{1}{2\pi} \int_{\Sp^1} e^{-in\theta} d\mu(\theta). 
\end{equation*}
\begin{prop}
 \label{prop:expression_Amax_Fourier}
 For all $\mu \in \mathcal{M}^+_0$, the following holds:
 \begin{equation}
  A_{max}(\mu) = \frac{1}{4\pi} \sum_{\substack{n \in \mathbb{Z} \\ n\neq \pm 1}} \frac{|\hat{\mu}(n)|^2}{1-n^2}
 \end{equation}
\end{prop}
\begin{proof}
 Let $\theta \in [0,2\pi)$ be for now fixed and $f_\theta(\alpha) = \sin |\theta-\alpha|$. After some calculations that we skip for the sake of concision, one can show that the Fourier coefficients of $f_\theta$ are given by:
 \begin{equation*}
  \widehat{f_\theta}(n) = \frac{1}{2\pi} \int_{0}^{2\pi} e^{-in\theta} f_\theta(\alpha) d\alpha = \left \lbrace\begin{aligned}
&\frac{1}{\pi(1-n^2)} e^{-i n\theta} + \frac{1}{2\pi} \left(\frac{e^{-i\theta}}{n-1} - \frac{e^{i\theta}}{n+1} \right) \ \ \text{for } n\neq \pm 1   \\
&-\frac{e^{i\theta}}{4\pi} + \left(\frac{1}{4\pi} +i\frac{\theta-\pi}{2\pi} \right) e^{-i\theta}  \ \ \text{for } n=1 \\
&-\frac{e^{-i\theta}}{4\pi} + \left(\frac{1}{4\pi} +i\frac{\pi-\theta}{2\pi} \right) e^{i\theta}  \ \ \text{for } n=-1
\end{aligned}\right.
 \end{equation*}
 Now, writing $f_\theta(\alpha) = \sum_{n\in \mathbb{Z}} \widehat{f_\theta}(n) e^{in\alpha}$ which converges everywhere on $\Sp^1$ and recalling that $\int_{\Sp^1} e^{i\theta} d\mu(\theta) = \int_{\Sp^1} e^{-i\theta} d\mu(\theta)=0$, we get that:
 \begin{align*}
  A_{max}(\mu) = \frac{1}{4} \int_0^{2\pi} \int_0^{2\pi} \sum_{\substack{n \in \mathbb{Z} \\ n\neq \pm 1}} \frac{1}{\pi(1-n^2)} e^{-i n\theta} e^{in\alpha} d\mu(\theta) d\mu(\alpha)
 \end{align*}
and since the series of functions in the above equation is uniformly converging on $\Sp^1 \times \Sp^1$:
 \begin{align*}
  A_{max}(\mu) &= \frac{1}{4} \sum_{\substack{n \in \mathbb{Z} \\ n\neq \pm 1}} \frac{1}{\pi(1-n^2)} \int_0^{2\pi} \left(\int_0^{2\pi} e^{-i n\theta} d\mu(\theta) \right) e^{in\alpha} d\mu(\alpha) \\
  &= \frac{1}{4} \sum_{\substack{n \in \mathbb{Z} \\ n\neq \pm 1}} \frac{1}{\pi(1-n^2)} (2\pi)^2 \hat{\mu}(n) \overline{\hat{\mu}(n)} \\
  &= \pi \sum_{\substack{n \in \mathbb{Z} \\ n\neq \pm 1}} \frac{|\hat{\mu}(n)|^2}{(1-n^2)}
 \end{align*}
\end{proof}
Note that as a direct corollary of Theorem \ref{thm:max_area_convex} and Proposition \ref{prop:expression_Amax_Fourier}, we recover the standard isoperimetric inequality in its general form, that is:
\begin{corollary}
 \label{cor:isoperimetric_inequality}
 For any curve $c$ in $\tilde{\mathcal{C}}$, we have:
 \begin{equation*}
  \text{Area}(c) \leq \frac{L(c)^2}{4\pi}
 \end{equation*}
 with equality if and only if $c$ is a circle. 
\end{corollary}
\begin{proof}
 By adequate rescaling, we can first restrict the proof to the case where $L(c) = \mu_c(\Sp^1) = 2\pi$. Then, from Theorem \ref{thm:max_area_convex}, we have that $\text{Area}(c) \leq \sup \{A_{max}(\mu) \ | \ \mu \in \mathcal{M}^+_0, \mu(\Sp^1)=2\pi\}$. Now, for any $\mu \in \mathcal{M}^+_0$ such that $\mu(\Sp^1)=2\pi$, we have $\hat{\mu}(0) = \mu(\Sp^1) = 2\pi$ and:
 \begin{equation*}
  A_{max}(\mu) = \frac{1}{4\pi} \sum_{\substack{n \in \mathbb{Z} \\ n\neq \pm 1}} \frac{|\hat{\mu}(n)|^2}{1-n^2} \leq \frac{1}{4\pi} |\hat{\mu}(0)|^2 = \pi.
 \end{equation*}
Therefore, $\text{Area}(c) \leq \pi = \frac{L(c)^2}{4\pi}$ with equality if and only if $A_{max}(\mu) = \pi$ which implies that in the above we have $\hat{\mu}(n)=0$ for all $|n|\neq 1$. Therefore, $\mu$ is the uniform measure on $\Sp^1$ and as $\text{Area}(c) = A_{max}(\mu)$, by the uniqueness of the maximizer in Theorem \ref{thm:max_area_convex}, we obtain that $c$ is the unit circle.    
\end{proof}

\section{A geodesic space structure on convex sets of the plane} 
\label{sec:metric_measure_convex}
The correspondence between convex curves and length measures that is given by Theorem \ref{thm:tangent_length_map} also suggests the idea of comparing convex sets through their associated length measures. In other words, one can transpose the construction of metrics on the set of convex domains to that of building metrics on the measure space $\mathcal{M}^+_0$. What makes this advantageous is that there are already many existing and well-known distances between measures that can be introduced to that end, and several past works \cite{zouaki2003representation,abdallah2014reconstruction} have exploited this idea for various purposes. Yet most of these works are focused on the computation and/or mathematical properties of the distance only. In the field of interest of the authors, namely shape analysis, an often equally important aspect is to define distances (typically Riemannian or sub-Riemannian) which also lead to relevant geodesics on the shape space: such geodesics indeed provide a way of interpolating between two shapes and are crucial to extend many statistical or machine learning tools to analyze shape datasets. In this section, we will therefore be interested in building numerically computable geodesic distances on $\mathcal{M}^+_0$ which will in turn induce metrics and corresponding geodesics on $\tilde{C}_{conv}$. Note that although we discuss this problem for planar curves as it is the focus of this paper, most of what follows can be adapted to larger dimensions by considering the general area measures of convex domains in $\R^n$. We first review several standard metrics on the space of measures of $\Sp^1$ and analyze their potential shortcomings when it comes to comparing planar convex curves. We then propose a new constrained optimal transport distance.  

\subsection{Kernel metrics}
As a dual space, the space of measures on $\Sp^1$ can be endowed with dual metrics based on a choice of norm on a set of test functions. These include classical measure metrics such as the L\'{e}vy-Prokhorov or the bounded Lipschitz distance. Those distances metrizes the weak convergence between measures and have many appealing mathematical properties. However, they are typically challenging to compute or even approximate in practice. An alternative is the class of metrics derived from reproducing kernel Hilbert spaces (RKHS) which have been widely used in shape analysis \cite{Glaunes2004,durrleman2009statistical,Charon2017} and in statistics \cite{gretton2007kernel,li2015generative}. We briefly recap and discuss such metrics in our context. The starting point is a continuous and positive definite kernel $K:\Sp^1 \times \Sp^1 \rightarrow \R$ to which, by Aronszajn theorem \cite{Aronszajn1950}, corresponds a unique RKHS of functions on $\Sp^1$. Let us denote this space by $\mathcal{H}$ and by $\mathcal{H}^*$ its dual. Now taking $\mathcal{H}$ as our space of test functions, one can introduce the norm $\|\cdot\|_{\mathcal{H}^*}$ on $\mathcal{M}^+_0$ defined by:
\begin{equation}
\label{eq:def_RKHS_distance}
 \|\mu\|_{\mathcal{H}^*} = \sup \{(\mu | f) \ | \ f \in \mathcal{H}, \ \|f\|_{\mathcal{H}} =1 \}
\end{equation}
and the corresponding distance $d_{\mathcal{H}^*}(\mu_0,\mu_1) = \|\mu_1 - \mu_0\|_{\mathcal{H}^*}$. In general, \eqref{eq:def_RKHS_distance} only gives a pseudo-distance on $\mathcal{M}^+_0$ but is shown in \cite{sriperumbudur2011universality} Theorem 6 that a necessary and sufficient condition to recover a true distance is for the kernel $K$ to satisfy a property known as $C_0$-universality, which includes several families of well-known kernels such as Gaussian, Cauchy... An important advantage of this RKHS framework is that, thanks to the reproducing kernel property, the distance can be directly expressed based on $K$ as follows: 
\begin{align*}
 &d_{\mathcal{H}^*}(\mu_0,\mu_1)^2 = \\
 &\iint_{\Sp^1\times \Sp^1} K(\theta,\theta') d\mu_0(\theta) d\mu_0(\theta') -2 \iint_{\Sp^1\times \Sp^1} K(\theta,\theta') d\mu_0(\theta) d\mu_1(\theta') + \iint_{\Sp^1\times \Sp^1} K(\theta,\theta') d\mu_1(\theta) d\mu_1(\theta').
\end{align*}
For discrete measures $\mu_0$ and $\mu_1$, the above integrals become double sums and can be all evaluated in closed form once the kernel $K$ is specified which makes these distances easy to compute in practice. However, since all are essentially dual metrics to some Hilbert space, it is easy to see that the resulting metric space is flat, namely the constant speed geodesic between two positive measures $\mu_0$ and $\mu_1$ on $\Sp^1$ is given by $\mu(t) = (1-t) \mu_0 + t \mu_1$ for $t\in [0,1]$. Furthermore, if both $\mu_0$ and $\mu_1$ belong to $\mathcal{M}^+_0$ then so does $\mu(t)$ for all $t\in [0,1]$ since:
\begin{equation*}
 \int_{\Sp^1} e^{i\theta} d\mu(t)(\theta) = (1-t) \int_{\Sp^1} e^{i\theta} d\mu_0(\theta) + t \int_{\Sp^1} e^{i\theta} d\mu_1(\theta) = 0.
\end{equation*}
Therefore $\mathcal{M}^+_0$ is a totally geodesic space for the metric $\|\cdot\|_{\mathcal{H}^*}$.   

\begin{figure}
    \begin{tabular}{cccc}
    \includegraphics[trim = 25mm 25mm 25mm 25mm ,clip,width=4cm]{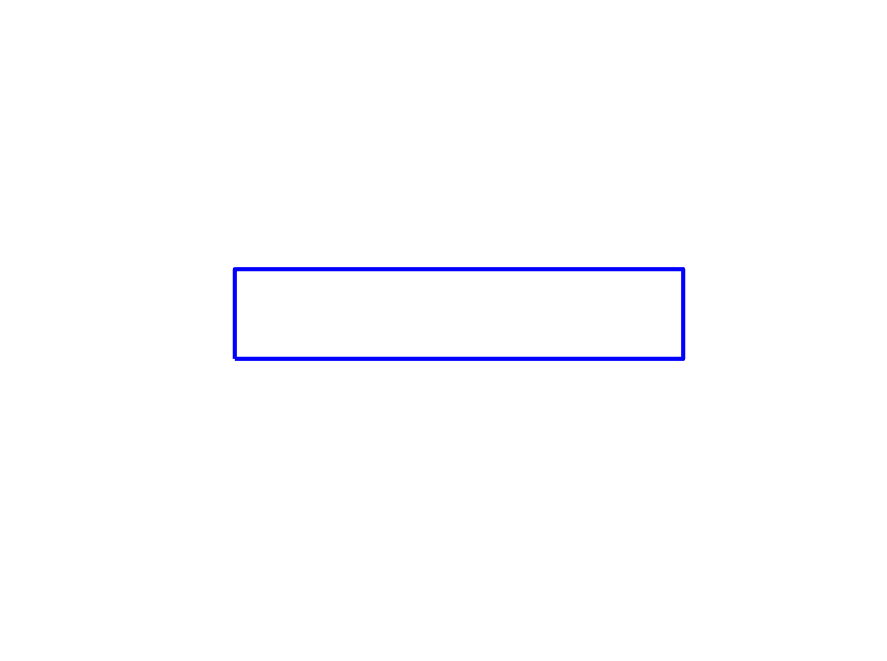} 
    &\includegraphics[trim = 25mm 25mm 25mm 25mm ,clip,width=4cm]{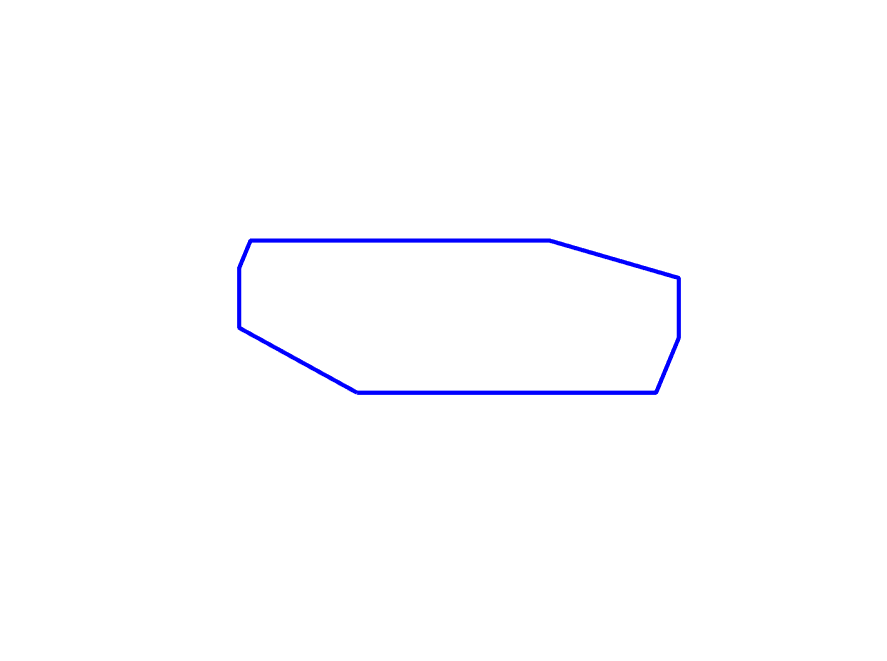}
    &\includegraphics[trim = 25mm 25mm 25mm 25mm ,clip,width=4cm]{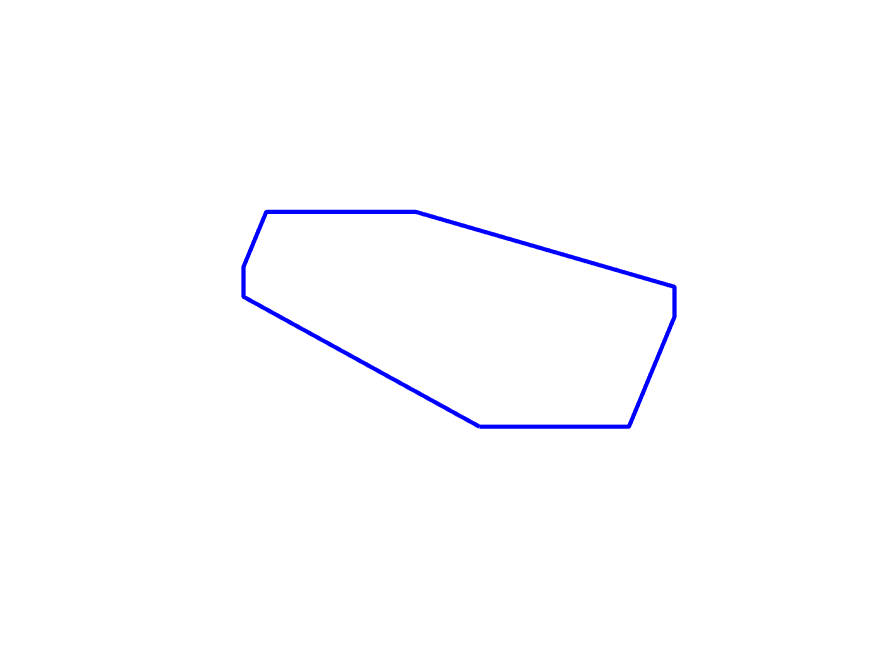} 
    &\includegraphics[trim = 25mm 25mm 25mm 25mm ,clip,width=4cm]{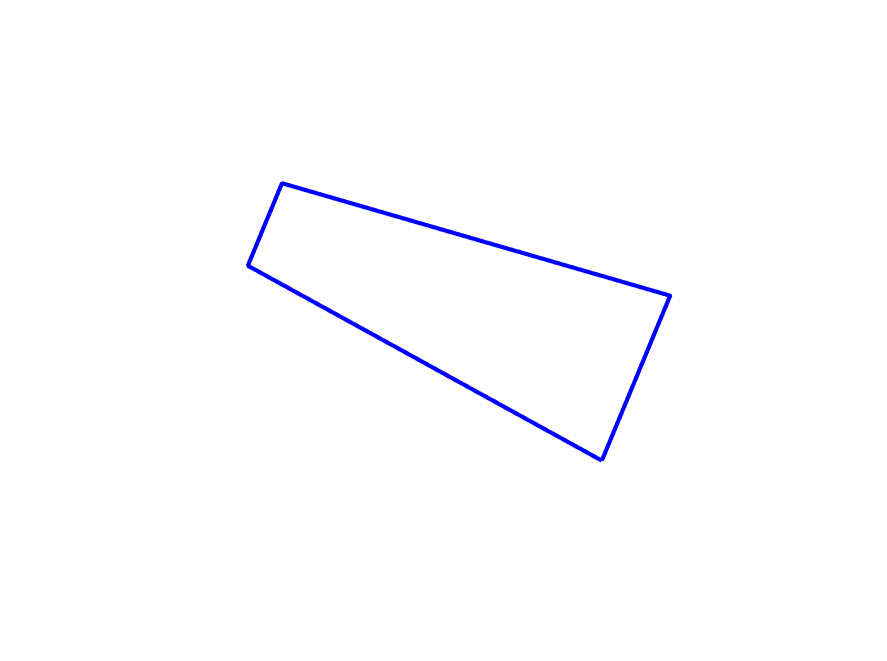}\\
    $t=0$ & $t=1/3$ & $t=2/3$ & $t=1$ \\
    \includegraphics[trim = 5mm 5mm 5mm 5mm ,clip,width=4cm]{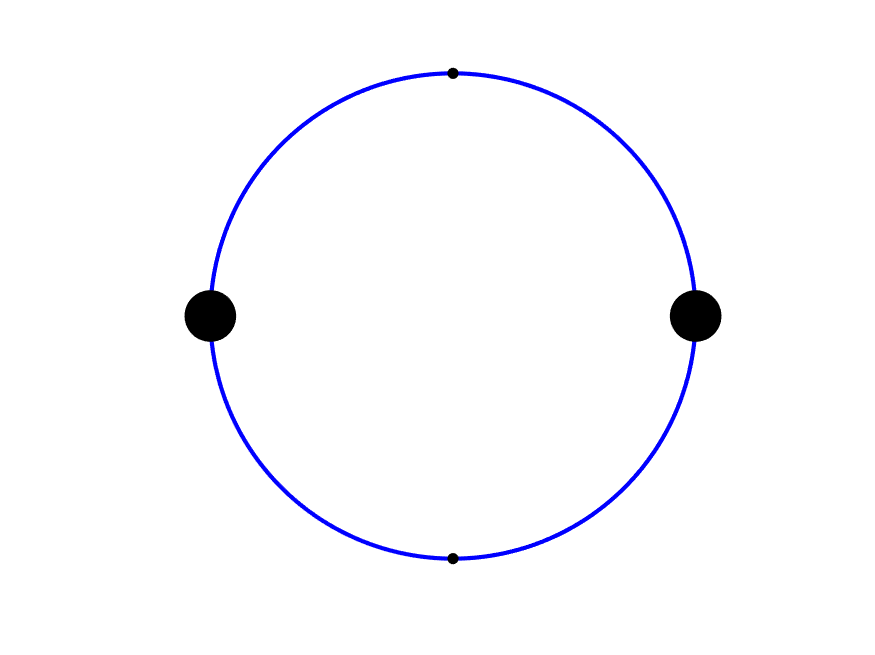} 
    &\includegraphics[trim = 5mm 5mm 5mm 5mm ,clip,width=4cm]{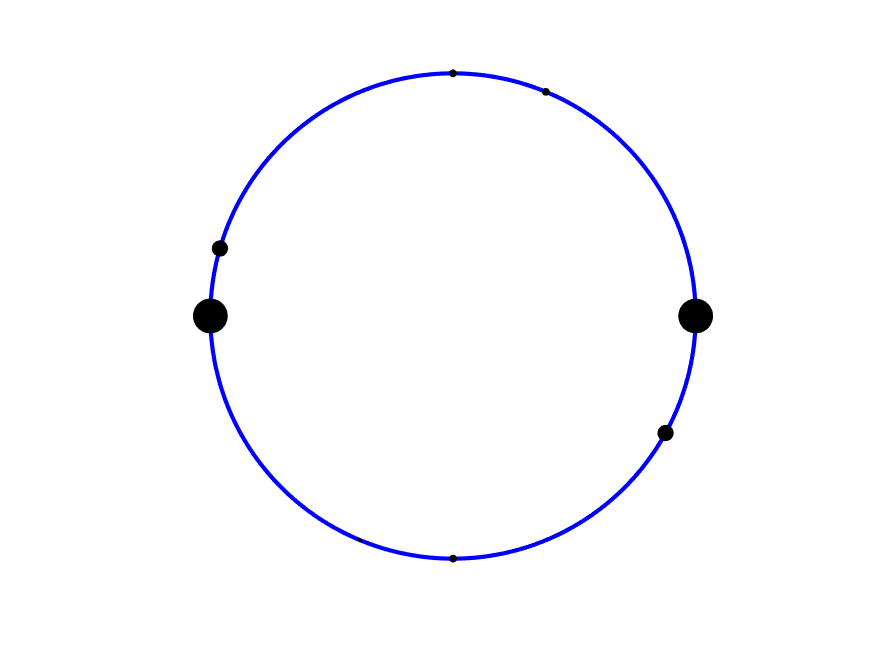}
    &\includegraphics[trim = 5mm 5mm 5mm 5mm ,clip,width=4cm]{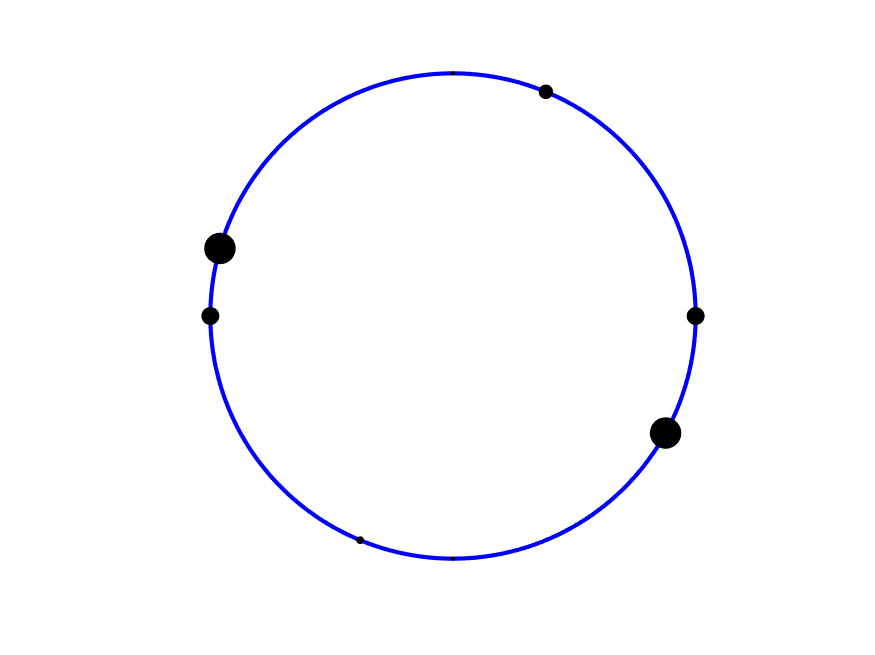} 
    &\includegraphics[trim = 5mm 5mm 5mm 5mm ,clip,width=4cm]{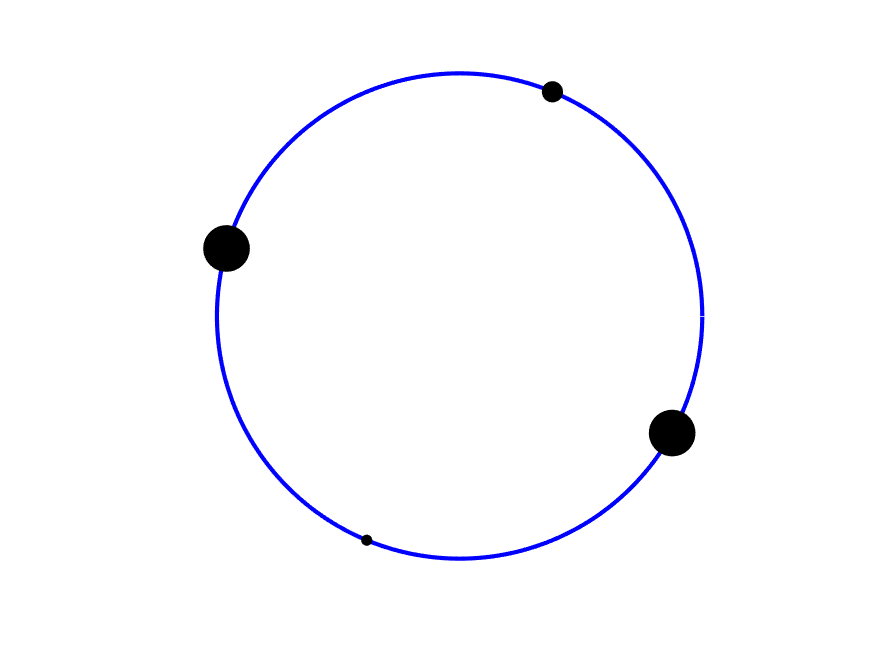}    
    \end{tabular}
    \caption{Geodesic for the kernel metrics. On top, the intermediate convex curves $c(t) \in \tilde{C}_{conv}$ and in the bottom row are shown the associated measures $\mu(t) \in \mathcal{M}^+_0$.}
    \label{fig:geod_kernel_metric}
\end{figure}

Now if we look at the metric on $\tilde{C}_{conv}$ that results from the identification of convex curves with their length measure in $\mathcal{M}^+_0$, we see that, by the isomorphism property of the mapping $M$ given by Corollary \ref{cor:isometry_convex_cones}, the geodesic $c(t) = M^{-1}(\mu(t))$ is simply the Minkowski combination of the two convex curves associated to $\mu_0$ and $\mu_1$. Note that although the distance will depend on the choice of kernel $K$, the geodesics are however independent of $K$. While the Minkowski sum may seem like a natural way to interpolate between two convex sets, it may also not be the most optimal from the point of view of shape comparison. Figure \ref{fig:geod_kernel_metric} shows an example of what a geodesic looks like both in the space of convex sets and in the space of measures. It is for instance clear from this example that the measure geodesic do not involve actual transportation of mass which in this case would be a more natural behavior. This is the shortcoming that the metrics of the following sections will attempt to address.             

\subsection{Wasserstein metrics}
A natural class of distances between probability measures is given by optimal transport and the Wasserstein metrics \cite{villani2008optimal}. Let us consider the usual geodesic distance on $\Sp^1$ defined by $d(\theta,\theta') = \min\{|\theta'-\theta|,2\pi-|\theta' - \theta|\}$. Given two probability measures $\mu_0,\mu_1 \in \mathcal{P}(\Sp^1)$, the 2-Wasserstein distance between them is defined by:
\begin{equation*}
 W_2(\mu_0,\mu_1) = \min \left\{ \iint_{\Sp^1 \times \Sp^1} d(\theta,\theta')^2 d\gamma \ | \ \gamma \in \Pi(\mu_0,\mu_1) \right\}
\end{equation*}
where $\Pi(\mu_0,\mu_1)$ is the set of all transport plans between $\mu_0$ and $\mu_1$ i.e. 
\begin{equation*}
 \Pi(\mu_0,\mu_1) = \{\gamma \in \mathcal{P}(\Sp^1 \times \Sp^1) \ | \ \gamma(A \times \Sp^1) = \mu_0(A), \ \gamma(\Sp^1 \times B) = \mu_1(B), \ \text{for all Borel sets } A,B \subset \Sp^1\}.
\end{equation*}
An optimal transport plan $\gamma$ can be shown to exist and it is known that the Wasserstein metric makes $\mathcal{P}(\Sp^1)$ into a length space with the corresponding geodesic being given by $\mu(t) = (\pi_t)_{\sharp} \gamma$ where for all $\theta,\theta' \in \Sp^1$, $t \mapsto \pi_t(\theta,\theta')$ is a unit constant speed geodesic between $\theta$ and $\theta'$, that is specifically:
\begin{equation*}
 \pi_t(\theta,\theta') = \left\lbrace\begin{aligned}
  &(1-t)\theta + t \theta', \ \text{if } -\pi\leq \theta'-\theta\leq \pi \\
  &(1-t)\theta + t (\theta'-2\pi), \ \text{if } \pi< \theta'-\theta< 2\pi \\
  &(1-t)\theta + t (\theta'+2\pi), \ \text{if } -2\pi< \theta'-\theta< -\pi 
  \end{aligned}\right.
\end{equation*}
Alternatively, this means that for all continuous function $f: \Sp^1 \rightarrow \R$: 
\begin{equation*}
 (\mu(t) | f) = \iint_{\Sp^1 \times \Sp^1} f(\pi_t(\theta,\theta')) d\gamma(\theta,\theta').
\end{equation*}
Unlike in the previous kernel framework, the Wasserstein distance and geodesics cannot be expressed in closed form but there are many well-established and efficient approaches to numerically estimate those, such as combinatorial methods \cite{delon2010fast} or Sinkhorn algorithm \cite{Cuturi2013} in the situation of discrete measures as well as methods based on the dynamical formulation of optimal transport \cite{benamou2000computational,papadakis2014optimal} for densities. 

\begin{figure}
    \begin{tabular}{cccc}
    \includegraphics[trim = 25mm 25mm 25mm 25mm ,clip,width=4cm]{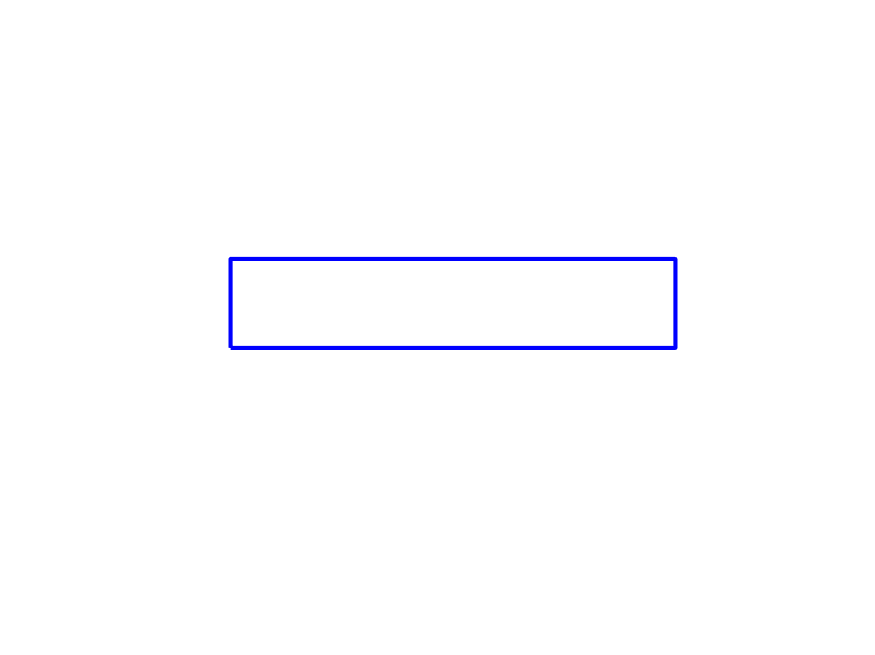} 
    &\includegraphics[trim = 25mm 25mm 25mm 25mm ,clip,width=4cm]{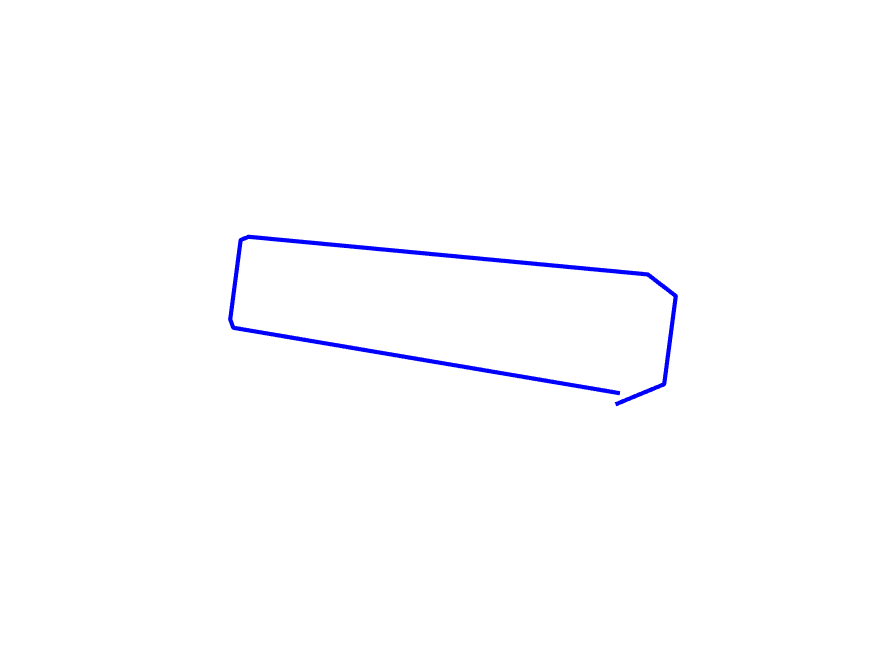}
    &\includegraphics[trim = 25mm 25mm 25mm 25mm ,clip,width=4cm]{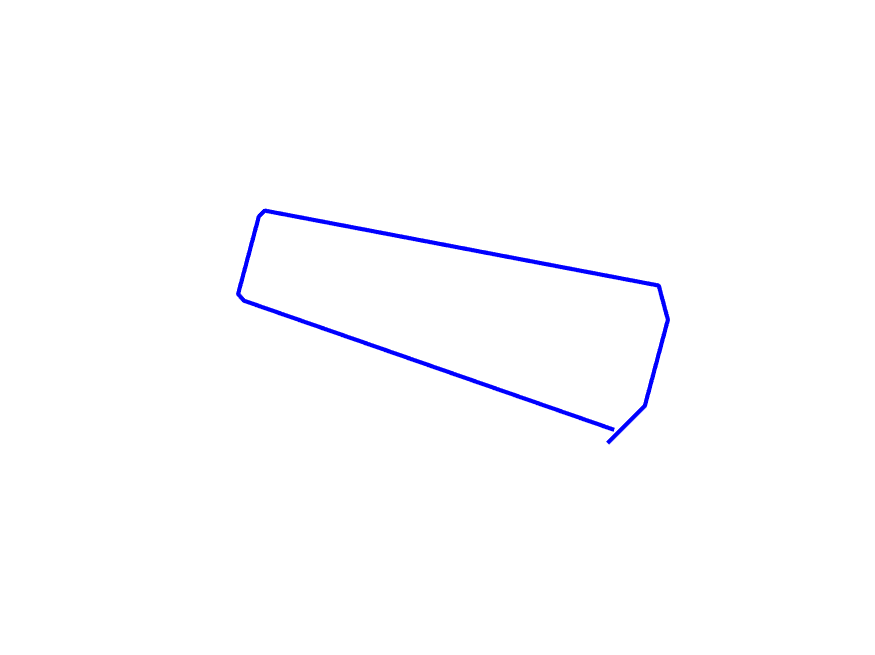} 
    &\includegraphics[trim = 25mm 25mm 25mm 25mm ,clip,width=4cm]{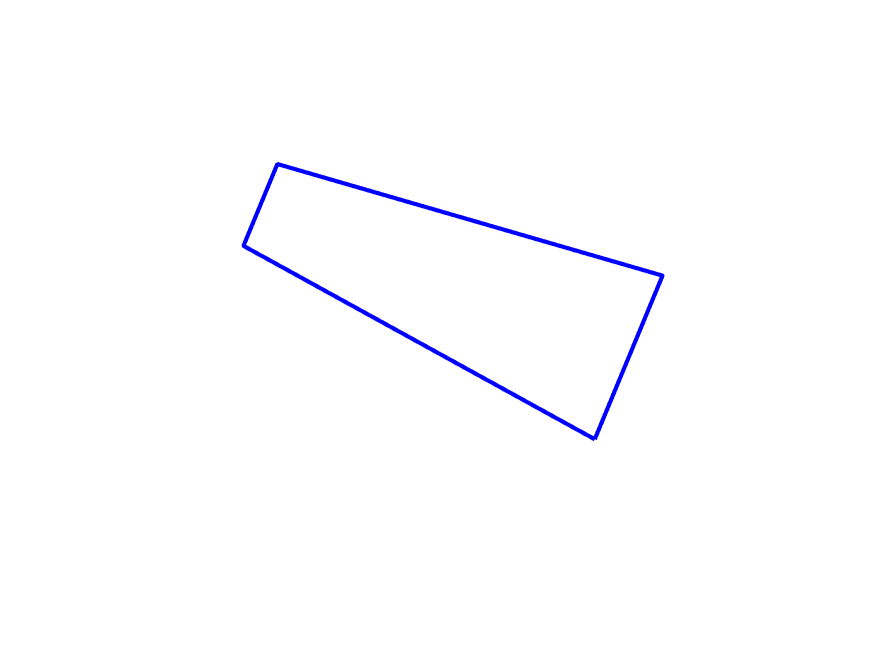}\\
    $t=0$ & $t=1/3$ & $t=2/3$ & $t=1$ \\
    \includegraphics[trim = 5mm 5mm 5mm 5mm ,clip,width=4cm]{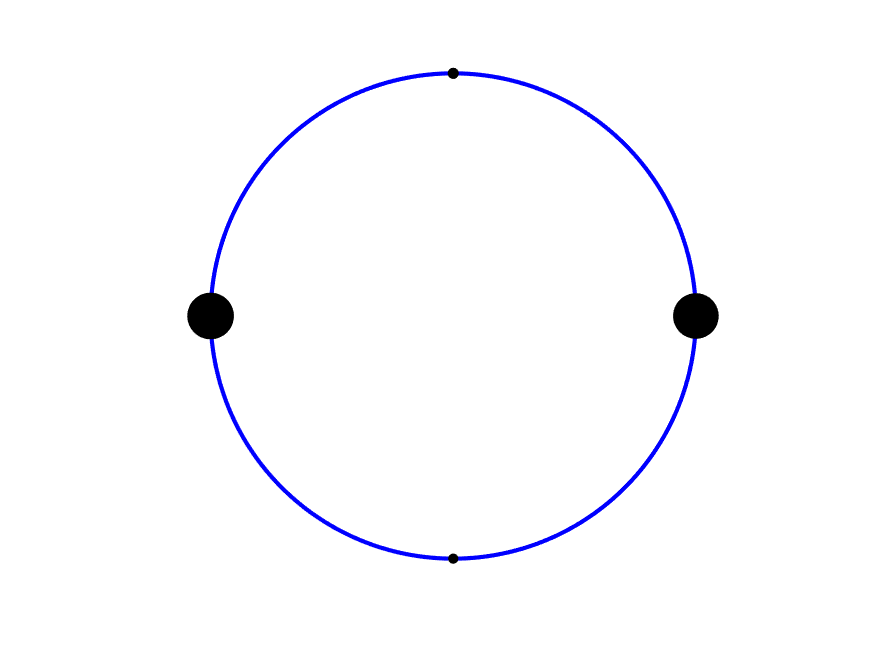} 
    &\includegraphics[trim = 5mm 5mm 5mm 5mm ,clip,width=4cm]{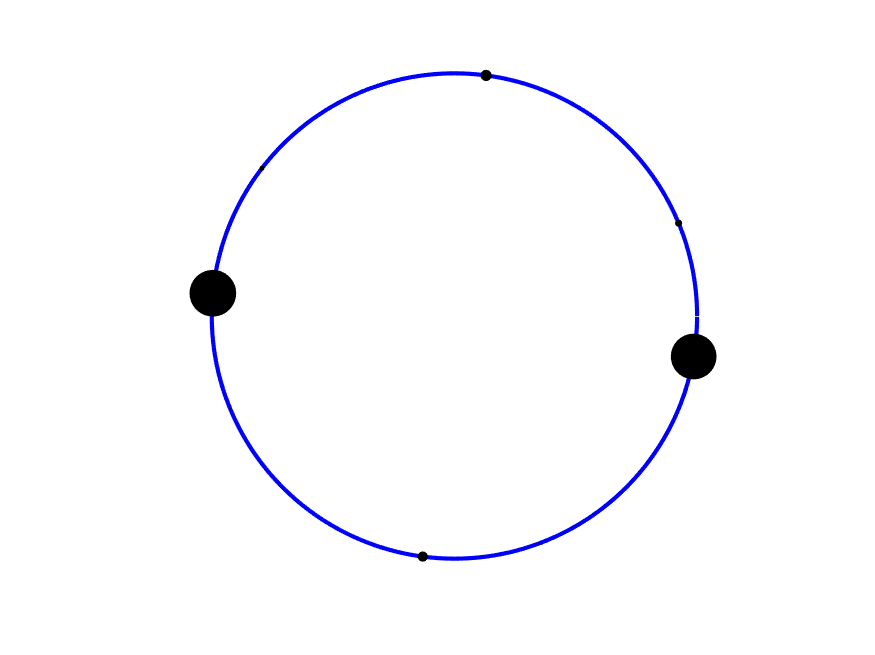}
    &\includegraphics[trim = 5mm 5mm 5mm 5mm ,clip,width=4cm]{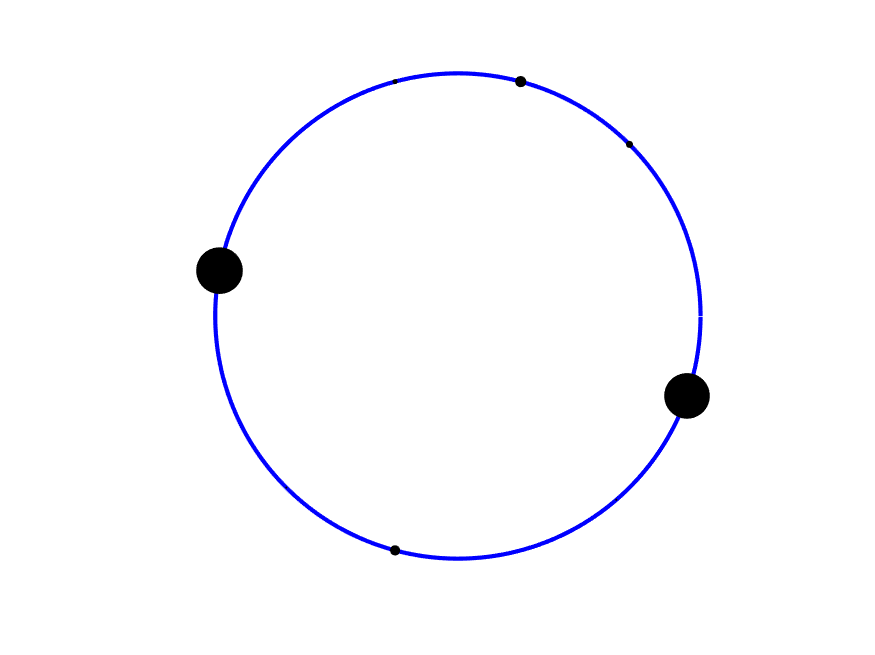} 
    &\includegraphics[trim = 5mm 5mm 5mm 5mm ,clip,width=4cm]{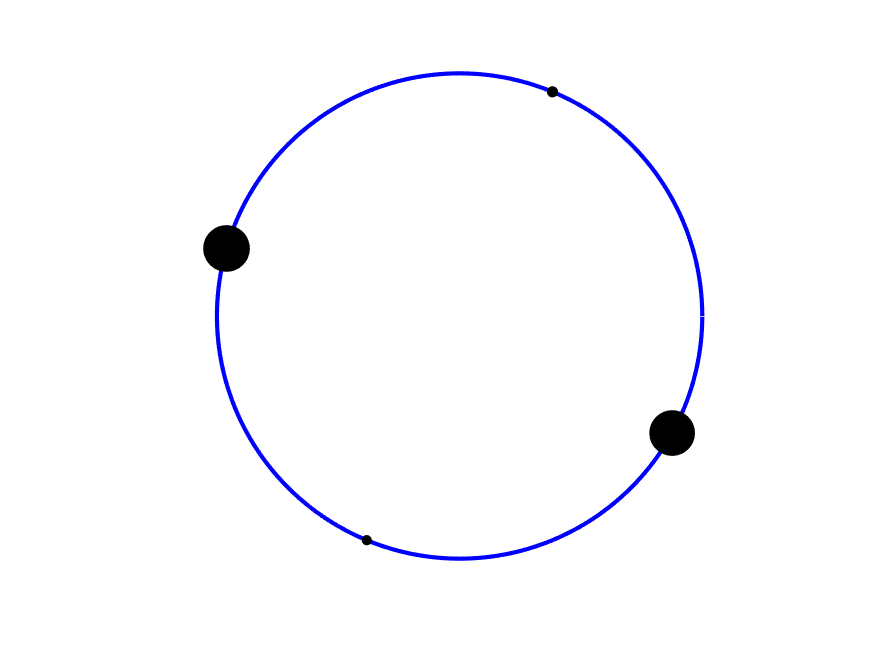}    
    \end{tabular}
    \caption{Geodesic for the Wasserstein metric computed with the Sinkhorn algorithm. On top, the intermediate convex curves $c(t) \in \tilde{C}_{conv}$ and in the bottom row are shown the associated measures $\mu(t)$. Note that the intermediate measures $\mu(t)$ do not belong to $\mathcal{M}^+_0$ which results in an opening of the reconstructed curves.}
    \label{fig:geod_OT_metric}
\end{figure}

When it comes to the comparison of convex shapes, the use of Wasserstein metrics on the length or area measures was proposed for instance in \cite{zouaki2003representation}. Note that since $W_2$ is defined between probability measures, this amounts in restricting to convex curves of length $1$ i.e. comparing curves modulo rescaling which is often natural in shape analysis. However, an important downside of the Wasserstein metric for this problem is that the subspace of probability measures in $\mathcal{M}^+_0$ is not totally geodesic for that metric, namely the above path of measures $\mu(t)$ connecting $\mu_0$ to $\mu_1$ does not generally stay in $\mathcal{M}^+_0$ when $\mu_0,\mu_1 \in \mathcal{P}(\Sp^1) \cap \mathcal{M}^+_0$. This means that the intermediate measures $\mu(t)$ cannot be canonically associated to a convex curve (or even to a closed curve as a matter of fact). We illustrate this in Figure \ref{fig:geod_OT_metric} that shows the Wasserstein geodesic between the same two discrete measures as in Figure \ref{fig:geod_kernel_metric} and the curves obtained from the reconstruction procedure of Remark \ref{rem:reconstruction_convex_curve}. This issue could be addressed a posteriori: for example the authors in \cite{zouaki2003representation} propose to consider a slightly modified path defined from the optimal transport plan. Specifically, they introduce $\tilde{\mu}(t)$ defined by:
\begin{equation*}
 (\tilde{\mu}(t) | f) = \iint_{\Sp^1 \times \Sp^1} f\left(\frac{(1-t)e^{i\theta} + t e^{i\theta'}}{|(1-t)e^{i\theta} + t e^{i\theta'}|} \right) |(1-t)e^{i\theta} + t e^{i\theta'}| d\gamma(\theta,\theta')
\end{equation*}
for which it is straightforward to check that $\tilde{\mu}(t) \in \mathcal{M}^+_0$ for all $t\in[0,1]$. A clear remaining downside however is that, despite relying on the optimal transport between $\mu_0$ and $\mu_1$, the paths of measures and corresponding convex curves are not actual geodesics associated to a metric. 

\subsection{A constrained Wasserstein distance on $\mathcal{M}^+_0$}
\label{ssec:constrained_Wasserstein}
In this section we shall instead modify the formulation of the Wasserstein metric so as to directly enforce the constraint defining $\mathcal{M}^+_0$. To do so, we will need to shift our focus to the so called dynamical formulation of optimal transport which was first introduced by the works of Benamou and Brenier in \cite{benamou2000computational} and derive a primal-dual scheme adapted from the work of \cite{papadakis2014optimal} to tackle the corresponding optimization problem. 

\subsubsection{Mathematical formulation}
\label{ssec:constrained_Wasserstein_theory}
Formally, the above Wasserstein metric can be obtained by minimizing $\int_0^1 \int_{\Sp^1} |v_t|^2 d\mu(t) dt$
over all vector fields $v_t(\cdot)$ and path of measures $\mu(t) \in \mathcal{M}$ that satisfy the boundary constraints $\mu(0)=\mu_0$, $\mu(1)=\mu_1$ and the continuity equation on $\Sp^1$ $\partial_t \mu(t) + \partial_\theta(v_t \mu(t)) =0$. To make this formulation more rigorous, one has to introduce the following definitions. 
\begin{definition}
 Let $(\rho,m)$ be a couple of measures on $[0,1] \times \Sp^1$ which we will write $\rho,m \in \mathcal{M}([0,1]\times \Sp^1)$. We say that $(\rho,m)$ satisfy the continuity equation $\partial_t \rho + \partial_\theta m = 0$ with boundary conditions $\rho(0) = \mu_0$ and $\rho(1)=\mu_1$ in the distribution sense if:
 \begin{equation}
 \label{eq:continuity_weak}
  \int_{[0,1]\times \Sp^1} \partial_t \phi \ d \rho + \int_{[0,1]\times \Sp^1} \partial_\theta \phi \ d m = \int_{\Sp^1} \phi(1,\cdot) d\mu_1 - \int_{\Sp^1} \phi(0,\cdot) d\mu_0. 
 \end{equation}
 for all $\phi \in C^1([0,1]\times \Sp^1)$.
\end{definition}
We recall the following usual property of the continuity equation
\begin{property}
\label{prop:disintegration_Lebesgue}
 If $(\rho,m)$ satisfies the above continuity equation with boundary conditions $\mu_0$ and $\mu_1$, then the time marginal of $\rho$ is the Lebesgue measure on $[0,1]$. In other words the measure $\rho$ disintegrates as $\rho= \rho_t \otimes dt$.    
\end{property}
\begin{proof}
By applying the disintegration theorem on measures of $[0,1] \times \Sp^1$ (cf Theorem 2.28 in \cite{Ambrosio2000}), we have the existence of finite measures $\rho_t$ on $\Sp^1$ for all $t\in [0,1]$ such that we can write $\rho = \rho_t \otimes \nu$ with $\nu$ the measure on $[0,1]$ given by $\nu(A) = \rho(A \times \Sp^1)$. Taking $\phi(t,\theta) = \psi(t)$ with $\psi \in C^1([0,1])$, we get from \eqref{eq:continuity_weak}:
 \begin{equation*}
  \int_{[0,1]} \psi'(t) d\nu(t) = \psi(1) - \psi(0). 
 \end{equation*}
 Since this holds for any $\psi \in C^1([0,1])$, we obtain that $\nu$ is the Lebesgue measure on $[0,1]$.  
\end{proof}
Next we define the set $B=\{(a,b) \in \R^2 \ | \ a+\frac{1}{2} |b|^2 \leq 0 \}$ and we will denote $\iota_B$ its convex indicator function: $\iota_B(a,b) =0$ for $(a,b) \in B$ and  $\iota_B(a,b) =+\infty$ otherwise. We also define $f: \R \times \R \rightarrow \R_+$ by:
\begin{equation*}
 f(d,m) = 
 \left\lbrace\begin{aligned}
  &\frac{|m|^2}{2d} \ \ \text{if } d>0 \\
  &0  \ \ \text{if } (d,m)=(0,0) \\
  &+\infty \ \ \text{otherwise}
  \end{aligned}\right.
\end{equation*}
Note that $f$ is a 1-homogeneous function and is the convex conjugate of $\iota_B$. Then it has been shown that given two measures $\mu_0$ and $\mu_1$, their Wasserstein distance is equal to:
\begin{equation}
\label{eq:Wasserstein_dynamic}
 W_2(\mu_0,\mu_1) = \inf_{\nu = (\rho,m)} \left\{ \int_{[0,1] \times \Sp^1} f\left(\frac{d\rho}{d\lambda}\right) d\lambda \ | \ \text{subj. to } \rho_0 = \mu_0, \ \rho_1 =\mu_1, \ \partial_t \rho + \partial_\theta m = 0\right\}.
\end{equation}
where $\lambda$ is any measure on $[0,1] \times \Sp^1$ such that $|\nu| \ll \lambda$ and $\frac{d\nu}{d\lambda}$ denotes the Radon-Nykodym derivative (the choice of such $\lambda$ does not affect the value in \eqref{eq:Wasserstein_dynamic} due to the homogeneity of $f$). Furthermore, the existence of optimal measures $\rho$ and $m$ can be proved and the path of measures $t \mapsto \rho_t$ given by Property \ref{prop:disintegration_Lebesgue} is a geodesic between $\mu_0$ and $\mu_1$.  

We now modify the above expression so as to enforce that the path $\rho_t$ remains in the subspace $\mathcal{M}^+_0$. 
\begin{definition}
 Let $\mu_0$ and $\mu_1$ be two probability measures in $\mathcal{M}^+_0$. With the same conventions as above, we define the constrained Wasserstein distance on $\mathcal{M}^+_0$ by:
 \begin{equation}
  \label{eq:mod_Wasserstein_def}
  \overline{W}_2(\mu_0,\mu_1) = \inf_{\nu = (\rho,m)} \left\{ \int_{[0,1] \times \Sp^1} f\left(\frac{d\nu}{d\lambda}\right) d\lambda \ | \ \text{subj. to } \left\{
    \begin{array}{lll}
         \partial_t\rho + \partial_\theta m= 0 \\
          \rho_0=\mu_0, \ \rho_1=\mu_1 \\
          \int_{\Sp^1} e^{i\theta} d\rho_t(\theta)=0, \ \text{for a.e } t\in[0,1]
    \end{array}
\right.\right\}.
 \end{equation}
\end{definition}
As we show next, $\overline{W}_2$ defines a distance and we also recover the existence of geodesics.
\begin{theorem}
\label{thm:existence_geodesics_constr_Wass}
For any two probability measures $\mu_0$ and $\mu_1$ in $\mathcal{M}^+_0$ such that $\overline{W}_2(\mu_0,\mu_1)<+\infty$, there exist $(\rho,m)$ achieving the minimum in \eqref{eq:mod_Wasserstein_def}. Moreover, $\overline{W}_2$ defines a distance on the space $\mathcal{P}(\Sp^1) \cap \mathcal{M}^+_0$.
\end{theorem}
\begin{proof}
The proof mainly requires adapting similar arguments developed for other variations of the optimal transport problem \cite{cardaliaguet2013geodesics,Chizat2018}. For any $\phi \in\mathcal{C}^1([0,1]\times\mathbb{S}^1)$ and $\psi\in\mathcal{C}([0,1])$, let us define
\begin{align*}
    J(\phi, \psi_1,\psi_2) &= \int_0^1\int_{\mathbb{S}^1}\iota_B(\partial_t\phi+\cos(\theta)\psi_1(t)+\sin(\theta)\psi_2(t),\partial_\theta\phi) d\theta dt + \int_{\mathbb{S}^1}\phi(0,\cdot)d\rho_0-\int_{\mathbb{S}^1}\phi(1,\cdot)d\rho_1 \\
    &= F \circ A(\phi, \psi_1,\psi_2) +G(\phi,\psi_1,\psi_2)
\end{align*}
where 
\begin{align*}
&F(\alpha, \beta)= \int_0^1\int_{\mathbb{S}^1}\iota_B(\alpha(t,\theta),\beta(t,\theta))d\theta dt \\
&A(\phi,\psi_1,\psi_2) = (\partial_t\phi+\cos(\theta)\psi_1(t)+\sin(\theta)\psi_2(t),\partial_\theta \phi) \\
&G(\phi,\psi_1,\psi_2)= \int_0^1\int_{\mathbb{S}^1}\phi(0,\cdot)d\rho_0- \int_0^1\int_{\mathbb{S}^1}\phi(1,\cdot)d\rho_1
\end{align*}
One can easily check that $F$ and $G$ are proper convex functions and are lower semi-continuous. 
Furthermore, taking for instance $\phi(t,\theta) = at$ with $a<0$ and $\psi(t)=0$, we see that $F$ is continuous at $A(\phi,\psi) = (a,0)$ since $(a,0)$ is in the interior of $B$. We can thus apply the Fenchel-Rockafellar duality theorem, which leads to 
\begin{equation}
\label{eq:Fenchel_Rockafellar_duality}
\underset{\phi\in\mathcal{C}^1([0,1]\times\mathbb{S}^1),\psi_1,\psi_2\in\mathcal{C}([0,1])}{\inf} J(\phi,\psi) = \underset{\nu \in \mathcal{M}([0,1]\times \Sp^1)\times\mathcal{M}([0,1]\times \Sp^1)}{\sup} \left\{-F^{*}(\nu)-G^*(-A^*(\nu)) \right\}.
\end{equation}
We now compute for $\nu=(\rho,m) \in \mathcal{M}([0,1]\times \Sp^1) \times \mathcal{M}([0,1]\times \Sp^1)$:
\begin{align*}
&G^*(-A^*\nu) = \underset{\phi, \psi_1,\psi_2}{\sup}\left\{\langle \nu , -A(\phi,\psi_1,\psi_2)\rangle + \int_{\mathbb{S}^1}\phi(1,\cdot)d\rho_1-\int_{\mathbb{S}^1}\phi(0,\cdot)d\rho_0 \right\} \\
&= \underset{\phi, \psi_1,\psi_2}{\sup}\left\{-\langle\nu,(\partial_t\phi,\partial_\theta \phi)\rangle -  \langle \nu ,(\cos(\theta)\psi_1(t)+\sin(\theta)\psi_2(t),0)\rangle + \int_{\mathbb{S}^1}\phi(1,\cdot)d\rho_1-\int_{\mathbb{S}^1}\phi(0,\cdot)d\rho_0 \right\} \\
&= \underset{\phi, \psi_1,\psi_2}{\sup}\left\{-\langle\nu,(\partial_t\phi,\partial_\theta \phi)\rangle+ \int_{\mathbb{S}^1}\phi(1,\cdot)d\rho_1-\int_{\mathbb{S}^1}\phi(0,\cdot)d\rho_0
 -  \langle \nu ,(\cos(\theta)\psi_1(t)+\sin(\theta)\psi_2(t),0)\rangle
 \right\}  \\
 &= \underset{\phi}{\sup}\left\{-\langle\nu,(\partial_t\phi,\nabla\phi)\rangle+ \int_{\mathbb{S}^1}\phi(1,\cdot)d\rho_1 -\int_{\mathbb{S}^1}\phi(0,\cdot)d\rho_0 \right\} + \underset{\psi_1,\psi_2}{\sup}\left\{-\langle \nu ,(\cos(\theta)\psi_1(t)+\sin(\theta)\psi_2(t),0)\rangle 
 \right\} \\
 &= \underset{\phi}{\sup}\left\{-\langle \nu, (\partial_t\phi,\partial_\theta \phi)\rangle + \int_{\mathbb{S}^1}\phi(1,\cdot)d\rho_1 -\int_{\mathbb{S}^1}\phi(0,\cdot)d\rho_0 \right\} + \underset{\psi_1}{\sup}\left\{-\int_{[0,1]\times\mathbb{S}^1} \cos(\theta) \psi_1(t) d\rho \right\} \\
 &\phantom{\underset{\phi}{\sup}\left\{-\langle \nu, (\partial_t\phi,\partial_\theta \phi)\rangle + \int_{\mathbb{S}^1}\phi(1,\cdot)d\rho_1 -\int_{\mathbb{S}^1}\phi(0,\cdot)d\rho_0 \right\}+}+ \underset{\psi_2}{\sup}\left\{-\int_{[0,1]\times\mathbb{S}^1} \sin(\theta) \psi_2(t) d\rho
 \right\} 
\end{align*}
and we find 
\begin{equation*}
\underset{\phi}{\sup}\left\{-\langle\nu,(\partial_t\phi,\partial_\theta \phi)\rangle+ \int_{\mathbb{S}^1}\phi(1,\cdot)d\rho_1 -\int_{\mathbb{S}^1}\phi(0,\cdot)d\rho_0 \right\} = \left\{
    \begin{array}{ll}
         0 &\text{if } \partial_t\rho + \partial_\theta m= 0, \rho(0,\cdot)=\rho_0, \rho(1,\cdot)=\rho_1 \\
         +\infty &\text{otherwise}
    \end{array}
\right.
\end{equation*}
from the definition of the weak continuity equation \eqref{eq:continuity_weak}. It follows that the supremum on the right hand side of \eqref{eq:Fenchel_Rockafellar_duality} can be restricted to measures $\rho$ and $m$ satisfying the continuity equation. Therefore, by Property \ref{prop:disintegration_Lebesgue}, we obtain that 
\begin{align*}
\underset{\psi_1}{\sup}\left\{-\int_{[0,1]\times\mathbb{S}^1}\cos(\theta)\psi_1(t)d\rho
 \right\} &= \underset{\psi_1}{\sup}\left\{-\int_0^1 \left(\int_{\Sp^1} \cos(\theta) d\rho_t(\theta) \right) \psi_1(t) dt
 \right\} \\
 &=\left\{
    \begin{array}{ll}
         0 &\mbox{ if for a.e } t\in[0,1], \ \int_{\mathbb{S}^1}\cos(\theta)d\rho_t = 0\\
         +\infty &\mbox{ otherwise}
    \end{array}
\right.
\end{align*}
and similarly 
\begin{align*}
\underset{\psi_2}{\sup}\left\{-\int_{[0,1]\times\mathbb{S}^1}\sin(\theta)\psi_2(t)d\rho
 \right\} =\left\{
    \begin{array}{ll}
         0 &\mbox{ if for a.e } t\in[0,1], \ \int_{\mathbb{S}^1}\sin(\theta)d\rho_t = 0\\
         +\infty &\mbox{ otherwise}
    \end{array}
\right.
\end{align*}
from which we deduce that 
\begin{equation}
\label{eq:dual_problem}
\underset{\nu }{\sup} \left\{-F^{*}(\nu)-G^*(-A^*(\nu)) \right\} = -\underset{\nu}{\inf} \left\{F^{*}(\nu) \ | \ \text{subj. to } \left\{
    \begin{array}{lll}
         \partial_t\rho + \partial_\theta m= 0 \\
          \rho_0=\mu_0, \ \rho_1=\mu_1 \\
          \int_{\Sp^1} e^{i\theta} d\rho_t(\theta)=0, \ \text{for a.e } t\in[0,1]
    \end{array} \right.
    \right\}
\end{equation}
Finally the convex conjugate of $F$ is computed for instance in \cite{Chizat2018} based on Theorem 5 of \cite{rockafellar1971integrals} and can be shown to be:
\begin{equation*}
 F^{*}(\nu) = \int_{[0,1]\times \Sp^1} f\left(\frac{d\nu}{d\lambda}\right) d\lambda.
\end{equation*}
This allows to conclude that \eqref{eq:dual_problem} corresponds to $-\overline{W}_2(\mu_0,\mu_1)$ and, since we assume that $\overline{W}_2(\mu_0,\mu_1)<+\infty$, we obtain by Fenchel-Rockafellar theorem the existence of a minimizer $(\rho,m)$ to \eqref{eq:mod_Wasserstein_def}.

In addition, we have by construction that $\overline{W}_2(\mu_0,\mu_1) \geq W_2(\mu_0,\mu_1)$ and therefore $\overline{W}_2(\mu_0,\mu_1) = 0$ implies that $\mu_0 = \mu_1$. The symmetry of $\overline{W}_2$ is also immediate from the fact that for any measure path $\nu=(\rho,m)$ satisfying the conditions in \eqref{eq:mod_Wasserstein_def}, one can consider $\tilde{\nu} = (\tilde{\rho}_t = \rho_{1-t},\tilde{m}_t = m_{1-t})$ the time-reversed paths which satisfy $\tilde{\rho}_0 = \mu_1$, $\tilde{\rho}_1 = \mu_0$, $\partial_t\tilde{\rho} + \partial_\theta \tilde{m}= 0$ and $\int_{\Sp^1} e^{i\theta} d\tilde{\rho}_t(\theta)=0$ while:
\begin{equation*}
 \int_{[0,1] \times \Sp^1} f\left(\frac{d\tilde{\nu}}{d\lambda}\right) d\lambda = \int_{[0,1] \times \Sp^1} f\left(\frac{d\nu}{d\lambda}\right) d\lambda.
\end{equation*}
Finally, the triangular inequality can be shown in similar way as for the standard Wasserstein metric i.e. by concatenation. Let $\mu_0,\mu_1,\mu_2 \in \mathcal{P}(\Sp^1) \cap \mathcal{M}^+_0$ and assume that $\overline{W}_2(\mu_0,\mu_1) <+\infty$ and $\overline{W}_2(\mu_1,\mu_2) <+\infty$ (otherwise the triangular inequality is trivially satisfied). From the above, we know that there exist $\nu=(\rho,m)$ and $\nu'=(\rho',m')$ that achieve the minimum in the distances $\overline{W}_2(\mu_0,\mu_1)$ and $\overline{W}_2(\mu_1,\mu_2)$ respectively. Letting $\tau^0:(t,\theta) \mapsto (t/2,\theta)$ and $\tau^1:(t,\theta) \mapsto ((t+1)/2,\theta)$, we define the following concatenated path $\bar{\nu}=(\bar{\rho},\bar{m})$:
\begin{equation*}
 \bar{\nu} = \left\{
    \begin{aligned}
         &((\tau^0)_\sharp \rho,2 (\tau^0)_\sharp m) \ \ \text{for } 0\leq t < \frac{1}{2} \\
         &((\tau^1)_\sharp \rho',2 (\tau^1)_\sharp m') \ \ \text{for } \frac{1}{2}\leq t \leq 1
    \end{aligned}
\right.
\end{equation*}
for which one can easily verify that $\bar{\rho}_0 = \mu_0$, $\bar{\rho}_1 = \mu_2$, $\partial_t\bar{\rho} + \partial_\theta \bar{m}= 0$ and $\int_{\Sp^1} e^{i\theta} d\bar{\rho}_t(\theta)=0$ for almost all $t \in [0,1]$. It follows that:
\begin{align*}
 \overline{W}_2(\mu_0,\mu_2) \leq  \int_{[0,1] \times \Sp^1} f\left(\frac{d\bar{\nu}}{d\lambda}\right) d\lambda &= \int_{[0,1] \times \Sp^1} f\left(\frac{d\nu}{d\lambda}\right) d\lambda + \int_{[0,1] \times \Sp^1} f\left(\frac{d\nu'}{d\lambda}\right) d\lambda \\
 &= \overline{W}_2(\mu_0,\mu_1) + \overline{W}_2(\mu_1,\mu_2).
\end{align*}
\end{proof}

In general, one cannot guarantee that the distance $\overline{W}_2$ is finite between any two measures. Nevertheless, it holds in particular when:
\begin{prop}
\label{prop:finiteness_constrained_Wasserstein}
 Let $\mu_0 = \rho_0(\theta) d\theta$ and $\mu_1 = \rho_1(\theta) d\theta$ be two measures in $\mathcal{P}(\Sp^1) \cap \mathcal{M}^+_0$ with densities with respect to the Lebesgue measure on $\Sp^1$ with $\rho_0,\rho_1 \in L^1(\Sp^1)$, and assume that there exists $c>0$ such that $\rho_0(\theta)\geq c$, $\rho_1(\theta)\geq c$ for almost all $\theta \in \Sp^1$. Then $\overline{W}_2(\mu_0,\mu_1) <+\infty$. 
\end{prop}
\begin{proof}
 With the above assumptions, let us consider the linear interpolation path $\rho_t(\theta) = (1-t) \rho_0(\theta) + t \rho_1(\theta)$ and define $m_t = H(\theta) d\theta$ where $H(\theta) = \int_0^{\theta} (\rho_1(\alpha)-\rho_0(\alpha)) d\alpha$. Then one easily checks that we have $\partial_t \rho + \partial_\theta m = 0$ in the sense of distributions. Furthermore, 
 \begin{equation*}
  \int_{\Sp^1} e^{i\theta} \rho_t(\theta) d \theta = (1-t) \int_{\Sp^1} e^{i\theta} \rho_0(\theta) d \theta + t \int_{\Sp^1} e^{i\theta} \rho_1(\theta) d \theta = 0.
 \end{equation*}
Then, taking $\lambda$ the Lebesgue measure on $[0,1] \times \Sp^1$ and since $\rho_t(\theta) \geq c$ for almost all $\theta$, we find:
\begin{equation*}
 \int_{0}^1 \int_{\Sp^1} f\left(\rho_t(\theta),m_t(\theta)\right) d\theta dt = \int_{0}^1 \int_{\Sp^1} \frac{H(\theta)^2}{2\rho_t(\theta)}  d\theta dt \leq \frac{1}{2c} \int_{\Sp^1} H(\theta)^2 d\theta <+\infty 
\end{equation*}
and thus $\overline{W}_2(\mu_0,\mu_1)$ is also finite. 
\end{proof}
Note that as a special case of this proposition, we deduce that the distance is finite between any two continuous strictly positive densities on $\Sp^1$.

Now, thanks to the bijection induced by $M$ between the set of convex curves of length one modulo translation which we will denote $\tilde{C}_{conv}^1$ and the set of probability measures of $\mathcal{M}^+_0$, we can equip $\tilde{C}_{conv}^1$ with the structure of a geodesic space by setting:
\begin{equation}
\label{eq:def_newdistance_convex}
 d(c_0,c_1) \doteq \overline{W}_2(M(c_0),M(c_1))
\end{equation}
for any $c_0,c_1 \in \tilde{C}_{conv}^1$. A geodesic between the two curves is then given by $c(t) = M^{-1}(\rho(t))$ where $\rho(t)$ is a geodesic between the measures $M(c_0)$ and $M(c_1)$ for the metric $\overline{W}_2$. 
\begin{remark}
The above distance on $\tilde{C}_{conv}^1$ is in addition equivariant to the action of rotations of the plane in the sense that for any $R \in SO(2)$, we have $d(R\cdot c_0 , R \cdot c_1) = d(c_0,c_1)$. This follows from the equivariance of the constrained Wasserstein distance $\overline{W}_2$ with respect to rotations. Indeed, we see that for any probability measures $\mu_0$ and $\mu_1$ in $\mathcal{M}^+_0$ and any $R \in SO(2)$, we have $R_{\ast} M(c_0) = M(R \cdot c_0)$ and $R_{\ast} M(c_1) = M(R \cdot c_1)$. Furthermore, if $\nu=(\rho,m)$ is a pair of measures satisfying the continuity equation together with the boundary conditions $\rho_0 = \mu_0$, $\rho_1 = \mu_1$ and the closure condition $\int_{\Sp^1} e^{i\theta} d\rho_t(\theta)=0$ for almost all $t\in [0,1]$ then one easily verifies that $\tilde{\nu}=(\tilde{\rho},\tilde{m})$ defined by $\tilde{\rho} = (R_{\ast} \rho_t) \otimes dt$ and $\tilde{m} = (R_{\ast} m_t) \otimes dt$ still satisfy the continuity equation, the closure condition and that:
\begin{equation*}
 \int_{[0,1] \times \Sp^1} f\left(\frac{d\tilde{\nu}}{d\lambda}\right) d\lambda = \int_{[0,1] \times \Sp^1} f\left(\frac{d\nu}{d\lambda}\right) d\lambda.
\end{equation*}
Then taking the infimum leads to $\overline{W}_2(M(R \cdot c_0),M(R \cdot c_1)) = \overline{W}_2(M(c_0),M(c_1))$. This equivariance property allows to further quotient out rotations and obtain a distance between unit length convex curves modulo rigid motions that writes:
\begin{equation*}
 d'(c_0,c_1) \doteq \min_{R \in SO(2)} d(c_0,R\cdot c_1).
\end{equation*}
\end{remark}

\subsubsection{Numerical approach}
\label{ssec:constrained_Wasserstein_numerics}
We now derive a numerical method to estimate the constrained Wasserstein distance $\overline{W}_2$ between two discretized densities of $\Sp^1$ and thereby the distance $d$ between convex curves given by \eqref{eq:def_newdistance_convex}. Our proposed approach mainly follows and adapts the first-order proximal methods introduced in \cite{papadakis2014optimal} for the computation of Wasserstein metrics, specifically the primal-dual algorithm which itself is a special instantiation of \cite{chambolle2011first}. We briefly present the main elements of the algorithm in the following paragraphs.  

We shall consider measures on $[0,1] \times \Sp^1$ discretized over two types of regular time/angle grids: centered and a staggered grids. We define the centered grid $\mathcal{G}_c$ with $P+1$ time samples and $N$ angular samples as:
\begin{equation*}
 \mathcal{G}_c = \left\{\left(t_i = \frac{i}{P}, \theta_j = 2\pi\frac{j}{N} \right) \ | \ 0\leq i \leq P, \ 0\leq j \leq N-1 \right\}.
\end{equation*}
Discrete measure pairs sampled on the centered grid will be then represented as an element $V$ of the space $\mathcal{E}_c \doteq \mathbb{R}^{\mathcal{G}_c} \times \mathbb{R}^{\mathcal{G}_c}$. The staggered grids consist of time and angular samples shifted to the middle of the samples of $\mathcal{G}_c$. Specifically, we introduce a time and a space staggered grid defined as follows:
\begin{align*}
 &\mathcal{G}_s^t = \left\{\left(t_i = \frac{i+1/2}{P}, \theta_j = 2\pi\frac{j}{N} \right) \ | \ -1\leq i \leq P, \ 0\leq j \leq N-1 \right\} \\
 &\mathcal{G}_s^\theta = \left\{\left(t_i = \frac{i}{P}, \theta_j = 2\pi\frac{j+1/2}{N} \right) \ | \ 0\leq i \leq P, \ -1\leq j \leq N-2 \right\}.
\end{align*}
We shall then denote $U=(\rho,m)$ an element of $\mathcal{E}_s \doteq \mathbb{R}^{\mathcal{G}_s^t} \times \mathbb{R}^{\mathcal{G}_s^\theta}$ which represents a pair of discretized measures of $[0,1] \times \Sp^1$ on the staggered grids and we introduce the interpolation operator $\mathcal{I}: \mathcal{E}_s \rightarrow \mathcal{E}_c$ defined by $\mathcal{I}(U)=(\bar{\rho},\bar{m})$ with 
\begin{align*}
&\bar{\rho}_{i,j} = \frac{\rho_{i-1,j} + \rho_{i,j}}{2}, \ \text{for } 0\leq i \leq P, \ 0\leq j \leq N-1. \\
&\bar{m}_{i,j} = \left\{
    \begin{aligned}
     &\frac{m_{i,j-1} + m_{i,j}}{2}, \ \text{for } 0\leq i \leq P, \ 0\leq j \leq N-2 \\
     &\frac{m_{i,N-2} + m_{i,-1}}{2}, \ \text{for } 0\leq i \leq P, \ j=N-1
    \end{aligned}
\right.
\end{align*}
Note that the slight difference between the time and angular components and with the approach of \cite{papadakis2014optimal} is due to the periodic boundary condition on the samples $\theta_j \in \Sp^1$. We also define a discrete divergence operator $\text{div}: \mathcal{E}_s \rightarrow \mathbb{R}^{\mathcal{G}_c}$ as $\text{div}(U)_{i,j} = P(\rho_{i,j} -\rho_{i-1,j}) + N(m_{i,j} -m_{i,j-1})$ for all $0\leq i \leq P, \ 0\leq j \leq N-1$. This allows to write a discrete version of \eqref{eq:mod_Wasserstein_def} as a convex minimization problem:
\begin{equation}
\label{eq:constrained_Wasserstein_discrete}
 \min_{U=(\rho,m)\in \mathcal{E}_s} \left\{\sum_{i=0}^{P} \sum_{j=0}^{N-1} f\left(\mathcal{I}(U)_{i,j}\right) + \iota_{\mathcal{C}(\rho_0,\rho_1)}(U) \right\}.
\end{equation}
where 
\begin{equation*}
 \mathcal{C}(\rho_0,\rho_1)\doteq \left\{ U\in \mathcal{E}_s \ \text{s.t } 
  \left\{
    \begin{aligned}
         &\rho_{0,j} = (\rho_0)_j, \ \rho_{P,j} = (\rho_1)_j, \ \text{for all } -1\leq j \leq N-2\\
         &\text{div}(U)_{i,j} = 0, \ \text{for all } 0\leq i \leq P, \ 0\leq j \leq N-1 \\
         &\sum_{j=0}^{N-1} \cos(\theta_j) \rho_{i,j} =0, \ \sum_{j=0}^{N-1} \sin(\theta_j) \rho_{i,j} =0 \ \text{for all } -1\leq i \leq P.
    \end{aligned}
 \right.
 \right\}
\end{equation*}
and $\iota_{\mathcal{C}(\rho_0,\rho_1)}(U) = 0$ if $U \in \mathcal{C}(\rho_0,\rho_1)$, $\iota_{\mathcal{C}(\rho_0,\rho_1)}(U) = +\infty$ if $U \notin \mathcal{C}(\rho_0,\rho_1)$. 

To solve the convex problem \eqref{eq:constrained_Wasserstein_discrete}, we will use a primal-dual method adapted from the work of \cite{chambolle2011first}. We first recall the definition of the proximal operator of a convex function $h: \mathbb{R}^k \rightarrow \mathbb{R}$:
\begin{equation*}
 \text{Prox}_{h}(x) = \underset{y\in \mathbb{R}^k}{\text{argmin}} \ \frac{1}{2} \|x-y\|^2 + h(y)
\end{equation*}
as well as the convex dual of $h$:
\begin{equation*}
 h^*(w) = \underset{y\in \mathbb{R}^k}{\max} \ \langle y,w \rangle - h(y).
\end{equation*}
The function to minimize in \eqref{eq:constrained_Wasserstein_discrete} is of the form $F(\mathcal{I}(U)) + G(U)$. The primal-dual algorithm we implement consists of the following iterative updates on the sequences $(U^{(n)},V^{(n)},\Gamma^{(n)}) \in \mathcal{E}_s \times \mathcal{E}_c \times \mathcal{E}_s$ starting from an initialization $U^{(0)}\in \mathcal{E}_s$, $V^{(0)}\in \mathcal{E}_c$ and $\Gamma^{(0)}=U^{(0)}$:
\begin{equation}
\label{eq:primal_dual_algo}
    \begin{aligned}
         &V^{(n+1)} = \text{Prox}_{\sigma F^*}(V^{(n)}+\sigma \mathcal{I} \Gamma^{(n)}) \\
         &U^{(n+1)} = \text{Prox}_{\tau G}(U^{(n)}-\tau \mathcal{I}^* V^{(n+1)}) \\
         &\Gamma^{(n+1)} = U^{(n+1)} + \theta(U^{(n+1)}-U^{(n)})
    \end{aligned}
\end{equation}
where $\sigma$, $\tau$ are positive step sizes and $0\leq \theta \leq 1$ an inertia parameter. It follows from the general result of \cite{chambolle2011first} that the algorithm converges to a solution of \eqref{eq:constrained_Wasserstein_discrete} if $\tau \sigma \|\mathcal{I}\|^2 <1$ where $\|\mathcal{I}\|$ denotes the operator norm of $\mathcal{I}$. We typically have $\|\mathcal{I}\|\approx 1$ as $\mathcal{I}$ is here an interpolation operator. 

It only remains to specify the exact expressions of the proximal operators appearing in the first two equations of \eqref{eq:primal_dual_algo}. First, by Moreau's identity, one has that for all $V \in \mathcal{E}_c$:
\begin{equation*}
 \text{Prox}_{\sigma F^*}(V) = V-\sigma \text{Prox}_{F/\sigma}(V/\sigma).
\end{equation*}
Now the proximal of $F$ is computed in \cite{papadakis2014optimal} Proposition 1 from which we obtain for all $\tilde{V} \in \mathcal{E}_c$:
\begin{equation*}
 \text{Prox}_{F/\sigma}(\tilde{V}) = \left(\text{Prox}_{f/\sigma}(\tilde{V}_{i,j})\right)_{\substack{0\leq i \leq P\phantom{-1} \\ 0\leq j \leq N-1}}
\end{equation*}
where 
\begin{equation*}
 \text{Prox}_{f/\sigma}(d,m) = 
   \left\{
    \begin{aligned}
         &\left(\frac{\sigma d r^\star(d,m)}{\sigma d + 1}, r^\star(d,m)\right) \ \ &\text{if } r^\star(d,m)>0 \\
         &\left(0, 0\right) &\text{otherwise} 
    \end{aligned}
 \right.
\end{equation*}
with $r^\star(d,m)$ being the largest real root of the polynomial equation $(x-d)(x+1/\sigma)^2-\frac{1}{2\sigma}m^2 = 0$. 

To express the second update in \eqref{eq:primal_dual_algo}, we first point out that $\mathcal{I}^*: \mathcal{E}_c \rightarrow \mathcal{E}_s$ is given by $(\rho,m) = \mathcal{I}^*(\bar{\rho},\bar{m})$ with:
\begin{align*}
&\rho_{i,j} = \left\{
    \begin{aligned}
     &\frac{\bar{\rho}_{i+1,j}}{2}, \ \text{for } i=-1, \ 0\leq j \leq N-2 \\
     &\frac{\bar{\rho}_{i,j} + \bar{\rho}_{i+1,j}}{2}, \ \text{for } 0\leq i \leq P-1, \ 0\leq j \leq N-1 \\
     &\frac{\bar{\rho}_{i,j}}{2}, \ \text{for } i=P, \ 0\leq j \leq N-2
    \end{aligned}
    \right.
\\
&m_{i,j} = \left\{
    \begin{aligned}
     &\frac{\bar{m}_{i,0} + \bar{m}_{i,N-1}}{2}, \ \text{for } 0\leq i \leq P, \ j=-1 \\ 
     &\frac{\bar{m}_{i,j} + \bar{m}_{i,j+1}}{2}, \ \text{for } 0\leq i \leq P, \ 0\leq j \leq N-2 
    \end{aligned}
\right.
\end{align*}
Moreover since $G$ is here the convex indicator function of the constraint set $\mathcal{C}(\rho_0,\rho_1)$, we have $\tau G = G$ for any $\tau >0$ and $\text{Prox}_{\tau G} = \text{Prox}_{G} = \text{Proj}_{\mathcal{C}(\rho_0,\rho_1)}$ the orthogonal projector on the convex set $\mathcal{C}(\rho_0,\rho_1)$. In addition to the usual boundary and divergence constraints, we also have to deal with the two extra closure conditions. We introduce the operator $\Xi: \mathcal{E}_s \rightarrow \mathbb{R}^{\mathcal{G}_c} \times (\mathbb{R}^{N} \times \mathbb{R}^{N}) \times (\mathbb{R}^{P+2} \times \mathbb{R}^{P+2})$ defined for any $U=(\rho,m) \in \mathcal{E}_s$ by:
\begin{equation*}
 \Xi(U) = \left(\text{div}(U),(\rho_{-1,j},\rho_{P,j})_{0\leq j \leq N-1},\left(\sum_{j=0}^{N-1}\cos(\theta_j) \rho_{i,j},\sum_{j=0}^{N-1}\sin(\theta_j) \rho_{i,j}\right)_{-1\leq i \leq P} \right)
\end{equation*}
which allows us to write $\mathcal{C}(\rho_0,\rho_1) = \big\{U \in \mathcal{E}_s \ | \ \Xi(U) = b \doteq \left(0,(\rho_0,\rho_1),(0,0) \right) \big\}$. Thus, we can express the orthogonal projection on $\mathcal{C}(\rho_0,\rho_1)$ as:
\begin{equation*}
 \text{Proj}_{\mathcal{C}(\rho_0,\rho_1)}(U) = U - \Xi^*\Delta^{-1}\Xi(U) + \Xi^*\Delta^{-1} b
\end{equation*}
where $\Delta=\Xi\Xi^*$ can be interpreted as a modified Laplace operator on the spatial-angular domain, which in practice we can precompute together with its inverse before iterating \eqref{eq:primal_dual_algo} or alternatively implement in the Fourier domain. 

\subsubsection{Numerical results}
\begin{figure}
    \begin{tabular}{cccc}
    \includegraphics[trim = 25mm 25mm 25mm 25mm ,clip,width=4cm]{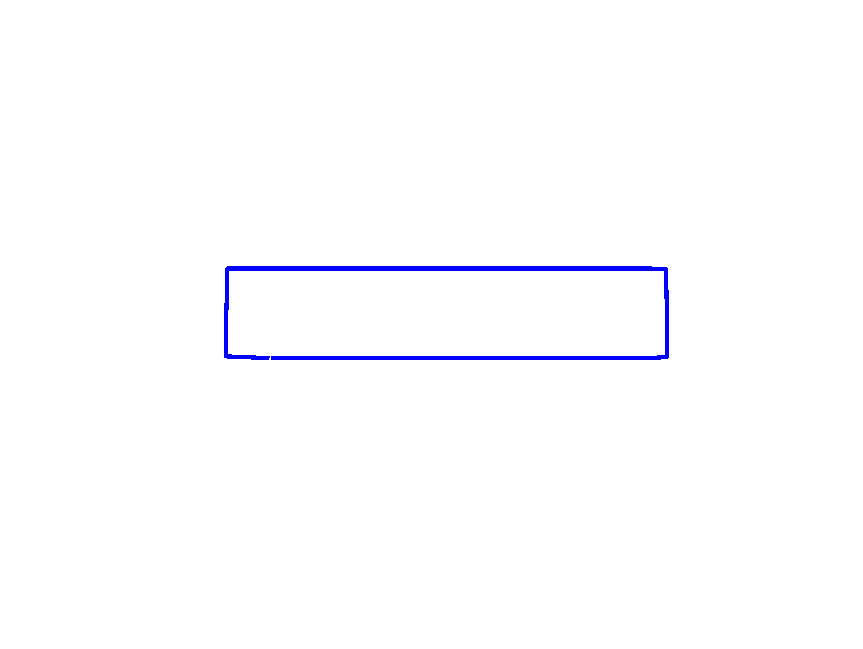} 
    &\includegraphics[trim = 25mm 25mm 25mm 25mm ,clip,width=4cm]{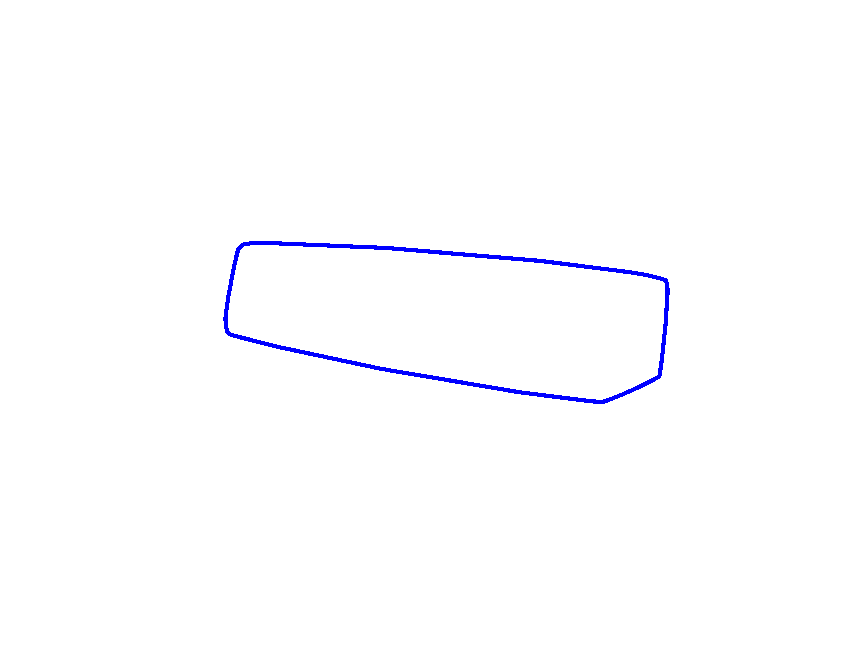}
    &\includegraphics[trim = 25mm 25mm 25mm 25mm ,clip,width=4cm]{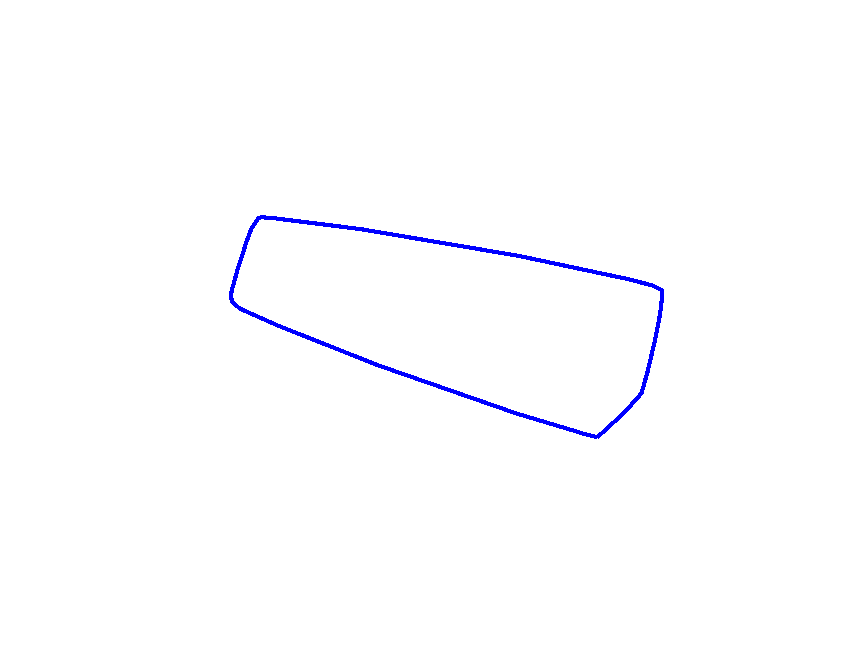} 
    &\includegraphics[trim = 25mm 25mm 25mm 25mm ,clip,width=4cm]{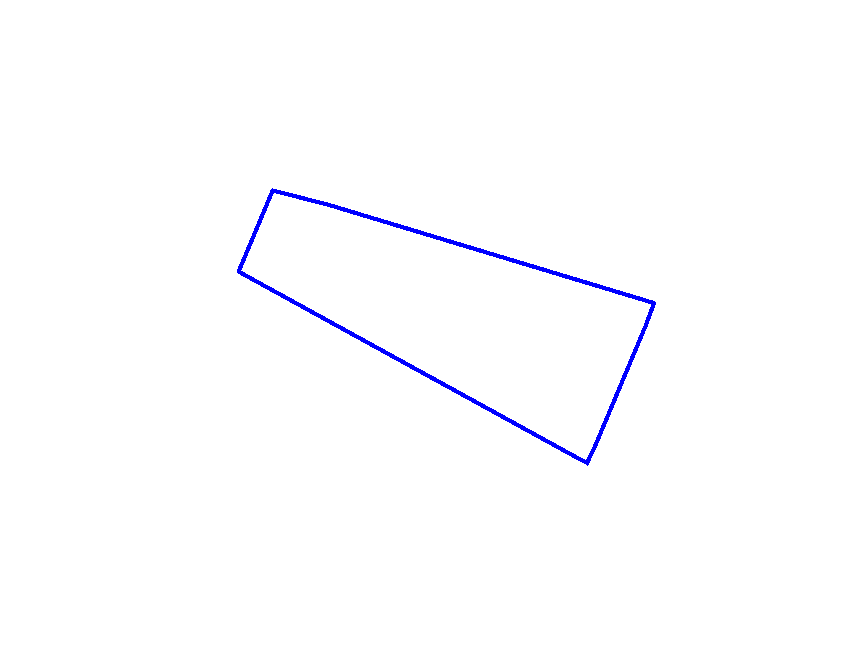}\\
    $t=0$ & $t=1/3$ & $t=2/3$ & $t=1$ \\
    \includegraphics[trim = 5mm 5mm 5mm 5mm ,clip,width=4cm]{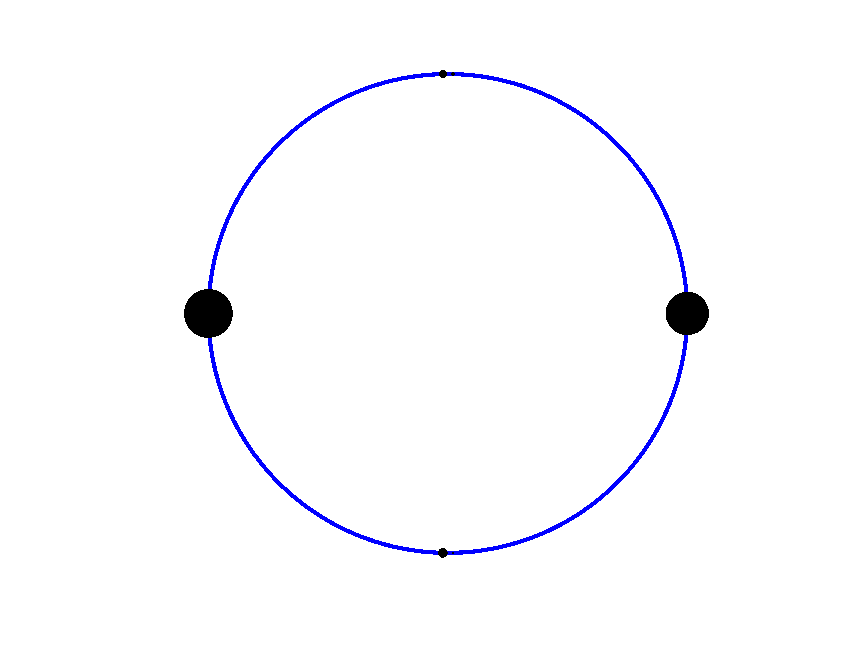} 
    &\includegraphics[trim = 5mm 5mm 5mm 5mm ,clip,width=4cm]{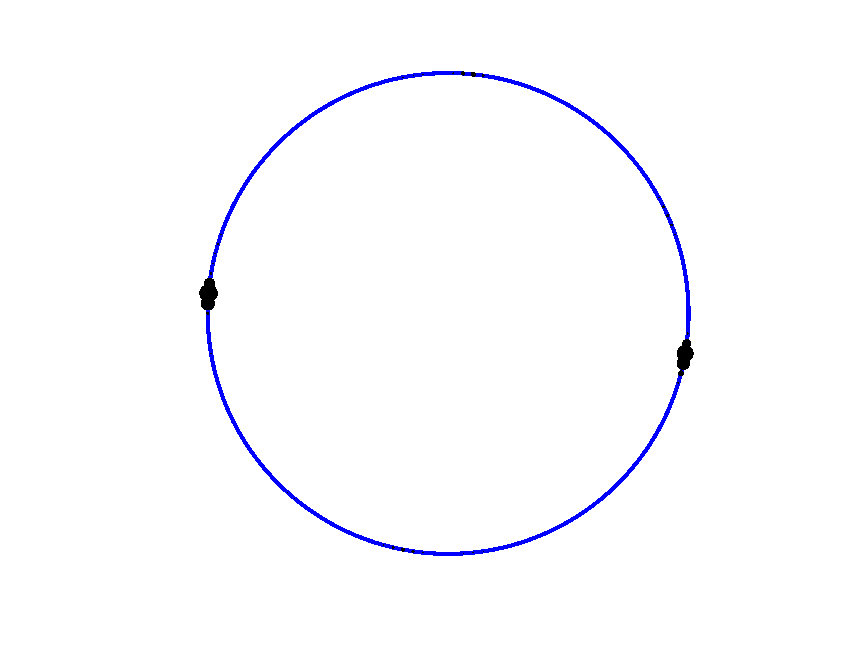}
    &\includegraphics[trim = 5mm 5mm 5mm 5mm ,clip,width=4cm]{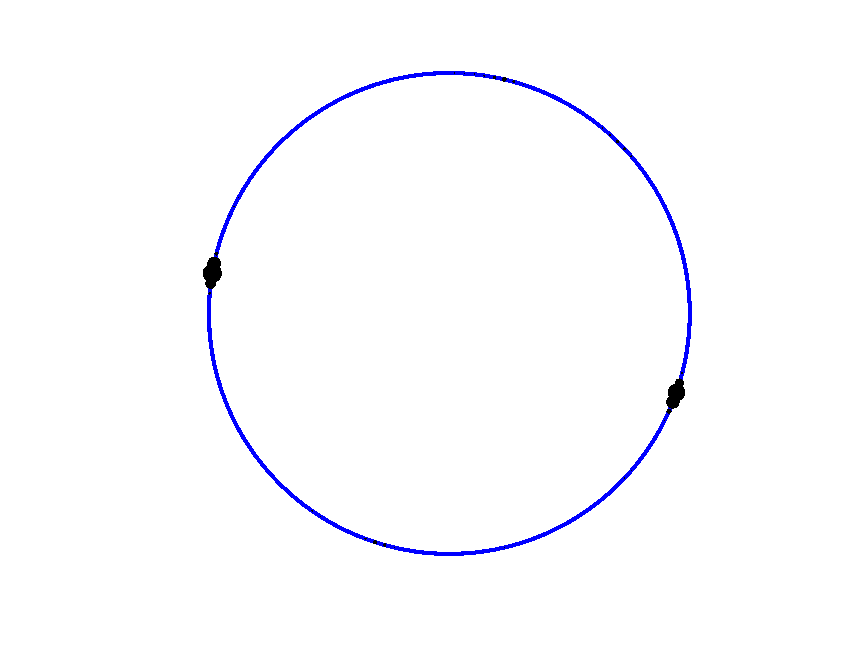} 
    &\includegraphics[trim = 5mm 5mm 5mm 5mm ,clip,width=4cm]{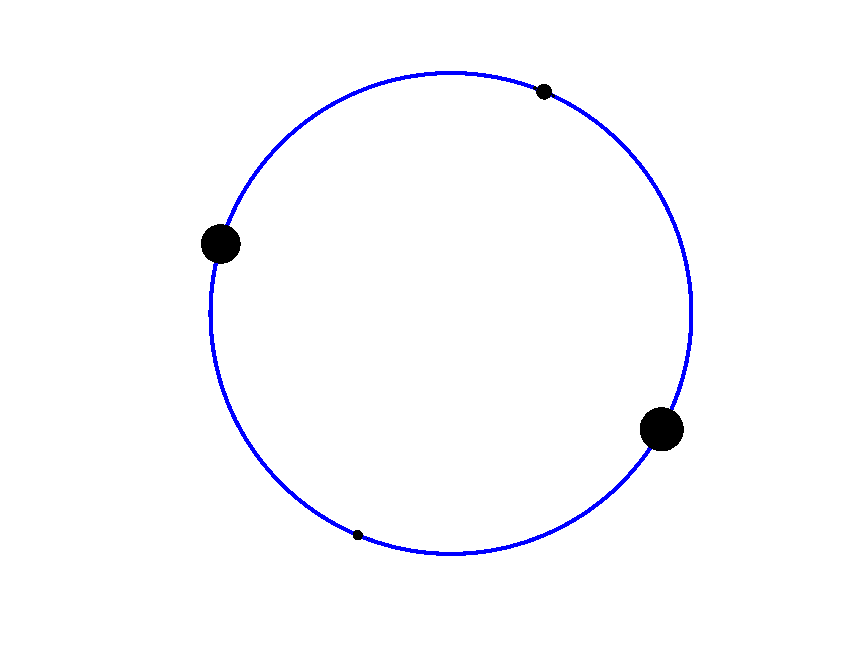}    
    \end{tabular}
    \caption{Geodesic for the constrained Wasserstein metric computed with the proposed primal-dual algorithm. On top, the intermediate convex curves $c(t) \in \tilde{C}_{conv}$ and in the bottom row are shown the associated measures $\mu(t)$.}
    \label{fig:geod_constrained_OT_metric1}
\end{figure}

We present a few numerical results obtained with the above approach which we implemented in Python. In all experiments, we take $N=150$ angular samples on $\Sp^1=[0,2\pi)$ and $P=50$ time samples in $[0,1]$. As for the step parameters of the primal-dual algorithm, we picked $\sigma = 0.5$, $\tau = 0.25$ and $\theta=0.6$. We run the algorithm \eqref{eq:primal_dual_algo} for 1000 iterations although we observe in practice that convergence to a minimum is typically reached after around 250 iterations (cf the energy plot in Figure \ref{fig:geod_constrained_OT_metric2}).  

First, in Figure \ref{fig:geod_constrained_OT_metric1}, is shown the geodesic between the same polygons as in Figures \ref{fig:geod_kernel_metric} and \ref{fig:geod_OT_metric} for the new constrained Wasserstein metric. Observe that unlike kernel metrics, the geodesic involves mass transportation rather than pure interpolation between the measures but that unlike the standard Wasserstein distance, the closure constraint is satisfied along the path and the corresponding shapes all stay within the set of closed convex curves. We further illustrate those effects with another example of a geodesic between two density measures and the corresponding geodesic between the convex curves in Figure \ref{fig:geod_constrained_OT_metric2}.  

\begin{figure}
    \includegraphics[width=5cm]{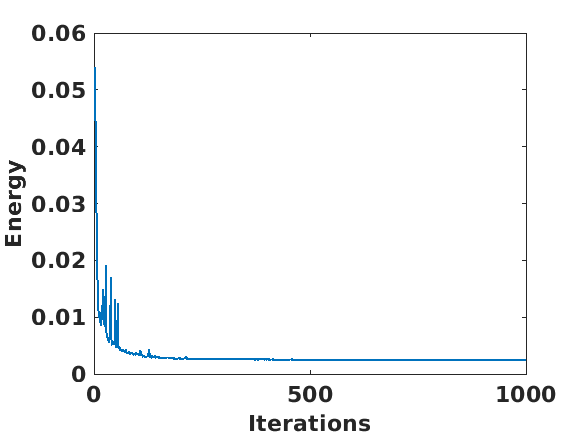} 
    \begin{tabular}{cccc}
    \includegraphics[trim = 25mm 10mm 25mm 10mm ,clip,width=4cm]{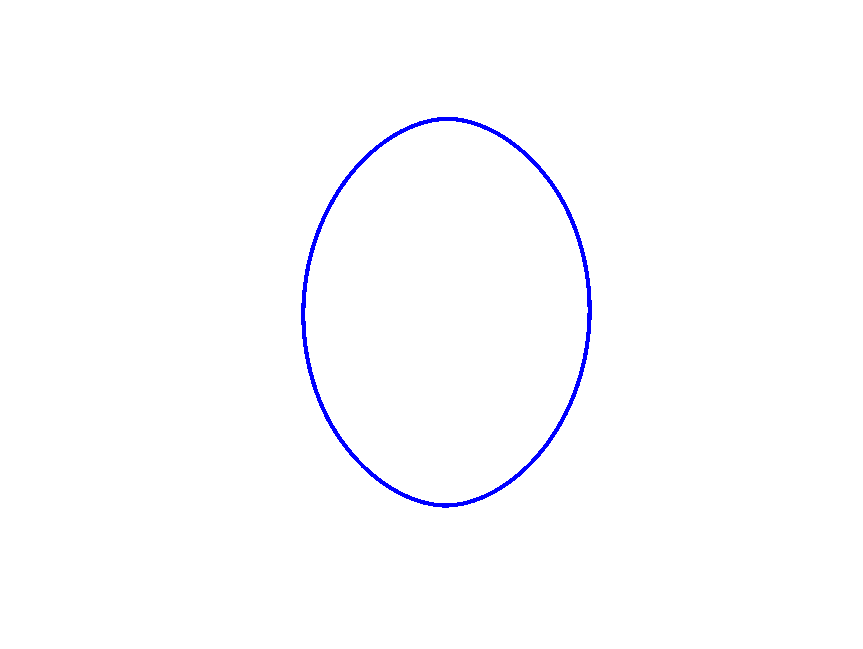} 
    &\includegraphics[trim = 25mm 10mm 25mm 10mm ,clip,width=4cm]{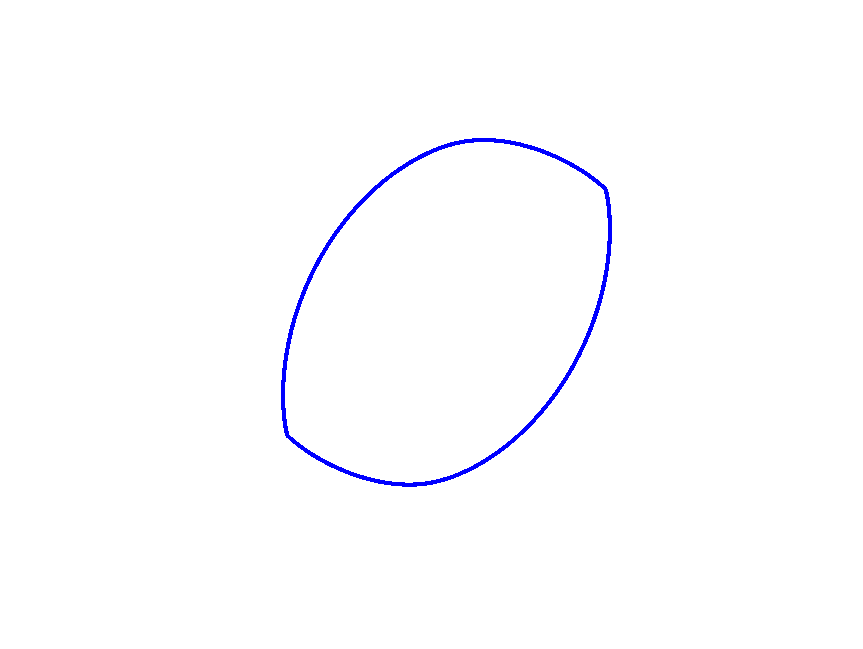}
    &\includegraphics[trim = 25mm 10mm 25mm 10mm ,clip,width=4cm]{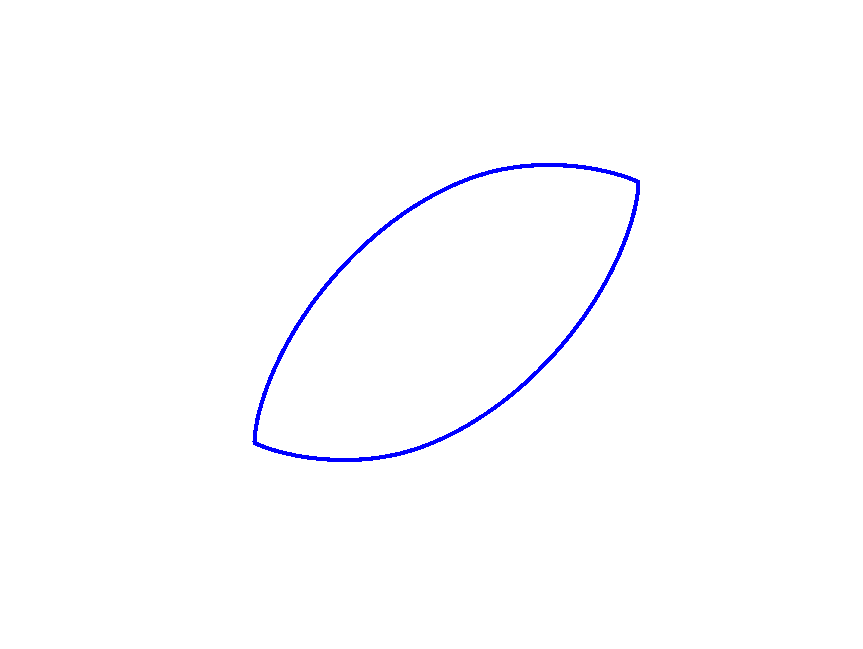} 
    &\includegraphics[trim = 25mm 10mm 25mm 10mm ,clip,width=4cm]{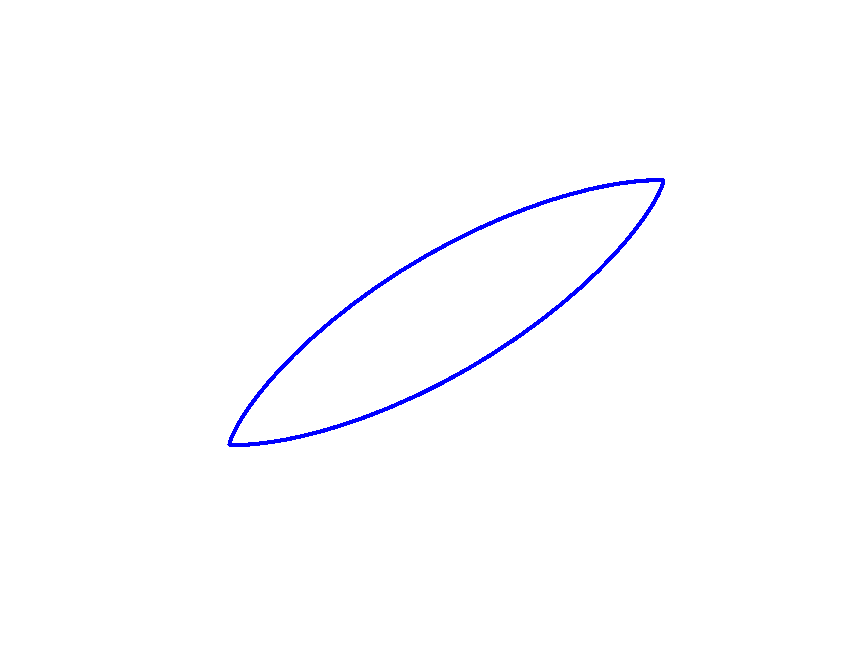}\\
    $t=0$ & $t=1/3$ & $t=2/3$ & $t=1$ \\
    \includegraphics[trim = 5mm 5mm 5mm 5mm ,clip,width=4cm]{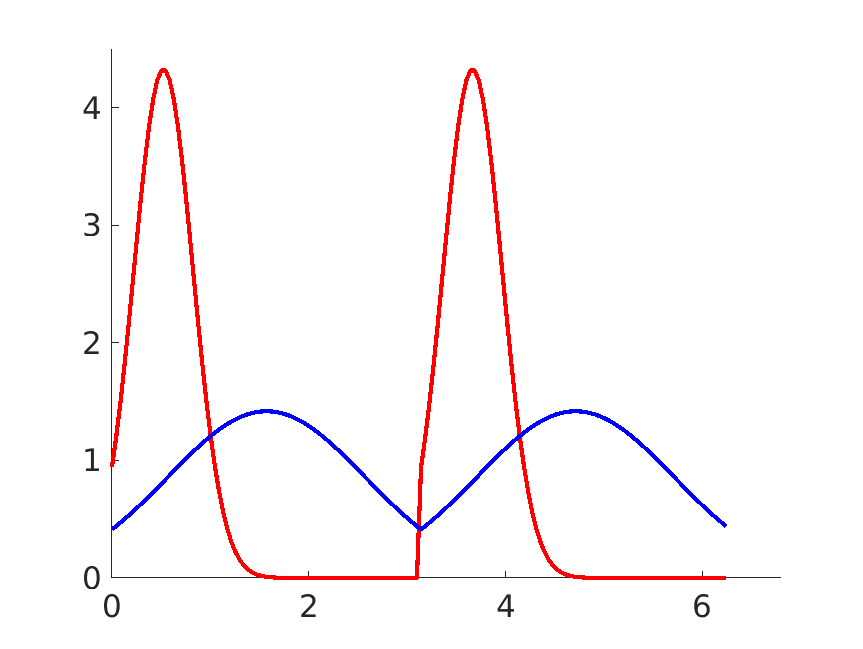} 
    &\includegraphics[trim = 5mm 5mm 5mm 5mm ,clip,width=4cm]{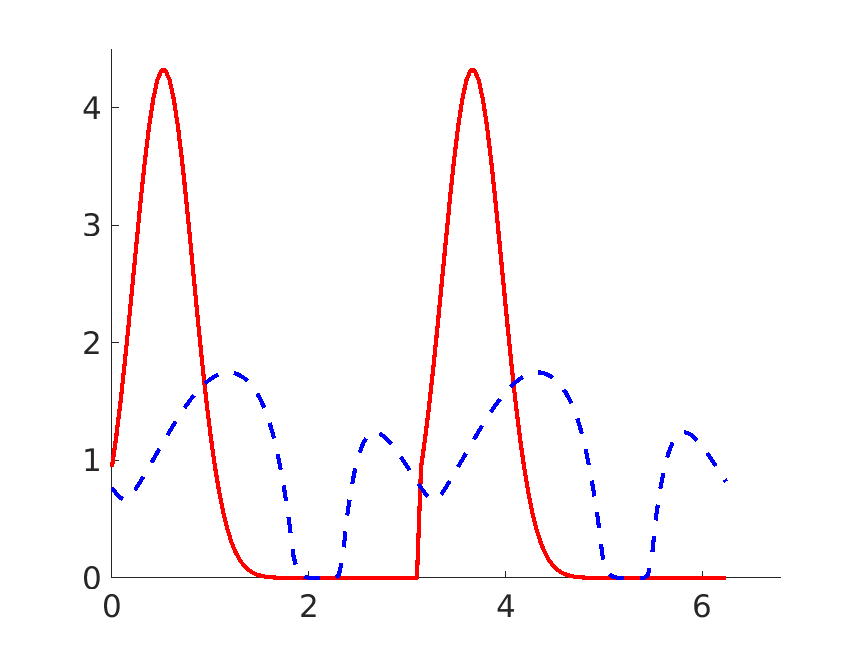}
    &\includegraphics[trim = 5mm 5mm 5mm 5mm ,clip,width=4cm]{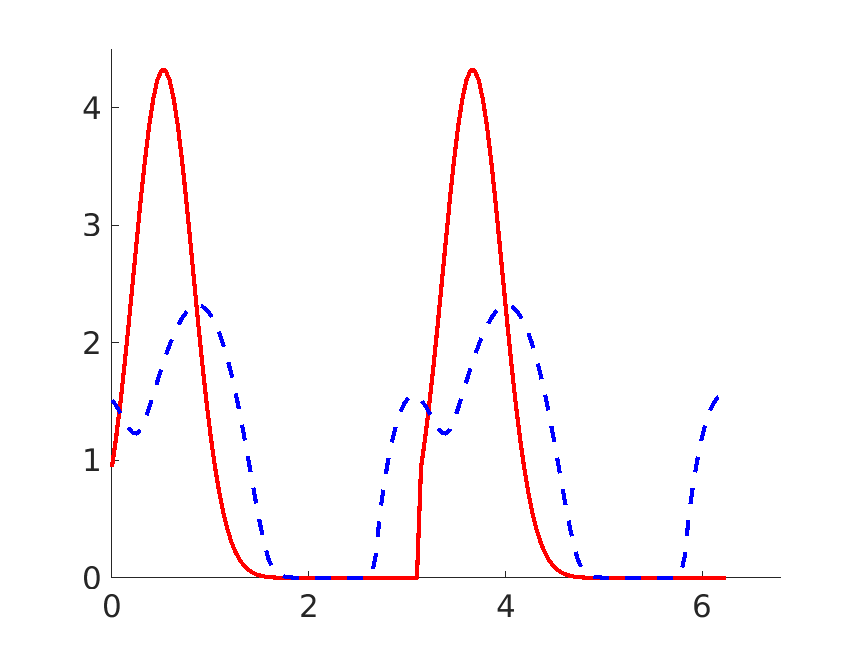} 
    &\includegraphics[trim = 5mm 5mm 5mm 5mm ,clip,width=4cm]{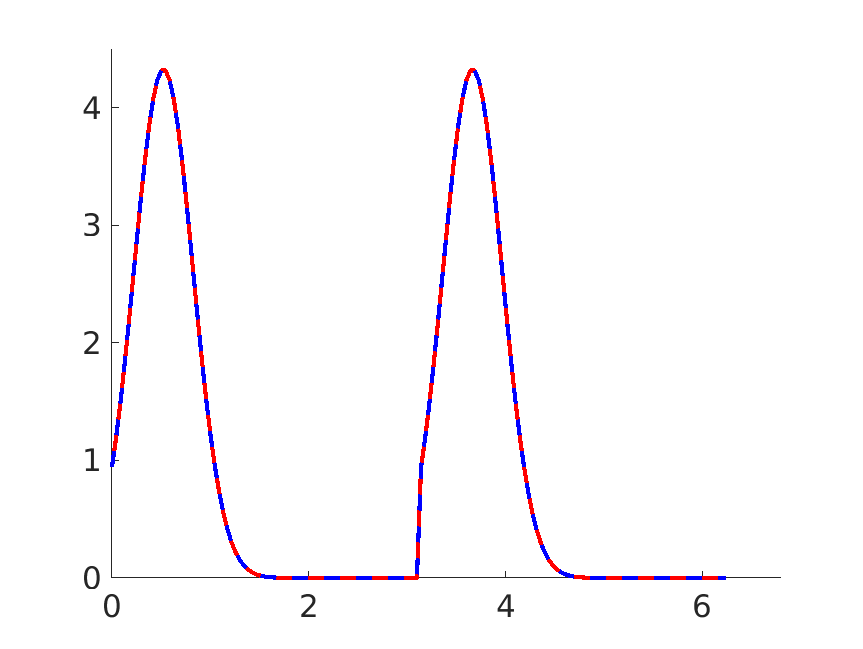}    
    \end{tabular}
    \caption{Geodesic for the constrained Wasserstein metric between two density measures. The evolution of the energy throughout the iterations of the proposed primal-dual algorithm is plotted on the top row to illustrate the convergence. On the middle row, are shown the intermediate convex curves $c(t) \in \tilde{C}_{conv}$ along the estimated geodesic. On the bottom row we display the final density $\rho_1$ in red and the density $\rho(t)$ along the geodesic in blue.}
    \label{fig:geod_constrained_OT_metric2}
\end{figure}

\section{Some open problems and future directions}
\label{sec:future_directions}
As a conclusion to this paper, we discuss a few possible topics and tracks for future investigation that arise as natural follow-ups to this work and the results we presented.

\subsection*{Unbalanced optimal transport for convex shape analysis}
In section \ref{ssec:constrained_Wasserstein}, we introduced a distance based on a constrained version of the Wasserstein metric between length measures which turns the set of convex closed curves of length one into a length space. There are yet several questions which are left partly unaddressed when it comes to this distance and its properties. In particular, it remains to be studied whether $\overline{W}_2$ is finite between any pair of probability measures of $\mathcal{M}^+_0$ or at least if the result of Proposition \ref{prop:finiteness_constrained_Wasserstein} could be extended to other families of measures. Besides, $\overline{W}_2$ is only defined for probability measures which implies that the resulting distance $d$ in \eqref{eq:def_newdistance_convex} requires the convex curves to be normalized to have unit length. This makes such a distance unadapted to retrieve and quantify e.g. global or local changes in scale. 

We believe that both of the previous issues could be addressed by considering an unbalanced extension of the metric $\overline{W}_2$ to all measures of $\mathcal{M}^+_0$, along similar lines as the unbalanced frameworks for standard optimal transport proposed in works such as \cite{Piccoli2014,liero2016optimal,Chizat2018}. In the context of this paper, this could be done by adding a source term to the continuity equation, namely consider $\partial_t \rho + \partial_\theta m = \zeta$ with $\zeta$ being a signed measure of $[0,1] \times \Sp^1$ that models an additional transformation of the measure $\rho$ in conjunction with mass transportation. Then, as in \cite{liero2016optimal,Chizat2018}, one could penalize this extra source term based on the Fischer-Rao metric of $\zeta$, which would lead to define the following distance between any $\mu_0,\mu_1 \in \mathcal{M}^+_0$:
\begin{equation*}
   \overline{W}_{FR}(\mu_0,\mu_1) = \inf_{\nu = (\rho,m,\zeta)} \left\{ \int_{[0,1] \times \Sp^1} g\left(\frac{d\nu}{d\lambda}\right) \ d\lambda \ | \ \text{subj. to } \left\{
    \begin{array}{lll}
         \partial_t\rho + \partial_\theta m= \zeta \\
          \rho_0=\mu_0, \ \rho_1=\mu_1 \\
          \int_{\Sp^1} e^{i\theta} d\rho_t(\theta)=0, \ \text{for a.e } t\in[0,1]
    \end{array}
\right.\right\}.
\end{equation*}
where $\lambda$ is such that $\rho,m,\zeta \ll \lambda$ and $g:\mathbb{R} \times \mathbb{R} \times \mathbb{R} \rightarrow \mathbb{R}_{+}$ is given by:
\begin{equation*}
 g(d,m,\zeta) = 
 \left\lbrace\begin{aligned}
  &\frac{|m|^2+\delta^2 \zeta^2}{2d} \ \ \text{if } d>0 \\
  &0  \ \ \text{if } (d,m,\zeta)=(0,0,0) \\
  &+\infty \ \ \text{otherwise}
  \end{aligned}\right.
\end{equation*}
with $\delta>0$ a weighting parameter between the two components of the cost. In this case, it becomes easy to show that $\overline{W}_{FR}(\mu_0,\mu_1)$ is finite for any measures $\mu_0,\mu_1 \in \mathcal{M}^+_0$ by simply considering the special path $\rho_t = (1-t)\rho_0 + t \rho_1$, $m_t = 0$ and $\zeta_t = \rho_1 - \rho_0$ for all $t \in [0,1]$. Although we leave it for future work, we expect the existence of geodesics and distance properties to hold for $\overline{W}_{FR}$ by extending the proof of Theorem \ref{thm:existence_geodesics_constr_Wass}. Moreover, the precise analysis of its topological properties and how the resulting metric between convex curves compare to other geometric distances remain to be examined.

\subsection*{Generalization to higher dimension}
Although this work focused on the case of planar curves, the concept of length measure extends to closed hypersurfaces in any dimension which is known as area measures, as mentioned in the introduction. Area measures, particularly of surfaces in $\R^3$, is a central concept in the Brunn-Minkowski theory of mixed volumes \cite{schneider_2013} but also appear in some applications such as object recognition \cite{hebert1995spherical} or surface reconstruction from computerized tomography \cite{prince1990reconstructing}. The area measure of a $(n-1)$-dimensional closed oriented submanifold (more generally rectifiable subset) $S$ of $\R^n$ can be defined, similarly to Definition \ref{def:length_measure}, as the positive Radon measure on $\Sp^{n-1}$ given by $\sigma_S(B) = \text{Vol}^{n-1}(\{x \in S \ | \ \vec{n}_S(x) \in B\})$ for all Borel subset $B \subset \Sp^{n-1}$ where $\vec{n}_S(x) \in \Sp^{n-1}$ denotes the unit normal vector to $S$ at $x$ and $\text{Vol}^{n-1}$ is the $(n-1)$ volume measure i.e. the Hausdorff measure of dimension $(n-1)$. Importantly, the Minkowski-Fenchel-Jessen theorem still holds for general area measures, namely the area measure again characterizes a convex set up to translation. However, there are two significant differences when $n\geq 3$ compared to the situation of planar curves. First the Minkowski sum of two convex does not generally correspond to the sum of their area measures. Second, reconstructing the convex shape from its area measure, even for discrete measures and polyhedra, is no longer straightforward: it is in fact an active research topic and several different approaches and algorithms have been proposed, see e.g. \cite{zouaki2003representation,lachand2005minimizing,Gardner2006}.   

In connection to the results presented in this paper, we expect for instance that an isoperimetric characterization of convex sets similar to Theorem \ref{thm:max_area_convex} should hold for the positive volume enclosed by $(n-1)$-dimensional non self-intersecting rectifiable sets that share the same area measure under the adequate regularity assumptions. Such a property has been shown in particular for polyhedra in \cite{Boroczky1986}. As for the construction of metrics and geodesics between convex shapes, the mathematical construction of the constrained Wasserstein metric presented in Section \ref{ssec:constrained_Wasserstein_theory} can be a priori adapted to area measures in any dimension, by replacing the closure constraint to $\int_{\Sp^{n-1}} x d\rho_t(x) = 0$ for almost all $t\in [0,1]$. However, the numerical implementation of such metrics along the lines of the approach of Section \ref{ssec:constrained_Wasserstein_numerics} would induce additional difficulties. Indeed, the computation of the operators on grids over higher-dimensional sphere becomes more involved and numerically intensive. Moreover, recovering the convex curves associated to a geodesic in the space of area measures would further require, as explained above, applying some reconstruction algorithm a posteriori. Finally, an important issue for future investigation is to derive an efficient implementation of the variation of the distance with respect to the discrete distributions, which could be then used for instance to estimate K\"{a}rcher means in the space of convex sets.

\section*{Acknowledgements}
The authors would like to thank F-X Vialard for some useful insights in relation to the optimal transport model presented in this paper. This work was supported by the National Science Foundation (NSF) under the grant 1945224.

\bibliographystyle{amsplain}
\bibliography{biblio}

\end{document}